\documentclass[10pt]{amsart}

\usepackage{xypic,amscd,amssymb,latexsym}
\usepackage{amsmath}	
\usepackage[breaklinks=true]{hyperref}

\usepackage{mathrsfs}

\usepackage{url}

\usepackage{graphicx}

\usepackage{cancel}

\usepackage{caption}
\usepackage{subcaption}

\begin{document}
\pagenumbering{arabic}

\newtheorem{theorem}{Theorem}[section]
\newtheorem{lemma}[theorem]{Lemma}
\newtheorem{proposition}[theorem]{Proposition}
\newtheorem{corollary}[theorem]{Corollary}
\newtheorem{definition}[theorem]{Definition}
\newtheorem{remark}[theorem]{Remark}
\newtheorem{notation}[theorem]{Notation}

\newcommand{\vs}[0]{\vspace{2mm}}

\newcommand{\mcal}[1]{\mathcal{#1}}
\newcommand{\ul}[1]{\underline{#1}}
\newcommand{\ulu}[2]{\underset{#2}{\underline{#1}}}
\newcommand{\ol}[1]{\overline{#1}}
\newcommand{\til}[1]{\widetilde{#1}}
\newcommand{\wh}[1]{\widehat{#1}}

\address{School of Mathematics, Korea Institute for Advanced Study (KIAS), 85 Hoegiro Dongdaemun-gu, Seoul 130-722, Republic of Korea}

\email[H.~Kim]{hkim@kias.re.kr, hyunkyu87@gmail.com}

\author{Hyun Kyu Kim}

\numberwithin{equation}{section}

\title[Ratio coordinates for higher Teichm\"uller spaces]{Ratio coordinates for higher Teichm\"uller spaces}

\begin{abstract}
We define new coordinates for Fock-Goncharov's higher Teichm\"uller spaces for a surface with holes, which are the moduli spaces of representations of the fundamental group into a reductive Lie group $G$. Some additional data on the boundary leads to two closely related moduli spaces, the $\mathscr{X}$-space and the $\mathscr{A}$-space, forming a cluster ensemble. Fock and Goncharov gave nice descriptions of the coordinates of these spaces in the cases of $G = {\rm PGL}_m$ and $G={\rm SL}_m$, together with Poisson structures. We consider new coordinates for higher Teichm\"uller spaces given as ratios of the coordinates of the $\mathscr{A}$-space for $G={\rm SL}_m$, which are generalizations of Kashaev's ratio coordinates in the case $m=2$. Using Kashaev's quantization for $m=2$, we suggest a quantization of the system of these new ratio coordinates, which may lead to a new family of projective representations of mapping class groups. These ratio coordinates depend on the choice of an ideal triangulation decorated with a distinguished corner at each triangle, and the key point of the quantization is to guarantee certain consistency under a change of such choices. We prove this consistency for $m=3$, and for completeness we also give a full proof of the presentation of Kashaev's groupoid of decorated ideal triangulations.
\end{abstract}

\maketitle

\tableofcontents

\section{Introduction, background, overview, and motivation}
\label{sec:introduction}

\subsection{Summary}

Fock and Goncharov \cite{FG03} defined the notion of a {\em cluster ensemble} as a pair of `positive spaces' $(\mathscr{X}, \mathscr{A})$, which is a generalization of  cluster algebras of Fomin and Zelevinsky \cite{FZ}. Here, {\em positive spaces} mean varieties equipped with an atlas whose transition maps  are {\em positive} rational maps, that is, maps in the form $f_1/f_2$ where $f_1,f_2$ are $\mathbb{Z}_{>0}$-linear combinations of monomials in the corresponding coordinates. A cluster ensemble relies on a similar combinatorial data as a cluster algebra does, and the positive transition maps should be given by certain formulas, depending on this  combinatorial data.
Main examples of cluster ensembles come from what are called the {\em higher Teichm\"uller spaces}, the basic case being the usual Teichm\"uller spaces of Riemann surfaces with boundaries and/or punctures. The higher Teichm\"uller spaces for compact Riemann surfaces are first discovered by N. J. Hitchin \cite{H92} (as pointed out to the author by a referee), and their cluster structures in the cases of bordered Riemann surfaces are extensively studied by Fock and Goncharov \cite{FG06}.

\vs

Let $S$ be a Riemann surface. Roughly speaking, the Teichm\"uller space of $S$ is a certain component of the representation variety ${\rm Hom}(\pi_1(S), {\rm PSL}_2(\mathbb{R})) / {\rm PSL}_2(\mathbb{R})$, while the higher Teichm\"uller space is defined similarly where ${\rm PSL}_2(\mathbb{R})$ is replaced by ${\rm PSL}_m(\mathbb{R})$, or more generally by any reductive real Lie group \cite{FG06}. The Teichm\"uller space, or more precisely, a certain fiber bundle over it, can be parametrized by Penner's lambda length coordinates \cite{Penner}, which depend on the choice of an ideal triangulation of $S$. From lambda length coordinates one obtains two other types of coordinate systems, namely Thurston's shear coordinates and Kashaev's ratio coordinates. Using the former, Chekhov-Fock \cite{Fo} \cite{FC} constructed a quantization of the Teichm\"uller space, while Kashaev \cite{Kash98} came up with another quantization using the latter. Then Fock-Goncharov \cite{FG09} established a more general construction, in particular giving a quantization of the higher Teichm\"uller spaces for all $m\ge 2$, with respect to a higher generalized version of shear coordinates. In the present paper, we constuct an analog of Kashaev's ratio coordinates for the higher Teichm\"uller spaces, and suggest a corresponding quantization. On one hand, this quantization is most likely not `equivalent' to the one of Fock-Goncharov, in view of \cite{Ki12}. On the other hand, there is a chance that this construction is realized in a certain tensor category of representations of a Hopf algebra, as an analog of the work of Frenkel-Kim \cite{FrKi}, where quantum Teichm\"uller theory is recovered by the representation theory of the `modular double of the Borel subalgebra of the quantum group $\mcal{U}_q(\frak{sl}(2,\mathbb{R}))$'. We expect that if we replace $\mcal{U}_q(\frak{sl}(2,\mathbb{R}))$ by $\mcal{U}_q(\frak{sl}(m,\mathbb{R}))$, we might recover the quantum higher Teichm\"uller theory constructed in the present paper by the representation theory of this Hopf algebra \cite{Ip14}.

\subsection{Quantization of Teichm\"uller spaces}
\label{subsec:quantization_of_Teichmuller_spaces}

Let $S$ be a compact oriented real $2$-dimensional topological surface of genus $g$ with $n\ge 0$ holes; one can think of $S$ as being obtained from a closed surface by removing $n$ non-intersecting open discs. The {\em classical Teichm\"uller space $\mathscr{T}_S$} of $S$ is defined as the space of all complex structures on $S$ modulo diffeomorphisms $S\to S$ isotopic (rel. boundary) to identity. It is equipped with the canonical {\em Weil-Petersson Poisson structure}, which is preserved under the action on $\mathscr{T}_S$ of the {\em mapping class group} $\Gamma_S$ of $S$, defined as the group of all orientation-preserving diffeomorphisms $S \to S$ modulo isotopy (rel. boundary). Therefore one can consider `quantizing' the Poisson manifold $\mathscr{T}_S$, as well as quantizing the $\Gamma_S$-action on it. This roughly means that we deform the commutative algebra of smooth functions on $\mathscr{T}_S$ `in the direction of the Poisson structure', by a one-parameter family of non-commutative algebras $\mathscr{T}^h_S$ with real parameter $h$, while to each element of $\Gamma_S$ we assign an algebra automorphism of $\mathscr{T}_S^h$ that recovers the classical $\Gamma_S$-action as $h \to 0$ and such that this assignment $\Gamma_S \to {\rm Aut}(\mathscr{T}_S^h)$ is a group homomorphism. Quantization of some different versions of the Teichm\"uller space was acquired by Kashaev \cite{Kash98} and by Chekhov-Fock \cite{Fo} \cite{FC} independently, with the help of a special function named {\em quantum dilogarithm} \cite{FK} \cite{F}. As usual in the quantum theories, the algebra $\mathscr{T}_S^h$ is realized as $*$-algebra of operators on a Hilbert space $\mathscr{H}$, and the automorphism of $\mathscr{T}_S^h$ associated to $g\in \Gamma_S$ as conjugation by some unitary operator $\rho_h(g)$ on $\mathscr{H}$. So one gets an assignment $g\mapsto \rho_h(g)$, which is a projective representation of $\Gamma_S$ on $\mathscr{H}$, meaning that for each $g_1,g_2\in \Gamma_S$ one has $\rho_h(g_1 g_2) = c_{g_1,g_2} \rho_h(g_1) \rho_h(g_2)$ for some constant $c_{g_1,g_2} \in {\rm U}(1)$. This family of projective representations of $\Gamma_S$ on $\mathscr{H}$ is often considered as one of the main results of {\em quantum Teichm\"uller theory}.

\vs

Kashaev and Chekhov-Fock used slightly different coordinate systems of the Teichm\"uller space, both based on Penner's lambda length coordinates \cite{Penner}. In fact, these coordinate systems parametrize different spaces, which we shall shortly see. First, by the Uniformization Theorem, one can identify $\mathscr{T}_S$ with the space of all faithful representations $\rho : \pi_1(S) \to {\rm PSL}_2(\mathbb{R})$ with discrete image, modulo conjugation by ${\rm PSL}_2(\mathbb{R})$. Each point $\rho$ of $\mathscr{T}_S$ equips $S$ with a hyperbolic metric, by realizing $S$ as the quotient of the upper half-plane by the action of the discrete group $\rho(\pi_1(S))$. For any chosen $\rho$, a boundary component $C$ of $S$ is called {\em cuspidal}, if a loop in $S$ freely homotopic to $C$ is sent via $\rho$ to a parabolic element of ${\rm PSL}_2(\mathbb{R})$. The {\em holed Teichm\"uller space $\mathscr{T}_S^+$} is defined as the branched covering of $\mathscr{T}_S$, whose fiber at a point of $\mathscr{T}_S$ consists of all possible choices of orientations on the non-cuspidal boundary components of $S$. The subspace $\mathscr{T}_S^u$ of $\mathscr{T}_S$ consisting of points such that all $n$ holes are cuspidal can be viewed as the {\em Teichm\"uller space of a surface $S$ with $n$ punctures}, where punctures are thought to be at geodesic infinity. The total space of the fiber bundle $\mathscr{T}_S^d \to \mathscr{T}_S^u$, whose fiber at each point is the space of all possible choices of horocycles at the $n$ punctures and therefore isomorphic to $\mathbb{R}_{>0}^n$, is called the {\em decorated Teichm\"uller space}. From now on, we assume $n\ge 1$.

\vs

For the lambda length coordinate system which parametrizes $\mathscr{T}_S^d$, we may assume that all $n$ holes are punctures, and choose an {\em ideal triangulation $\tau$ of $S$}, that is, a collection of mutually non-intersecting unoriented paths between punctures up to homotopy (rel. punctures), such that the complementary region in $S$ is a disjoint union of triangles. Per each point of $\mathscr{T}_S^d$, stretch the edges of $\tau$ to geodesics with respect to the relevant hyperbolic metric on $S$. Then we are given a choice of a horocycle at each puncture; in the universal cover $\mathbb{H}$, the upper half-plane, a horocycle based at a point on $\mathbb{R}$ can be thought of as an Euclidean circle in $\mathbb{H}$ tangent to $\mathbb{R}$ at the relevant point, and a horocycle based at the point $\infty$ can be thought of as a Euclidean straight line in $\mathbb{H}$ parallel to $\mathbb{R}$. To each edge $e$ of $\tau$ we assign the {\em lambda length} $\lambda_e = \exp(\delta_e/2)$, where $\delta_e$ is the hyperbolic length of the edge truncated at both ends by the horocycles (so $\delta_e$ could be negative too). Then there is no restriction on or algebraic relations among the $6g-6+3n$ lambda length coordinate functions, and thus we identify $\mathscr{T}_S^d$ with $\mathbb{R}_{>0}^{6g-6+3n}$ real analytically (\cite{Penner}):
\begin{align*}
{\renewcommand{\arraystretch}{1.2}
\begin{array}{rcl}
\mathscr{T}_S^d & \stackrel{\sim}{\longrightarrow} & \mathbb{R}_{>0}^{6g-6+3n} \\
{\rm point} & \longmapsto & \left\{ \lambda_e \right\}_{e \, : \, {\rm edges~of~}\tau}.
\end{array}}
\end{align*}
The pullback of the Weil-Petersson $2$-form on $\mathscr{T}_S^u$ to $\mathscr{T}_S^d$ can be written in terms of lambda lengths as
\begin{align}
\label{eq:2-form}
\omega = - 2 \sum_t \left( d\log \lambda_{e_{t,1}} \wedge d\log \lambda_{e_{t,2}}
+ d\log \lambda_{e_{t,2}} \wedge d\log \lambda_{e_{t,3}}
+ d\log \lambda_{e_{t,3}} \wedge d\log \lambda_{e_{t,1}} \right),
\end{align}
where $t$ runs through all ideal triangles of the chosen ideal triangulation $\tau$, where $e_{t,1}, e_{t,2}, e_{t,3}$ are the edges of the triangle $t$ labeled in a counter-clockwise manner. Meanwhile, to construct a quantization which does not depend on the choice of $\tau$, we investigate what happens if we choose another ideal triangulation. There are simple ways of changing an ideal triangulation to another, called {\em flips}, described as follows. For an edge $e$ of an ideal triangulation $\tau$, consider the ideal quadrilateral formed by the two ideal triangles having $e$ as one of their sides. Now replace $e$ by the other diagonal $e'$ of this ideal quadrilateral; the edges of $\tau$ other than $e$, together with $e'$, form a new ideal triangulation $\tau'$ of $S$, and we say that $\tau'$ is obtained from $\tau$ by applying a flip along $e$. Then, only one lambda length changes, namely $\lambda_e \leadsto \lambda_{e'}$, and there is a formula relating the new lambda length $\lambda_{e'}$ to the old lambda lengths; if $a,b,c,d$ are the edges of the above mentioned ideal quadrilateral, labeled cyclically in a counter-clockwise way, then the {\em Ptolemy relation} holds:
\begin{align}
\label{eq:Ptolemy_relation}
\lambda_e \lambda_{e'} = \lambda_a \lambda_c + \lambda_b \lambda_d.
\end{align}
Therefore the new coordinate $\lambda_{e'}$ can be written as a rational function in the old coordinates, and one can prove that the $2$-form \eqref{eq:2-form} is invariant under this change-of-coordinate transformation; namely, after a flip, the pullback of the $2$-form $\omega$ for $\tau$ is coincides with the expression like in RHS of \eqref{eq:2-form} for the new triangulation $\tau'$.

\vs

Meanwhile, for an edge $e$ of an ideal triangulation $\tau$, consider the unique ideal quadrilateral of $\tau$ having $e$ as a diagonal, and label the four sides of it by $a,b,c,d$ counter-clockwise, so that $a,b,e$ form an ideal triangle of $\tau$ and $c,d,e$ another. Define {\em Thurston's shear coordinate $Z_e$} by $\log \frac{\lambda_b \lambda_d}{\lambda_a \lambda_c}$, that is, as the logarithm of a cross-ratio of four lambda lengths. Then $Z_e$ measures the signed hyperbolic distance of the two points on $e$ which are the feet of the perpendiculars dropped from the two vertices of the ideal quadrilateral that are not the endpoints of $e$. This time, $Z_e$'s do {\em not} depend on the choice of horocycles, and satisfy a certain linear equation per each puncture: the sum of shears along the edges traverse to a loop in $S$ freely homotopic to a puncture is zero. Modulo these equations, the shear coordinates parametrize the genuine Teichm\"uller space $\mathscr{T}_S^u$ of the $n$-punctured surface $S$. In fact, if we forget these $n$ equations, then the shear coordinates parametrize $\mathscr{T}_S^+$, where the absolute value of the sum of shears around a hole gives a geodesic length of the hole, while its sign gives the orientation of the hole; thus we identify $\mathscr{T}_S^+$ with $\mathbb{R}^{6g-6+3n}$:
\begin{align*}
{\renewcommand{\arraystretch}{1.2}
\begin{array}{rcl}
\mathscr{T}_S^+ & \stackrel{\sim}{\longrightarrow} & \mathbb{R}^{6g-6+3n} \\
{\rm point} & \longmapsto & \left\{ Z_e \right\}_{e \, : \, {\rm edges~of~}\tau},
\end{array}}
\end{align*}
credited as Thurson-Fock's Theorem in \cite{PennerBook}.
The Poisson structure can be written in terms of the Poisson brackets  $\{ Z_{e_1}, Z_{e_2} \} = \varepsilon_{e_1,e_2} = a_{e_1,e_2} - a_{e_2,e_1} \in \{-2,-1,0,1,2\}$, where $a_{e,f}$ is the number of ideal triangles in $\tau$ having $e,f$ as its sides in which $f$ is the clockwise immediate next side to $e$. In terms of the exponential coordinate $X_e = \exp(Z_e)$, it can be written as
\begin{align}
\label{eq:Poisson_bracket2}
\{ X_{e_1}, X_{e_2} \} = \varepsilon_{e_1,e_2} X_{e_1} X_{e_2}.
\end{align}
If we flip $\tau$ along an edge $k$ to get $\tau'$, then the new coordinates are given by the rational transformations of the original:
\begin{align}
\label{eq:X_mutation}
X'_{e'} = \left\{ \begin{array}{ll}
X_e ( 1 + X_k^{-1})^{-\varepsilon_{e,k}} & \mbox{if $\varepsilon_{e,k}>0$,} \\
X_e (1+X_k)^{-\varepsilon_{e,k}} & \mbox{if $\varepsilon_{e,k}<0$}, \\
X_e & \mbox{otherwise,} \\
\end{array} \right.
\end{align}
where $X'_{e'}$ is the new coodinate for the edge $e'$ of $\tau'$ corresponding to an edge $e$ of $\tau$. This Poisson structure is compatible with the $2$-form $\omega$ \eqref{eq:2-form} for the lambda lengths, and is invariant under the transformation \eqref{eq:X_mutation} in the sense that $\{ X'_{e_1'}, X'_{e_2'} \} = \varepsilon_{e_1',e_2'}' X'_{e_1'} X'_{e_2'}$, where $\varepsilon'$ is defined for $\tau'$ analogously to $\varepsilon$. Chekhov and Fock \cite{FC} constructed a non-commutative algebra topologically generated by $\wh{Z}_e$ with relations $[\wh{Z}_{e_1}, \wh{Z}_{e_2}] = i h \{Z_{e_1}, Z_{e_2}\}$, together with an expression $\rho_h(\mu_k)$ in terms of $\wh{Z}_e$'s such that the algebra isomorphism given as conjugation by $\rho_h(\mu_k)$ is the deformation of the mutation formula \eqref{eq:X_mutation} in an appropriate sense, such that $\rho_h$ sends composition of mutations to composition of isomorphisms. This result can be written in terms of the algebra generated by the exponentials $\wh{X}_e^{\pm 1} = e^{\pm \wh{Z}_e}$ with relations $\wh{X}_{e_1} \wh{X}_{e_2} = q^{2 \varepsilon_{e_1,e_2}} \wh{X}_{e_2} \wh{X}_{e_1}$, where $q = e^{i h/2}$. The elements $\wh{Z}_e$ and $\wh{X}_e$ are realized as self-adjoint operators on a Hilbert space, and $\rho_h(\mu_k)$ as a unitary operator involving the quantum dilogarithm function. This $\rho_h$ induces a family projective representations of the mapping class group $\Gamma_S$ on the Hilbert space, since elements of $\Gamma_S$ are realized as compositions of flips, hence as compositions of mutations.

\vs

Meanwhile, Kashaev \cite{Kash98} introduced a different coordinate system of (decorated) Teichm\"uller spaces using lambda lengths, so that the $2$-form $\omega$ \eqref{eq:2-form} becomes `diagonal'. For this coordinate, one needs to consider a decoration on a triangulation of $S$, namely the choice of a distinguished corner for each ideal triangle. We indicate this by a dot $\bullet$ in pictures, and call such a data a {\em dotted triangulation} of $S$. For each ideal triangle $t$, label the sides of it by $e_{t,1}$, $e_{t,2}$, $e_{t,3}$ counterclockwise, such that $e_{t,2}$ is the side opposite to the dotted corner of $t$. Define coordinates $q_t$ and $p_t$ as
\begin{align}
\label{eq:log_Kashaev_coordinates}
q_t = \log \lambda_{e_{t,1}}  - \log \lambda_{e_{t,2}} \quad\mbox{and}\quad
p_t = \log \lambda_{e_{t,3}} - \log \lambda_{e_{t,2}}.
\end{align}
Sometimes we use their exponentials
$$
Y_t = e^{q_t} = \lambda_{e_{t,1}}/\lambda_{e_{t,2}}\quad\mbox{and}\quad
Z_t = e^{p_t} = \lambda_{e_{t,3}}/\lambda_{e_{t,2}},
$$
which Kashaev calls the {\em ratio coordinates} for an apparent reason. The good thing about this coordinate system
\begin{align*}
{\renewcommand{\arraystretch}{1.2}
\begin{array}{rcl}
\mathscr{T}_S^d & \longrightarrow & \mathbb{R}^{8g-8+4n} \\
{\rm point} & \longmapsto & \left\{ (q_t,p_t) \right\}_{t \, : \, {\rm triangles~of~}\tau},
\end{array}}
\end{align*}
is that it behaves like a Darboux basis, that is, the $2$-form \eqref{eq:2-form} becomes
\begin{align}
\nonumber
\omega = - 2 \sum_t  dp_t \wedge dq_t,
\end{align}
and therefore admits a `canonical' quantization; the coordinate functions $q_t,p_t$ are replaced by the operators
\begin{align}
\label{eq:canonical_quantization}
\wh{q}_t = x_t, \quad \wh{p}_t = - i h \, \frac{\partial}{\partial x_t} \quad\mbox{on}\quad  L^2(\mathbb{R}^{4g-4+2n}, \, \bigwedge_t  dx_t),
\end{align}
so that
\begin{align}
\label{eq:canonical_commutation_relations}
[\wh{q}_t, \wh{p}_s] = i h \, \delta_{t,s}, \qquad [\wh{q}_t, \wh{q}_s] = [\wh{p}_t, \wh{p}_s] =0,
\end{align}
where $\delta_{t,s}$ is the Kronecker delta. More important is the rational transformation of $\{ (Y_t,Z_t) \}_t$ under a change of dotted triangulation, and its quantization. In Kashaev's quantization, each change of dotted triangulation is represented as the conjugation by a unitary operator, whose expression involves the quantum dilogarithm function again. As a result, one obtains another family of projective representations of $\Gamma_S$ on some Hilbert space, since elements of $\Gamma_S$ are realized as transformations of dotted triangulations. The author showed in \cite{Ki12} that the resulting projective representation is not equivalent to the one from Chekhov-Fock quantization, in the `universal' case. 

\vs

One thing to ponder is what the space parametrized by Kashaev's coordinates is exactly. First, since Kashaev's coordinates are only ratios of lambda lengths, they may not parametrize the full decorated Teichm\"uller space. More serious problem is that as the Kashaev coordinates are defined as difference or ratios, in many cases they are not independent of each other, and satisfy some linear or algebraic equations. In the above mentioned Kashaev quantization these equations are forgotten, which means that what is quantized is the total space of a certain fiber bundle over the space parametrized by the Kashaev coordinates; more precisely, as pointed out to the author by a referee, it is non-canonically identified with the space $P\til{\mathscr{T}}_S^d \times H^1(S,\mathbb{R})$, where $P\til{\mathscr{T}}_S^d = \mathscr{T}_S^d/\mathbb{R}_{>0}$ is Penner's projectivized decorated Teichm\"uller space. In any case, the end result provides an interesting and well-defined family of projective representations of $\Gamma_S$, which can be said to deform the classical coordinate changes. Another aspect to discuss is the use of dotted triangulations of $S$. One may first think that introducing the dots to triangles seems ad hoc. However, in certain senses, the dotted triangulations arise naturally in some important mathematical structures. For example, presciption of the decomposition orders of tensor products in a tensor category, that is, how to put parentheses in the expression of a tensor product like $(V_1\otimes (V_2\otimes V_3))\otimes V_4$, can be graphically encoded using the dotted triangulations, and the author, together with Igor Frenkel, proved that there indeed exists a tensor category of representations of some Hopf algebra, such that the operators associated to changes of decomposition orders, that is, to changes of dotted triangulations, are equivalent to the operators coming from the Kashaev quantization of Teichm\"uller spaces \cite{FrKi}. Also, Teschner \cite{T} used dotted triangulations as the main combinatorial ingredient in his attempt for a proof of Verlinde's conjecture that the quantum Teichm\"uller space is equivalent to the space of conformal blocks of Liouville conformal field theory. In addition, computations for Kashaev's quantization are usually easier than the ones for Chekhov-Fock's quantization.

\subsection{Quantization of higher Teichm\"uller spaces}

The lambda length coordinates and shear coordinates depend on the choice of a combinatorial data given by an ideal triangulation $\tau$ of $S$. To each $\tau$ we get a collection of coordinate functions associated to edges of $\tau$, whose Poisson structures \eqref{eq:2-form}, \eqref{eq:Poisson_bracket2} are described in terms of $\tau$. For each chosen edge $k$, we can flip $\tau$ along $k$ to get another triangulation $\tau'$, and the coordinate functions change according to certain rules \eqref{eq:Ptolemy_relation}, \eqref{eq:X_mutation}. A remarkable fact is that such a structure is a naturally arising example of a {\em cluster algebra} of Fomin and Zelevinsky \cite{FZ}, or more precisely, of a {\em cluster ensemble} of Fock and Goncharov \cite{FG03} which generalizes cluster algebras. Fock and Goncharov \cite{FG03} \cite{FG09} gave a quantization of such structures, which can then be applied to examples other than Teichm\"uller spaces. Prominent examples of cluster ensembles come from Fock-Goncharov's versions of what are called the {\em higher Teichm\"uller spaces} \cite{FG03} \cite{FG06}. Recall that the Teichm\"uller space is a certain component of the space of all representations $\pi_1(S) \to {\rm PSL}_2(\mathbb{R})$ modulo conjugation in ${\rm PSL}_2(\mathbb{R})$. For a split reductive real Lie group $G$, Fock and Goncharov \cite{FG06} defined {\em positive} representations $\pi_1(S) \to G$, and we study the moduli space $\mathscr{L}_{G,S}^+$ of positive representations $\pi_1(S) \to G$ modulo conjugation by $G$. For $G={\rm PGL}_2(\mathbb{R})$ we recover the classical Teichm\"uller space $\mathscr{T}_S$, and in general we call the moduli space $\mathscr{L}_{G,S}^+$ the {\em higher Teichm\"uller space}. As in the case of Teichm\"uller spaces, Fock-Goncharov defines certain fiber bundles over it concerning the data on the boundary of $S$, one being $\mathscr{A}_{G,S}^+$ as an analog of Penner's decorated Teichm\"uller space $\mathscr{T}_S^d$ parametrized by lambda lengths, another being $\mathscr{X}_{G,S}^+$ as an analog of the holed Teichm\"uller space $\mathscr{T}_S^+$ parametrized by shear coordinates; see \S\ref{subsec:higher_Teichmuller_spaces} of the present paper for their precise definitions. These two moduli spaces are also loosely called the {\em higher Teichm\"uller spaces} of $S$.

\vs

In the present paper, we only deal with the cases when $G = {\rm PGL}_m(\mathbb{R})$ or ${\rm SL}_m(\mathbb{R})$. For a chosen ideal triangulation $\tau$ of $S$, one divides each edge of $\tau$ into $m$ pieces, and by drawing grid lines `parallel' to the original edges one obtains an {\em $m$-triangulation associated to $\tau$}, which is a refinement of $\tau$. The newly introduced small edges that are not part of edges of $\tau$ are given orientations from those of edges of $\tau$, which in turn are given from the restriction of the orientation of the surface to ideal triangles. Thus one gets a collection of vertices($=$ intersections of grid lines) and arrows running among them, that is, a quiver. This quiver provides the combinatorial data needed in the definitions of cluster algebras and cluster ensembles. For each vertex one associates a coordinate function of $A$-type, described in \S\ref{subsec:A_X_coordinates}, which is a generalization of Penner's lambda length coordinate for $m=2$. In cluster theory one assigns a mutation at each vertex, which is a rule for transforming the quiver into another one, while transforming the collection of vertex coordinate functions into another in a certain way depending on the quiver. The $A$-type coordinates, which parametrize the space $\mathscr{A}_{{\rm SL}_m,S}^+$, change like Penner's lambda lengths under each mutation. Meanwhile, there is a degenerate $2$-form on $\mathscr{A}_{{\rm SL}_m,S}^+$ given analogously to \eqref{eq:2-form}, invariant under mutations. The triangles $t$ of $\tau$ in the sum \eqref{eq:2-form} should be replaced by the small `upside-down' triangles of the $m$-triangulation associated to $\tau$; here, the word `upside-down' should be understood with respect to the ideal triangle of $\tau$ in which the relevant small triangle is contained in. Suppose that $\tau'$ is obtained from $\tau$ by flipping along an edge of $\tau$. Then the quivers for the $m$-triangulations associated to $\tau$ and $\tau'$ are related by a certain sequence of $m(m-1)/2$ mutations. On the other hand, Fock-Goncharov \cite{FG06} defines $X$-type coordinates for vertices of the quiver, as certain cross ratios or triple ratios of $A$-type coordinates, depending on the quiver data. These $X$-type coordinates parametrize the space $\mathscr{X}_{{\rm PGL}_m,S}^+$, and transform like shear coordinates under mutations. There is a Poisson bracket on $\mathscr{X}_{{\rm PGL}_m,S}^+$ given analogously to \eqref{eq:Poisson_bracket2}, where $\varepsilon_{ij}$ is now the signed number of arrows from the vertex $i$ to $j$ in the quiver. This Poisson structure is invariant under mutations, and is compatible with the $2$-form on $\mathscr{A}_{{\rm SL}_m,S}^+$ with respect to the above-mentioned map $\mathscr{A}_{{\rm SL}_m,S}^+ \to \mathscr{X}_{{\rm PGL}_m,S}^+$ given on the level of functions by cross ratios and triple ratios. This pair $(\mathscr{X}^+_{{\rm PGL}_m,S}, \mathscr{A}_{{\rm SL}_m,S}^+)$ is an example of Fock-Goncharov's cluster ensemble.

\vs

Using Fock-Goncharov's quantization construction \cite{FG09}, one obtains a quantization of the space $\mathscr{X}_{{\rm PGL}_m,S}^+$; one replaces the $X$-type coordinate functions by positive-definite self-adjoint operators on a Hilbert space, and represent the mutations as conjugation by some unitary operators. Like in quantum Teichm\"uller theory, since elements of $\Gamma_S$ are realized as flips of ideal triangulations, and since flips of ideal triangulations are in turn realized as sequences of mutations of quivers for $m$-triangulations, one thus obtains a projective representation of $\Gamma_S$ from quantum higher Teichm\"uller theory. 

\vs

In the present paper, we consider an analog of Kashaev's ratio coordinate system associated to dotted triangulations and a corresponding quantization, in the case of higher Teichm\"uller theory. From the dots of the dotted triangulation of $\tau$, one can introduce dots to the small upside-down triangles of the $m$-triangulation associated to $\tau$. We formulate the resulting combinatorial data as a {\em quiver with dotted triangles} (see \S\ref{subsec:ratio_coordinates}). For each small dotted triangle of this quiver with dotted triangles, we define two coordinates as logarithm of ratios of the $A$-coordinates for the three vertices, just as Kashaev did in the case $m=2$. Then the $2$-form on $\mathscr{A}_{{\rm SL}_m,S}^+$ becomes `diagonal' again, so that the ratio coordinates admit canonical quantization. Each `mutation' of quiver with dotted triangles induces rational transformation on these ratio coordinates similarly as in Kashaev's case $m=2$. In order to describe the full sequence of mutations of quiver with dotted triangles that realizes a tranformation of underlying dotted triangulation of $S$, we introduce one more kind of `trivial change' of quiver with dotted triangles. The `mutation' is quantized using Kashaev's formula for $m=2$, and the `trivial change' is quantized by a certain unitary operator which is an analog of Fourier transformation.

\vs

What is done is as follows. First, for each small dotted triangle $t$, define analogs of the Kashaev coordinates by the formula \eqref{eq:log_Kashaev_coordinates}, where we replace $\lambda_{e_{t,j}}$ ($j=1,2,3$) by $\Delta_{P_{t,j}}$ ($j=1,2,3$) where $P_{t,j}$ are the vertices of $t$ and $\Delta_{P_{t,j}}$ is the corresponding $A$-type coordinate for $\mathscr{A}_{G,S}^+$ of Fock-Goncharov; for now, let us denote them by $q_t$ and $p_t$ again. Then, just as before, there are some linear relations satisfied by $q_t$'s and $p_t$'s. Consider the algebra $\mathscr{K}^h$ topologically generated by the symbols $\wh{q}_t$ and $\wh{p}_t$, modded out by the relations \eqref{eq:canonical_commutation_relations} and the quantized linear relations; we just replace $q_t$'s and $p_t$'s by $\wh{q}_t$'s and $\wh{p}_t$'s in the linear relations. Then the quantum version of the `mutation' and `trivial change' of quivers with dotted triangles are given as algebra automorphisms of $\mathscr{K}^h$, realized as conjugation by certain expressions in terms of $\wh{q}_t$'s and $\wh{p}_t$'s. In the end, to each transformation of dotted triangulations of $S$ we associate an automorphism of $\mathscr{K}^h$ given as the conjugation by some expression. It is easy to check that these automorphisms recover the classical coordinate changes of the new coordinates $q_t$'s and $p_t$'s as the quantum parameter $h$ goes to $0$; the main point to check for consistency of this quantization is that this association sends composition of transformations of dotted triangulations of $S$ to composition of automorphisms of $\mathscr{K}^h$. For this, we first show in \S\ref{subsec:proof_of_completeness} that the known `Kashaev relations' (see Lem.\ref{lem:relations_of_elementary_moves_of_Kashaev_groupoid} and \eqref{eq:new_elementary_moves_relations}) among elementary transformations of dotted triangulations generate the whole set of relations, which has been assumed in the literature without a complete proof, and then we show that the above automorphisms do satisfy these Kashaev relations. In the case of Kashaev quantization for $m=2$, this was proved without using the linear relations of $\wh{q}_t$'s and $\wh{p}_t$'s, and therefore we just drop these equations and represent the corresponding algebra by \eqref{eq:canonical_quantization} on a Hilbert space, yielding a projective representation of $\Gamma_S$. In the present paper, we prove this consistency for $m=3$, {\em using} the linear relations of $\wh{q}_t$'s and $\wh{p}_t$'s. In particular, this time we cannot represent $\wh{q}_t$'s and $\wh{p}_t$'s as we did in \eqref{eq:canonical_quantization}, because those operators do not satisfy the desired linear relations. So our quantization is done only in an algebraic sense at the moment. However, one may be able to find a suitable natural representation of the algebra $\mathscr{K}^h$ on some Hilbert space where $\wh{q}_t$'s and $\wh{p}_t$'s are represented as self-adjoint operators, which can be a topic of future research. This then would yield a new family of projective representations of mapping class groups. In the meantime, one may search for a proof for $m>3$ too.

\vs

In \cite{FrKi}, Kashaev's quantization of Teichm\"uller spaces is reconstructed in the tensor category of certain class of representations of the Hopf algebra related to the Borel subalgebra of $\mcal{U}_q(\frak{sl}(2,\mathbb{R}))$. Ivan Ip \cite{Ip14} studied the representation theory of a similar algebra with $\mcal{U}_q(\frak{sl}(2,\mathbb{R}))$ replaced by $\mcal{U}_q(\frak{sl}(m,\mathbb{R}))$, and showed that it has a similar tensor category structure. Thus, as an analog of \cite{FrKi}, from these categories we obtain certain representations of the groupoid of transformations of dotted triangulations of certain genus $0$ surfaces. We expect that it is related to the version of quantum higher Teichm\"uller theory constructed in the present paper. This, which is in fact the motivation of the present paper, is work in progress with Ip; it then may also lead to a higher generalization of Teschner's work on `modular functors' \cite{T}. In addition, one can ask what the moduli space coordinatized by the new ratio coordinates constructed in the present paper is, i.e., what kind of structure on $S$ it parametrizes. This new space $\mathscr{K}$ has an interesting geometry in itself, because it is a positive space, as its cousins $\mathscr{A}$ and $\mathscr{X}$ are.

\vs

\noindent{\bf Acknowledgments.} I thank Ivan Ip for motivation and helpful discussions. I thank Dylan Allegretti for his help on understanding the works of Fock-Goncharov. Finally, I'd like to thank the referee for the reviewing and for helpful comments.

\section{Ratio coordinate for higher Teichm\"uller space}
\label{sec:ratio_coordinate}

\subsection{Higher Teichm\"uller spaces $\mathscr{X}^+_{{\rm PGL}_m, \wh{S}}$ and $\mathscr{A}^+_{{\rm SL}_m,\wh{S}}$}
\label{subsec:higher_Teichmuller_spaces}

In the present subsection, we would like to give definitions of the objects that are being coordinatized, which we collect mainly from \cite{FG06}, and also from \cite{FG03}. Let $G$ be a split reductive algebraic group over $\mathbb{Q}$. In the present paper, $G$ is either ${\rm SL}_m$ or ${\rm PGL}_m$, where $m\ge 2$. One can feed in any base field $F$ to get ${\rm SL}_m(F)$ or ${\rm PGL}_m(F)$; we may let $F=\mathbb{R}$ for our purposes. Let $S$ be a compact oriented surface with $n\ge 0$ holes, i.e. $S = \ol{S} - D_1 \cup \ldots \cup D_n$, where $\ol{S}$ is a compact oriented real $2$-dimensional manifold of genus $g$ without boundary, while $D_1,\ldots,D_n$ are non-intersecting open discs in $\ol{S}$.
\begin{definition}[\cite{FG06}]
A {\em marked surface $\wh{S}$} is a pair $(S,\{x_1,\ldots,x_k\})$, where $S$ is a compact oriented surface with holes, and $\{x_1,\ldots,x_k\}$ ($k\ge 0$) is a finite set of distinct points on the boundary $\partial S$ which are  considered up to isotopy, called {\em boundary points}. The {\em punctured boundary $\partial \wh{S}$} of $\wh{S}$ is defined as $\partial \wh{S} = \partial S - \{x_1,\ldots,x_k\}$. Each connected component of $\partial \wh{S}$ that is not a full circle without a boundary point is called a {\em boundary arc}.
\end{definition}
\begin{definition}[\cite{FG06}]
For a marked surface $\wh{S}$, let $N$ be the number of all connected components of $\partial\wh{S}$. We say $\wh{S}$ is {\em hyperbolic} if 1) $g>1$, 2) $g=1$, $N>0$, or 3) $g=0$, $N\ge 3$.
\end{definition}

Let $\mathscr{B}$ denote the flag variety of $G$, parametrizing all Borel subgroups of $G$. If we choose a Borel subgroup $B$ of $G$, then $\mathscr{B} = G/B$. Let $U := [B,B]$, which is a maximal unipotent subgroup in $G$, and let $\mathscr{A}$ denote the principal affine variety of $G$. These become more concrete for $G = {\rm SL}_m$. We can choose $B$ to be the group of all upper triangular matrices in ${\rm SL}_m$, and take $U$ to be the group of all upper triangular matrices in ${\rm SL}_m$ with all diagonal entries being $1$. Let $\mathscr{B}$ be the space of all full flags in $F^m$, where $F$ is the base field and a {\em full flag} in $F^m$ is a filtration of vector spaces $V_1 \subseteq V_2 \subseteq \cdots \subseteq V_m = F^m$ such that $\dim_F V_j = j$ for each $j$. Thinking of the bases, one can encode this as an ordered $m$-tuple of vectors $(v_1,v_2,\ldots,v_m)$, where $v_1,\ldots,v_j$ span $V_j$; so, two tuples $(v_1,\ldots,v_m)$ and $(v_1',\ldots,v_m')$ represent a same flag if ${\rm Span}\{v_1,\ldots,v_j\} = {\rm Span}\{v_1',\ldots,v_j'\}$ for each $j$. Then $G = {\rm SL}_m$ acts transitively on $\mathscr{B}$ on the left by $g.(v_1,\ldots,v_m) = (gv_1,\ldots,gv_m)$, where the stabilizer of the standard flag $(e_1,\ldots,e_m)$ is the subgroup $B$ of upper triangular matrices; thus $\mathscr{B} = G/B$. On the other hand, let $\mathscr{A}$ be the space of all full affine flags in $F^m$, where a {\em full affine flag}\footnote{this may not be a standard terminology} is an $m$-tuple of vectors $(v_1,\ldots,v_m)$ where the equivalence relation of $(v_1,\ldots,v_m)$ and $(v_1',\ldots,v_m')$ is now given by the condition $v_1\wedge v_2 \wedge \cdots \wedge v_j = v_1' \wedge \cdots \wedge v_j'$ for each $j$. Then $G={\rm SL}_m$ acts transitively on $\mathscr{A}$ as before, where the stabilizer of the standard affine flag $(e_1,\ldots, e_m)$ is now the subgroup $U$ of unipotent upper triangular matrices; thus $\mathscr{A} = G/U$.

\vs

Now, let $\mathscr{L}$ be a $G$-local system on $S$, that is, a principal $G$-bundle on $S$ equipped with a flat connection. The fiber of this bundle at each point of $S$ is a principal homogenous space of $G$, that is, a copy of $G$ without multiplication structure, and the flat connection tells us how the fibers at nearby points are `connected'. If we assume that $G$ acts on $\mathscr{L}$ from right, the associated {\em flag bundle $\mathscr{L}_\mathscr{B}$} and the {\em principal affine bundle $\mathscr{L}_\mathscr{A}$} are defined by
$$
\mathscr{L}_\mathscr{B} := \mathscr{L} \times_G \mathscr{B} = \mathscr{L}/B; \qquad\qquad \mathscr{L}_\mathscr{A} := \mathscr{L}/U.
$$
Observe that we can use the same description of $G/B$ as space of full flags in $F^m$, also in the case $G = {\rm PGL}_m$.

\begin{definition}[Def.1.2.~of \cite{FG06}]
Let $G = {\rm PGL}_m$. A {\em framed $G$-local system on a marked surface $\wh{S}$} is a pair $(\mathscr{L},\beta)$ where $\mathscr{L}$ is a $G$-local system on $S$ and $\beta$ is a flat section of the restriction of the flag bundle $\mathscr{L}_\mathscr{B}$ to the punctured boundary $\partial \wh{S}$.

Let $\mathscr{X}_{G,\wh{S}}$ be the moduli space of framed $G$-local systems on $\wh{S}$.
\end{definition}
The above definition also works for any split reductive algebraic group $G$ \cite{FG06}. One way of viewing the data $(\mathscr{L},\beta)$ is as follows. From the monodromy of the flat connection of $\mathscr{L}$ we get a homomorphism $\pi_1(S) \to G$, which determines how the parallel transport is done. So, once we choose a flag at a point on $S$, we can drag it along a path in $S$ to a flag at a different point, in a `flat' manner; this dragging depends only on the homotopy class of the path. At each connected component of the punctured boundary $\partial \wh{S}$ we choose a point, and choose a flag at that point; these completely determine the data $\beta$, with the help of parallel transport. See \cite[Chap.2]{FG06} for more details. The following definition can be interpreted in a similar way.

\begin{definition}[Def.1.3.~of \cite{FG06}]
\label{def:A-space}
Suppose $m\ge 3$ is odd, and let $G={\rm SL}_m$. A {\em decorated $G$-local system on a marked surface $\wh{S}$} is a pair $(\mathscr{L}, \alpha)$, where $\mathscr{L}$ is a $G$-local system on $S$, and $\alpha$ a flat section of the restriction of $\mathscr{L}_\mathscr{A}$ to the punctured boundary $\partial \wh{S}$.

Let $\mathscr{A}_{G,\wh{S}}$ be the moduli space of decorated $G$-local systems on $\wh{S}$.
\end{definition}
A definition similar to Def.\ref{def:A-space} works for the cases when $G$ is of type $E_6,E_8,F_4,G_2$, where $G$ is simply-connected and the center of $G$ has odd order \cite{FG06}. For other cases, definition of $\mathscr{A}_{G,\wh{S}}$ is more subtle. Here we state the definition in \cite{FG06} for the cases $G = {\rm SL}_m$ for even $m$. Let $w_0$ a natural lift in $G$ of the longest Weyl group element, and let $s_G = w_0^2$. Then $s_G$ is in the center of $G$, and square of $s_G$ is identity. In case of $G= {\rm SL}_m$ with even $m$, this element $s_G$ is the diagonal matrix with all the diagonal entries being $-1$.  Let $T'S$ be the tangent bundle to $S$, with the zero section removed. Its fundamental group $\pi_1(T'S, x)$ is a central extension of $\pi_1(S,y)$ by $\mathbb{Z}$, where $x\in T'_yS$. Denote by $\sigma_S$ a generator of this central subgroup $\mathbb{Z}$; then $\sigma_S$ is well defined up to a sign. For each boundary component $C_i$ of $S$, let ${\bf C}_i$ be a little annulus in $S$ containing $C_i$ as its boundary, and let $x_1,\ldots,x_p$ be the boundary points on $C_i$. Identify ${\bf C}_i$ as $C_i \times [0,1]$, and let ${\bf C}_i'$ be the part of ${\bf C}_i$ corresponding to $(C_i \setminus\{x_1,\ldots,x_p\}) \times [0,1]$.
\begin{definition}[Def.2.3, 2.4.~of \cite{FG06}]
Suppose $m\ge 2$ is even, and let $G = {\rm SL}_m$. A {\em twisted $G$-local system on $S$} is a local system on $T'S$ with the monodromy $s_G$ around $\sigma_S$.
Let $\mathscr{L}$ be a $G$-local system on $T'S$ representing a twisted local system on $S$. A {\em decoration} on $\mathscr{L}$ is the choice of a locally constant section $\alpha$ of the restriction of the principal affine bundle $\mathscr{L}_\mathscr{A}$ to $\cup_i {\bf C}_i'$. 
A {\em decorated} twisted $G$-local system on a marked surface $\wh{S}$ is a pair $(\mathscr{L}, \alpha)$, where $\mathscr{L}$ is a twisted $G$-local system on $S$, and $\alpha$ is a decoration on $\mathscr{L}$. 

Let $\mathscr{A}_{G,\wh{S}}$ be the moduli space of decorated twisted $G$-local systems on $\wh{S}$.
\end{definition}

One of the main results of \cite{FG06} is that these moduli spaces are `positive':
\begin{theorem}[Thm.1.4. of \cite{FG06}]\label{thm:positive_moduli_spaces} Let $G$ be a split semi-simple simply-connected algebraic group, and $G' = G/{\rm Center}(G)$. Let $\wh{S}$ be a marked hyperbolic surface with $n>0$ holes. Then the moduli spaces $\mathscr{X}_{G',\wh{S}}$ and $\mathscr{A}_{G,\wh{S}}$ have positive atlases.
\end{theorem}

A space having a positive atlas roughly means that it has a collection of coordinate systems, such that the coordinate change map between any two coordinate systems is a positive rational map. For our purposes, we can think of a coordinate system as an identification of an open dense subset of the space with $(F^*)^\ell$ for some positive integer $\ell$, where $F$ is the base field. A positive rational map is a map in the form $f_1/f_2$ where $f_1,f_2$ are $\mathbb{Z}_{>0}$-linear combinations of monomials in the corresponding coordinates; the open subsets are glued together along these (bi)rational maps. See \cite{FG03} for a precise definition using `coordinate groupoids'. Then it makes sense to think of all points of the space whose coordinates are all positive real, for one coordinate system, hence for all coordinate systems.

\begin{definition}[Def.1.8.~of \cite{FG06}: higher Teichm\"uller spaces]
Let $G,\wh{S}$ be as in Thm.\ref{thm:positive_moduli_spaces}. The {\em higher Teichm\"uller spaces} $\mathscr{X}^+_{G',\wh{S}}$ and $\mathscr{A}_{G,\wh{S}}^+$ are defined as the $\mathbb{R}_{>0}$-points $\mathscr{X}_{G',\wh{S}}(\mathbb{R}_{>0})$ and $\mathscr{A}_{G,\wh{S}}(\mathbb{R}_{>0})$ of the moduli spaces $\mathscr{X}_{G,\wh{S}}$ and $\mathscr{A}_{G,\wh{S}}$, respectively.
\end{definition}

Moreover, we also have a map from the $\mathscr{A}$-space to the $\mathscr{X}$-space, induced by the canonical map $G \to G'$, as in \cite{FG06}. Since $s_G$ is in the center of $G$, a twisted $G$-local system $\mathscr{L}$ on $T'S$ provides a $G'$-local system $\mathscr{L}'$ on $S$. The canonical projection $p : \mathscr{L}_\mathscr{A} \to \mathscr{L}'_\mathscr{B}$ yields the map
\begin{align}
\label{eq:p_definition}
p : \mathscr{A}_{G,\wh{S}} \to \mathscr{X}_{G',\wh{S}}, \qquad (\mathscr{L},\alpha) \mapsto (\mathscr{L}', \beta), \quad \beta:= p(\alpha).
\end{align}
Thus we now have definitions of the moduli spaces $\mathscr{X}_{{\rm PGL}_m, \wh{S}}^+$ and $\mathscr{A}_{{\rm SL}_m, \wh{S}}^+$, together with the map $p : \mathscr{A}_{{\rm SL}_m,\wh{S}} \to \mathscr{X}_{{\rm PGL}_m,\wh{S}}$ \eqref{eq:p_definition}.

\subsection{$A$-type and $X$-type coordinates}
\label{subsec:A_X_coordinates}

We now describe the coordinate systems of Fock-Goncharov \cite{FG06} for the spaces $\mathscr{X}_{{\rm PGL}_m,\wh{S}}$ and $\mathscr{A}_{{\rm SL}_m,\wh{S}}$, for a marked hyperbolic surface $\wh{S} = (S,\{x_1,\ldots,x_k\})$ with $n>0$ holes, where $m\ge 2$. These coordinate systems require the choice of a combinatorial data on the surface, called an {\em $m$-triangulation}. First, shrink each boundary component of $\wh{S}$ without boundary points to a puncture. 
\begin{definition}[\cite{FG06}: ideal triangulation]
\label{def:ideal_triangulation}
An {\em ideal triangulation} of $\wh{S}$ is a triangulation of $S$ up to isotopy, with the vertices at either the punctures or the boundary arcs, such that each boundary arc carries exactly one vertex of the triangulation, and each puncture serves as a vertex. The edges of an ideal triangulation either lie inside of $S$, or inside the boundary of $S$.
\end{definition}
We also call the edges of an ideal triangle of an ideal triangulation the {\em sides} of that triangle.
\begin{definition}[\cite{FG06}: $m$-triangulation]
\label{def:m-triangulation}
Take the triangle
$$
T_m := \{(x,y,z) \in \mathbb{R}^3 : x+y+z=m, \,x,y,z \ge 0\},
$$
and consider its triangulation given by the lines $\{x=p\}$, $\{y=p\}$, and $\{z=p\}$ in $T_m$, where $p \in \mathbb{Z}$, $0\le p\le m$. An {\em $m$-triangulation} of a triangle is a triangulation isotopic to this one. An edge of an $m$-triangulation is called an {\em internal edge} if it does not lie on a side of the original triangle. The {\em $m$-triangulation} of an ideal triangulation $\tau$ of $\wh{S}$ is the triangulation obtained by $m$-triangulating all ideal triangles of $\tau$.
\end{definition}
In particular, each edge of $\tau$ contains $m+1$ vertices of the $m$-triangulation of $\tau$. The coordinate for each integral point on the triangle $T_m$ tells us how `far' the point is from each edge of $T_m$; for example, the $x$-coordinate tells the distance to the line $\{x=0, \, y+z=m\}$, i.e., how many internal edges we need to reach the line from the point. Two of the coordinates being zero means that the point is one of the vertices of $T_m$, and one of the coordinates being zero means that the point is on one of the sides of $T_m$. For each ideal triangle of $\tau$, the orientation of the surface $S$ induces a clockwise orientation on the edges of this ideal triangle, and hence an orientation on the internal edges of the $m$-triangulation of $\tau$, by setting the orientation of each internal edge to be the same as that of an edge of $\tau$ parallel to it. So we get a certain quiver on the surface $S$; we recall that a quiver is a graph with oriented edges, and we also call the oriented edges of a quiver {\em arrows}. See Fig.\ref{fig:til_T_ts_realization} for an example, induced by the clockwise orientation of the plane. One can show that the quiver we just obtained has no cycle of length $1$ or $2$; this follows from the corresponding statement for $m=2$.

\vs

We shall assign Fock-Goncharov coordinates to elements of the following set:
\begin{definition}[\cite{FG06}, modified]
For an ideal triangulation $\tau$ of $\wh{S}$, define the two sets:
\begin{align*}
I_m^\tau & = \{\mbox{vertices of the $m$-triangulation of $\tau$}\} - \{\mbox{vertices of $\tau$}\}, \\
J_m^\tau & = I_m^\tau - \{\mbox{the vertices at the boundary arcs of $S$}\}.
\end{align*}
\end{definition}
Now let us describe how the $A$-type coordinates are defined in \cite{FG06}. Consider an element $v$ of $I_m^\tau$. Suppose that it is contained in an ideal triangle $t$ of $\tau$. Any element of $I_m^\tau$ contained in $t$ can be described as $(a_1,a_2,a_3)$ for some nonnegative integers $a_1,a_2,a_3$ with $a_1+a_2+a_3=m$, as in Def.\ref{def:m-triangulation}. These numbers tell us how far $v$ is from each edge of $t$. Suppose $a_1,a_2,a_3>0$ for now. Let $u_1,u_2,u_3$ be the vertices of $t$, so that the distance from $v$ to the edge of $t$ opposite to $u_j$ is $a_j$, which is the minimal number of internal edges needed to connect $v$ to that edge; assume that we arranged the situation so that $u_1\to u_2\to u_3\to u_1$ is clockwise direction, with respect to the orientation of $t$ inherited from that of $S$. The data $\alpha$ of a given point $(\mathscr{L},\alpha)$  of $\mathscr{A}_{{\rm SL}_m,\wh{S}}$ assign a full affine flag to each of the vertices of $\tau$ that are either at punctures of $S$ or at the boundary arcs of $S$; see the similar discussion just before Def.\ref{def:A-space} about the data $\beta$. Using the flat connection of $\mathscr{L}$, we drag the three full affine flags at $u_1,u_2,u_3$ to $v$ along paths in $t$, by parallel transport; denote these full affine flags at $v$ by $(x_1,\ldots,x_m)$, $(y_1,\ldots,y_m)$, and $(z_1,\ldots,z_m)$, respectively. Then we form an element $x_1\wedge \cdots \wedge x_{a_1} \wedge y_1 \wedge \cdots \wedge y_{a_2} \wedge z_1 \wedge \cdots \wedge z_{a_3}$ of the top exterior power of $F^m$. By pairing with a monodromy invariant volume form on $F^m$, we get a real number, and this is the $A$-coordinate $\Delta_v(\mathscr{L}, \alpha)$ associated to the vertex $v$. For the volume form, it seems that we can choose any single volume form on $F^m$ and use it all the time. Now suppose that one of $a_1,a_2,a_3$ is zero; then, since $v\in I_m^\tau$, we see that exactly one of them is zero. Without loss of generality, let $a_3=0$. Then we use the element $x_1\wedge\cdots\wedge x_{a_1} \wedge y_1 \wedge \cdots\wedge y_{a_2}$ in the above description, and by pairing with the chosen volume form we get $\Delta_v(\mathscr{L}, \alpha)$ associated to the vertex $v$. See \cite{FG06} for more details, for example to see why these are well-defined; for a precise description, they use the notion of `face paths'. Fock-Goncharov \cite{FG06} describes the $X$-coodinates of $\mathscr{X}_{{\rm PGL}_m,\wh{S}}$, as cross-ratios and triple-ratios of $A$-coordinates. The result can be described in terms of the map $p: \mathscr{A}_{{\rm SL}_m,\wh{S}} \to \mathscr{X}_{{\rm PGL}_m,\wh{S}}$ \eqref{eq:p_definition}, given by
\begin{align}
\label{eq:our_p}
p^* X_i = \prod_{j \in I_m^\tau} (\Delta_j)^{\varepsilon_{ij}},
\end{align}
where $\varepsilon_{ij}$ is the the number of arrows from $i$ to $j$ minus the number of arrows from $j$ to $i$, in the quiver of the $m$-triangulation of $\tau$.

\subsection{Coordinate rings and cluster structure of $\mathscr{X}$-space and $\mathscr{A}$-space}
\label{subsec:coordinae_rings}

We now recall the definition of a {\em cluster ensemble}, defined in \cite{FG03}, as an analog of a cluster algebra which is defined in \cite{FZ}. We use a slightly modified version as used in \cite{FG06}. We first need to define a combinatorial data, called {\em seed}.
\begin{definition}[seed]
A {\em seed} ${\bf i} = (I,J,\varepsilon_{ij}, d_i)$ consists of a finite set $I$, an integer-valued function $(\varepsilon_{ij})$ on $I \times I$, called a {\em cluster function}, a $\mathbb{Q}_{>0}$-valued {\em symmetrizer function $d_i$} on $I$ such that $\til{\varepsilon}_{ij} := d_i \varepsilon_{ij}$ is skew-symmetric, and a subset $J$ of $I$. The complement $I - J$ is called the {\em frozen subset} of $I$.\footnote{A `seed' in the usual cluster theory is equipped with cluster variables, like our upcoming $X_i$'s and $\Delta_i$'s. A seed as defined here is called a `feed' in \cite{FG09} as a joke, to be distinguished from a usual `seed'.}
\end{definition}
If $\varepsilon_{ij}$ is skew-symmetric, we set $d_i=1$. There is a rule for transforming a seed into another one `along the direction' of any element $k$ of $J$, which is usually called a {\em mutation}.
\begin{definition}[mutation]
\label{def:mutation}
Given a seed ${\bf i} = (I,J,\varepsilon, d)$, to every non-frozen element $k \in J$ is associated a new seed $\mu_k({\bf i}) = {\bf i}' = (I',J',\varepsilon',d')$, given by $I' := I$, $J' := J$, $d' := d$, and 
\begin{align}
\varepsilon'_{ij} := 
\left\{ {\renewcommand{\arraystretch}{1.2}
\begin{array}{ll}
-\varepsilon_{ij} & \mbox{if } k\in \{i,j\}, \\
\varepsilon_{ij} + \frac{ |\varepsilon_{ik}| \varepsilon_{kj} + \varepsilon_{ik} |\varepsilon_{kj}|}{2} & \mbox{if } k \notin \{i,j\}.
\end{array}} \right.
\end{align}
We say ${\bf i}'$ is obtained by applying to ${\bf i}$ the {\em mutation in the direction $k$}, or {\em mutation at $k$}, and write as $\mu_k$ in short.
\end{definition}
As observed in \cite{FG03}, a seed is a version of the notion of a quiver, enhanced by multipliers at vertices. In the case when $\varepsilon_{ij}$ is skew-symmetric and therefore $d_i =1$, one can visualize a seed as a quiver, as follows. The set of vertices is in bijection with $I$, so one can label the vertices by elements of $I$. Between the vertices $i$ and $j$, there are $|\varepsilon_{ij}|$ arrows from $i$ to $j$ if $\varepsilon_{ij}>0$. In particular, all arrows between any two vertices are of a same orientation, and there is no arrow from a vertex to itself. The following mutation rule for quivers is well-known.

\begin{lemma}[quiver mutation]
\label{lem:quiver_mutation}
Seed mutation $\mu_k$ in the direction $k \in J$ can be interpreted in terms of a {\em quiver mutation} as follows. First, from the initial quiver associated to the seed ${\bf i}$, reverse the direction of all the arrows going from or to the vertex $k$. Now, for each ordered pair $(i,j)$ of distinct vertices $i,j \in I$ different from $k$, we `add' $\varepsilon_{ik} \varepsilon_{kj}$ arrows from $i$ to $j$ if $\varepsilon_{ik}>0$ and $\varepsilon_{kj}>0$. This addition of arrows can be understood graphically as `completing cycles of length $3$ passing through $k$ in the middle'. Then, we make cancellation of arrows after such an addition, i.e. if there are two arrows running between two vertices with opposite orientations then we delete both arrows; repeat until there are no such cycles of length $2$.
\end{lemma}

To each seed is associated a split algebraic torus, and to each mutation a birational map between the tori, which is in fact the heart of the cluster theory. Recall that a split algebraic torus is a product of the multiplicative group $\mathbb{G}_m$, where $\mathbb{G}_m(F) = F^*$ for each field $F$. For a cluster ensemble we consider the following two split algebraic tori for a given seed ${\bf i}$:
\begin{align}
\label{eq:seed_tori}
\mathscr{X}_{\bf i} := (\mathbb{G}_m)^J, \qquad\qquad
\mathscr{A}_{\bf i} := (\mathbb{G}_m)^I,
\end{align}
called the {\em seed $\mathscr{X}$-torus} and the {\em seed $\mathscr{A}$-torus}, respectively. Let $\{X_i\}_{i\in J}$ and $\{\Delta_i\}_{i\in I}$ be the natural coordinates for each of the two tori. The seed mutation $\mu_k$ from ${\bf i}$ to ${\bf i}'$ induces a rational map between the corresponding two $\mathscr{X}$-tori and that between the corresponding two $\mathscr{A}$-tori. These rational maps are denoted by $\mu_k : \mathscr{X}_{\bf i} \to \mathscr{X}_{{\bf i}'}$ and $\mu_k : \mathscr{A}_{\bf i} \to \mathscr{A}_{{\bf i}'}$, given by the formulas
\begin{align}
\label{eq:X_mutation_formula}
& \mu_k^* X_i' = \left\{
{\renewcommand{\arraystretch}{1.2}
\begin{array}{ll}
X_k^{-1} & \mbox{if $i=k$}, \\
X_i(1 + X_k^{-{\rm sgn}(\varepsilon_{ik})})^{-\varepsilon_{ik}} & \mbox{if $i\neq k$}, 
\end{array}}
\right. \\
\label{eq:A_mutation_formula}
& \mu_k^* \Delta_i' = \left\{
{\renewcommand{\arraystretch}{1.2}
\begin{array}{ll}
\Delta_i & \mbox{if $i \neq k$}, \\
\displaystyle \Delta_k^{-1} \left( \prod_{\{ j \in I : \varepsilon_{kj}>0 \} } \Delta_j^{\varepsilon_{kj}}  +  \prod_{\{ j \in I : \varepsilon_{kj}<0 \} } \Delta_j^{-\varepsilon_{kj}} \right) & \mbox{if $i = k$},
\end{array}}
\right.
\end{align}
where $X_i'$ and $\Delta_i'$ denote the cluster coordinates of the tori associated to the mutated seed $\mu_k({\bf i}) = {\bf i}'$, and ${\rm sgn}(\varepsilon)$ is defined to be $1$ if $\varepsilon\ge 0$ and $-1$ otherwise. If exactly one of the two sets $\{ j \in I : \varepsilon_{kj} >0 \}$ and $\{ j \in I : \varepsilon_{jk} < 0 \}$ is empty, then the corresponding product $\prod_{\{ j \in I : \varepsilon_{kj}>0 \} } \Delta_j^{\varepsilon_{kj}}$ or $\prod_{\{ j \in I : \varepsilon_{kj}<0 \} } \Delta_j^{-\varepsilon_{kj}}$ in \eqref{eq:A_mutation_formula} is set to be $1$. If $\varepsilon_{kj} = 0$ for all $j\in I$, then $\mu_k^* \Delta_k' = 2 \Delta_k^{-1}$, and we set $\mu_k^* X_k' = X_k^{-2}$.

\vs

In our example, we set $I = I_m^\tau$, $J = J_m^\tau$ for an ideal triangulation $\tau$ of a marked hyperbolic surface $\wh{S}$, where $A$- and $X$-coordinates are as given in \S\ref{subsec:A_X_coordinates}.
\begin{proposition}[\cite{FG03}]
For our case just mentioned, the map $p$ \eqref{eq:our_p} commutes with the mutations \eqref{eq:X_mutation_formula} and \eqref{eq:A_mutation_formula}.
\end{proposition}

\vs

For a marked hyperbolic surface $\wh{S}$ and an ideal triangulation $\tau$ of it, the space $\mathscr{A}_{{\rm SL}_m,\wh{S}}$ is equipped with the `Weil-Petersson' $2$-form
\begin{align}
\label{eq:Omega_SL_m}
\Omega_{{\rm SL}_m,\wh{S}} = \sum_{i_1, i_2 \in I_m^\tau} \varepsilon_{i_1,i_2} \, d\log \Delta_{i_1} \wedge d \log \Delta_{i_2},
\end{align}
invariant under the mutation formula \eqref{eq:A_mutation_formula} in the sense described in \S\ref{subsec:quantization_of_Teichmuller_spaces}, and the space $\mathscr{X}_{{\rm PGL}_m,\wh{S}}$ has a Poisson structure given by
$$
\{ X_{j_1}, X_{j_2} \}_{{\rm PGL}_m,\wh{S}} =  \varepsilon_{j_1,j_2} \, X_{j_1} X_{j_2}, \qquad \forall j_1,j_2\in J^\tau_m,
$$
invariant under the mutation formula \eqref{eq:X_mutation_formula} in the sense described in \S\ref{subsec:quantization_of_Teichmuller_spaces}. One may view these invariances as well-definedness of the above two equalities. Then these two structures are compatible under the map $p$; the fibers of $p$ are the leaves of the null-foliation of $\Omega_{{\rm SL}_m,\wh{S}}$, and the symplectic structure induced on $p(\mathscr{A}_{{\rm SL}_m,\wh{S}})$ coincides with the restriction of the Poisson structure on $\mathscr{X}_{{\rm PGL}_m,\wh{S}}$. These should be formulated more precisely using the groupoid of ideal triangulations of $\wh{S}$, but roughly speaking, these data, the $\mathscr{A}$-space, the $\mathscr{X}$-space, and the map $p$, constitute a {\em cluster ensemble} of Fock-Goncharov \cite{FG03}.

\subsection{Ratio coordinates}
\label{subsec:ratio_coordinates}

For an analog of dotted triangulations of Kashaev, we consider the $m(m-1)/2$ `upside-down triangles' contained in each of the ideal triangle of an ideal triangulation $\tau$ of a marked surface $\wh{S}$, where by upside-down triangles we mean the triangles formed by internal edges of the $m$-triangulation (Def.\ref{def:m-triangulation}) of a triangle which are $1/m$-scaled and $180$ degree rotated copies of the parent triangle. To each upside-down triangle, we choose a distinguished corner, indicated by a dot $\bullet$, and we shall assign two new coordinates to each of these dotted upside-down triangles. In order to conveniently keep track of the `mutations' of such data, we introduce the notion of a {\em quiver with dotted triangles}.
\begin{definition}[quiver with dotted triangles]
Consider a quiver. A {\em shaded triangle} is an  unordered choice of three distinct vertices of the quiver, such that there exist at least two edges of the quiver running among them, together with a distinguished vertex out of the three. A shaded triangle is called {\em ordinary} if all its sides are edges of the quiver, and {\em defective} if one of its sides is not an edge of the quiver; we call the side of a defective shaded triangle that is not an edge of the quiver an {\em invisible edge}. We allow only the cases when the sides of a shaded triangle are cyclically oriented, whether it is ordinary or defective.

\vs

A {\em quiver with dotted triangles} is a quiver, together with the choice of a set of shaded triangles, along with a bijection from this set to an index set $I$, such that
\begin{enumerate}
\item every vertex of the quiver is a vertex of at least one shaded triangle,
\item a defective shaded triangle always appear in a pair, where these two share one side, which is an invisible edge, and the sides of thus formed quadrilateral are cyclically oriented.
\end{enumerate}
\end{definition}
Each shaded triangle is depicted with grey color in pictures (following its name), the distinguished vertex is indicated by a dot $\bullet$ in that corner, an invisible edge is drawn as a dotted line, and the $I$-labels of triangles are written within brackets $[~]$. See Fig.\ref{fig:til_T_ts_realization} for examples.

\vs

Let us give an example of a quiver with dotted triangles, denoted by $Q_m$, associated to a `dotted triangulation' of a marked hyperbolic surface $\wh{S}$.
\begin{definition}[\cite{Ki12}]
\label{def:dotted_triangulation}
A {\em dotted triangulation $(\tau,D,L)$} of $\wh{S}$ is an ideal triangulation $\tau$ of $\wh{S}$ together with the rule $D$ assigning a distinguished corner to each ideal triangle of $\tau$ indicated by a dot $\bullet$ in pictures, and a bijection $L$ from the set of all ideal triangles of $\tau$ to some index set $I$. We denote $(\tau,D,L)$ by $\tau_{\rm dot}$ if $D$ and $L$ are clear from the context.
\end{definition}
Let $\tau_{\rm dot} = (\tau,D,L)$ be a dotted triangulation of $\wh{S}$. Consider the $m$-triangulation of $\tau$, and let the quiver associated to it, as described in \S\ref{subsec:higher_Teichmuller_spaces}, be the underlying quiver for $Q_m$. In particular, the set of vertices of this quiver equals $I_m^\tau$. All the upside-down triangles constitute the set of shaded triangles of $Q_m$. For each ideal triangle of $\tau$, the choice of a distinguished corner for this triangle from the data $\tau_{\rm dot}$ induces dots for all the upside-down triangles contained in it, by assigning to each upside-down triangle the dot at the  corner farthest to the distinguished corner of its parent ideal triangle. See Fig.\ref{fig:til_T_ts_realization} for an example. The $m(m-1)/2$ shaded triangles in an ideal triangle $t$ get labeled as follows: rotate the whole ideal triangle $t$ so that its distinguished corner is placed on top. Then we divide the shaded triangles into $m-1$ rows; the first row has one shaded triangle, and the second has two, and so on. The $c$-th shaded triangle from the left at the $r$-th row gets labeled by the symbol $t_{r,c}$; so $r$ and $c$ indicate the row and column numbers representing the position of the shaded triangle $t_{r,c}$ with respect to the distinguished corner of $t$. For later use, let us define the index set $\mcal{S}_m$ for the shaded triangles of each ideal triangle:
\begin{align}
\label{eq:mcal_S_m}
\mcal{S}_m := \{ (r,c) \in \mathbb{Z}^2 \, : \, 1\le r\le m-1, ~ 1\le c \le r \}.
\end{align}
This completes the description of our quiver with dotted triangles $Q_m = Q_m(\tau_{\rm dot})$. 

\vs

For any given quiver with dotted triangles, we assign `ratio coordinates' to each shaded triangle:
\begin{definition}[ratio coordinates associated to a quiver with dotted triangles]
\label{def:ratio_coordinates}
Suppose that we have a quiver with dotted triangles $Q$. For each vertex $v$ of $Q$, let $\Delta_v$ be the cluster coordinate at $v$ of the seed $\mathscr{A}$-torus \eqref{eq:seed_tori} associated to the underlying quiver of $Q$. 
For each shaded triangle $s$, the {\em logarithmic ratio coordinates} $q_s$, $p_s$ are defined by 
\begin{align}
q_s := \log \Delta_{a_1} - \log \Delta_{a_2}, \qquad
r_s := \log \Delta_{a_3} - \log \Delta_{a_2},
\end{align}
where $a_1,a_2,a_3$ are the vertices of $s$ in the cyclic order given by the orientations of the sides of $s$ with the dot at $a_2$. The {\em ratio coordinates $Y_s$, $Z_s$} are defined to be their exponentials:
\begin{align}
\label{eq:YZ}
Y_s := \exp(q_s) = \Delta_{a_1} / \Delta_{a_2}, \qquad
Z_s := \exp(p_s) = \Delta_{a_3} / \Delta_{a_2}.
\end{align}
Let the {\em higher Kashaev algebra} $\mathscr{K}^Q(F)$ be the quotient algebra of the algebra generated by all $Y_s$'s and $Z_s$'s with their inverses, as a subalgebra of the algebra generated by all $\Delta_v^{\pm 1}$'s over the base field $F$.
\end{definition}

\begin{definition}[ratio coordinates for higher Teichm\"uller spaces]
Let $G = {\rm SL}_m$, $m\ge 2$. Let $\wh{S}$ be a marked hyperbolic surface, and let $\tau_{\rm dot}$ be a dotted triangulation of $\wh{S}$. Then the {\em ratio coordinates for higher Teichm\"uller space $\mathscr{A}_{G,\wh{S}}$} are given by the coordinates defined in Def.\ref{def:ratio_coordinates} for the quiver with dotted triangles $Q =Q_m = Q_m(\tau_{\rm dot})$. 
\end{definition}
As the shaded triangles of $Q_m(\tau_{\rm dot})$ are labeled by $t_{r,c}$, we get a system of `coordinate functions' $\{ (q_{t_{r,c}}, p_{t_{r,c}} ) \}_{t,r,c}$, or the exponential version $\{ (Y_{t_{r,c}}, Z_{t_{r,c}} ) \}_{t,r,c}$, for the space $\mathscr{A}_{G,\wh{S}}$. However, these `coordinates' are not algebraically independent:
\begin{lemma}[linear/algebraic relations among ratio coordinates]
\label{lem:linear_algebraic_relations_among_ratio_coordinates}
Consider any closed loop in the underlying quiver of $Q=Q_m(\tau_{\rm dot})$. Suppose $a_1,a_2,\ldots,a_{\ell+1}$ is the sequence of vertices of $Q$ which this loop traverses through, so that $a_{\ell+1} = a_1$. For each $j=1,2,\ldots,\ell$, the edge connecting $a_j$ and $a_{j+1}$ is an edge of some shaded triangle $s_j$ of $Q$, and therefore $\Delta_{a_j}/\Delta_{a_{j+1}}$ coincides with one of $Y_{s_j}^{\pm 1}$, $Z_{s_j}^{\pm 1}$, or $(Y_{s_j}/Z_{s_j})^{\pm 1}$. Thus the identity $\prod_{j=1}^n (\Delta_{a_j}/\Delta_{a_{j+1}}) = 1$ becomes an algebraic equation satisfied by $Y_s$'s and $Z_s$'s. Its logarithmic version $\sum_{j=1}^n (\log \Delta_{a_j} - \log \Delta_{a_{j+1}})=0$ becomes a linear equation satisfied by $q_s$'s and $p_s$'s. \qed
\end{lemma}

\vs

It is straightforward to verify the following, which suggests a `canonical' quantization of these coordinate systems, as we shall see in \S\ref{sec:new_quantization}:
\begin{lemma}
The $2$-form $\Omega_{{\rm SL}_m,\wh{S}}$ \eqref{eq:Omega_SL_m} becomes `diagonal' in terms of the logarithmic ratio coordinates:
\begin{align}
\label{eq:2-form_diagonal_higher}
\Omega_{{\rm SL}_m,\wh{S}} = - 2 \sum_{t,r,c} d p_{t_{r,c}} \wedge d q_{t_{r,c}}.
\end{align}
\end{lemma}

\subsection{Coordinate change maps for ratio coordinates}

To summarize, a dotted triangulation $\tau_{\rm dot} = (\tau,D,L)$ of $\wh{S}$ induces a quiver with dotted triangles $Q_m = Q_m(\tau_{\rm dot})$, which in turn gives rise to a system of functions $\{ (q_{t_{r,c}}, p_{t_{r,c}} ) \}_{t,r,c}$ or $\{ (Y_{t_{r,c}}, Z_{t_{r,c}} ) \}_{t,r,c}$ on the moduli space $\mathscr{A}_{{\rm SL}_m,\wh{S}}$ which we call `(ratio) coordinate functions', where $t$ runs through the ideal triangles of $\tau$, and $(r,c)$ runs in $\mcal{S}_m$ \eqref{eq:mcal_S_m}. 

\vs

For another choice $\tau'_{\rm dot}$, one obtains another collection of ratio coordinates $\{ (Y'_{t'_{r,c}}, Z'_{t'_{r,c}} ) \}_{t',r,c}$, and  we would like to express these new coordinates in terms of the old ones $Y_{t_{r,c}}$, $Z_{t_{r,c}}$. In order to do this, we will investigate the change of the quiver with dotted triangles $Q_m(\tau_{\rm dot})\leadsto Q_m(\tau'_{\rm dot})$ induced by the change of the dotted triangulation $\tau_{\rm dot} \leadsto \tau'_{\rm dot}$. We study the `elementary' changes of quivers with dotted triangles in the present subsection, together with the corresponding coordinate change formulas. In \S\ref{subsec:Kashaev_groupoid} we realize the change $Q_m(\tau_{\rm dot}) \leadsto Q_m(\tau'_{\rm dot})$ as a composition of elementary changes of quivers with dotted triangles, and thus express $Y'_{t_{r,c}'}$, $Z'_{t_{r,c}'}$ in terms of $Y_{t_{r,c}}$, $Z_{t_{r,c}}$.

\begin{definition}[elementary transformations of quivers with dotted triangles]
\label{def:elementary_transformations_of_qdt}
We describe three types of {\em elementary transformations} of quivers with dotted triangles. 
Let $Q$, $Q'$ be quivers with dotted triangles having same index set $I$. 
We say that $Q'$ is obtained by applying the respective elementary transformation to $Q$, if one of the following holds.

\begin{enumerate}
\item {\em mutation $T_{rs}$, for $r,s\in I$ :} \\
The shaded triangles of $Q$ labeled by $r$ and $s$ must share exactly one vertex. There must be two incoming arrows to and two outgoing arrows from this common vertex, all four arrows being distinct. Let $a,b,c,d,e$ denote the vertices of the shaded triangles $r$ and $s$ of $Q$, so that $a,b,c$ are the vertices of $r$ and $c,d,e$ are those of $s$, where $a$ and $d$ are the endpoints of the outgoing arrows from the common vertex $c$. The dot of $r$ must be at $b$, and that of $s$ at $c$. The mutation $T_{rs}$ is defined only for such $Q$.

\vs

The underlying quiver of $Q'$ is obtained by applying to that of $Q$ the quiver mutation in Lem.\ref{lem:quiver_mutation} at the vertex $c$; let us call $a,b,c,d,e$ the vertices of $Q'$ corresponding to the ones of $Q$ with the same names respectively. The shaded triangles of $Q$ other than $r$ and $s$ and the labeling rule for them stay the same for $Q'$, although sides of some shaded triangles may change from an edge to an invisible edge. We replace the shaded triangles $r$ and $s$ by new ones, also labeled by $r$ and $s$,  given as follows. The new triangle $r$ of $Q'$ is located at the vertices $a,c,e$, and the new $s$ at $b,c,d$, while the dot of new $r$ is at $c$, and that of new $s$ at $b$.

\begin{figure}[htbp!]
\begin{subfigure}[b]{0.36\textwidth}
\hspace{-3mm} \includegraphics[width=60mm]{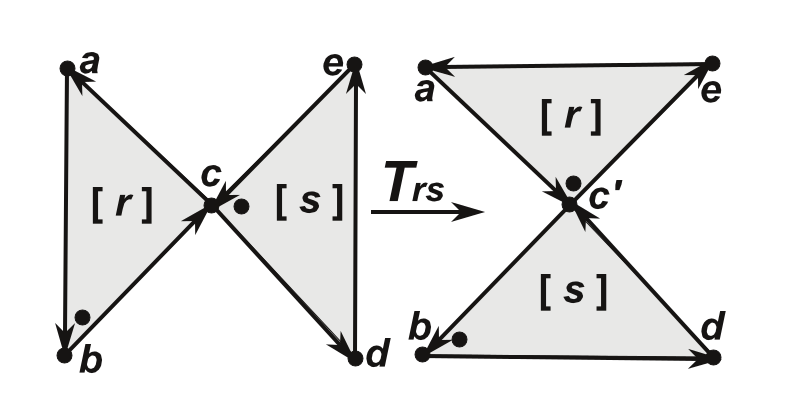}
\caption{An example of $T_{rs}$}
\label{fig:action_of_T_rs}
\end{subfigure}
\hfill
\begin{subfigure}[b]{0.32\textwidth}
\hspace{-3mm} \includegraphics[width=50mm]{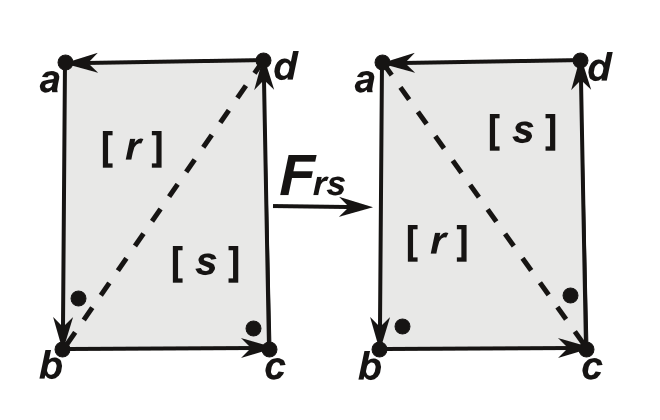}
\caption{An example of $F_{rs}$}
\label{fig:action_of_F_rs}
\end{subfigure}   
\hfill
\begin{subfigure}[b]{0.30\textwidth}
\includegraphics[width=40mm]{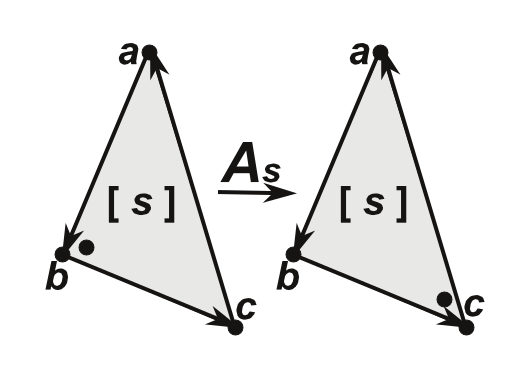}
\caption{An example of $A_s$}
\label{fig:action_of_A_s}
\end{subfigure}   
\hspace{-2mm}
\vspace{-2mm}
\caption{Some elementary transformations of quivers with dotted triangles}
\label{fig:elementary_transformations_of_Qdt}
\end{figure}

\item {\em invisible flip $F_{rs}$, for $r,s\in I$ :} \\
The shaded triangles $r$ and $s$ of $Q$ must be both defective, sharing an invisible edge. The four sides of the quadrilateral formed by $r$ and $s$ must be oriented cyclically. Let $a,b,c,d$ be the vertices of this quadrilateral, so that $a,b,d$ are the vertices of $r$ and $b,c,d$ those of $s$, while the orientations of the sides are $a\to b\to c\to d \to a$. The dot of $r$ must be at $b$, that of $s$ at $c$. 

\vs

The underlying quiver of $Q'$ is same as that of $Q$. The shaded triangles of $Q$ other than $r$ and $s$ and the labeling rule for them stay the same for $Q'$. We replace the triangles $r$ and $s$ in $Q$ by the ones of same name in $Q'$ as follows; the triangle $r$ of $Q'$ is at the vertices $a,b,c$ with the dot at $b$, and the triangle $s$ of $Q'$ is at $c,d,a$ with the dot at $c$. This can be thought of as rotating the invisible edge of the quadrilateral formed by $r$ and $s$ of $Q$ counterclockwise to obtain $Q'$.

\item {\em dot change $A_s$, for $s\in I$ : } \\
The only difference between $Q'$ and $Q$ is the position of the dot $\bullet$ of the shaded triangle $s$. The dot of triangle $s$ in $Q'$ is at the next corner to the one for the dot of $s$ in $Q$, with respect to the ordering given by the orientation of the sides.

\item {\em label change $P_\sigma$, for a permutation $\sigma$ of $I$ : } \\
The only difference between $Q'$ and $Q$ are the labels for shaded triangles. For each $s\in I$, the shaded triangle of $Q$ labeled by $s$ is labeled by $\sigma(s)$ in $Q'$.
\end{enumerate}
\end{definition}

In Def.\ref{def:ratio_coordinates}, to each quiver with dotted triangles $Q$ we associated ratio coordinates, which are expressed in terms of the coordinate functions $\Delta_v$ attached to vertices $v$ of the underlying quiver of $Q$. For each elementary transformation of quivers with dotted triangles applied to $Q$, consider the induced transformation of the underlying quiver, and correspondingly induced coordinate functions $\Delta'_{v'}$ at the vertices of the new quiver. The only nontrivial case is the mutation $T_{rs}$, in which case the induced transformation of the underlying quivers is the mutation at $c$. These new coordinate functions $\Delta'_{v'}$ are used for the new ratio coordinates. 

\vs

We are able to express the ratio coordinates for the new quiver with dotted triangles in terms of the old ratio coordinates:

\begin{lemma}[coordinate change formulas for elementary transformations of quivers with dotted triangles]
\label{lem:coordinate_change_for_elementary_transformation_of_Qdt}
Suppose that a quiver with dotted triangles $Q'$ is obtained from $Q$ by applying an elementary transformation; let $I$ be the index set of the shaded triangles of $Q$ and $Q'$. The ratio coordinate functions $Y'_s$, $Z'_s$ for $Q'$ are expressed in terms of $Y_s$, $Z_s$ for $Q$ as follows.
\begin{enumerate}
\item {\em when $Q' = T_{rs}.Q$ \, :}
\begin{align}
\label{eq:T_rs_coordinate_change_formula}
\left\{ {\renewcommand{\arraystretch}{1.2}
\begin{array}{l}
Y'_r = (Z_s + Y_r^{-1} Y_s)^{-1}, \\
Z'_r = (Y_r Z_r^{-1} Y_s^{-1} Z_s + Z_r^{-1})^{-1},
\end{array}} \right.
\qquad
{\renewcommand{\arraystretch}{1.2}
\begin{array}{l}
Y'_s = Y_r Z_s + Y_s, \\
Z'_s = Z_r Z_s.
\end{array}}
\end{align}

\item {\em when $Q' = F_{rs}.Q$ \, :}
\begin{align}
\label{eq:F_rs_coordinate_change_formula}
\left\{ {\renewcommand{\arraystretch}{1.2}
\begin{array}{l}
Y'_r = Y_r, \\
Z'_r = Z_r Z_s^{-1},
\end{array}} \right.
\qquad
{\renewcommand{\arraystretch}{1.2}
\begin{array}{l}
Y'_s = Y_s Y_r, \\
Z'_s = Z_s.
\end{array}}
\end{align}

\item {\em when $Q' = A_s. Q$ \, : }
\begin{align}
\label{eq:A_s_coordinate_change_formula}
Y'_s = Z_s^{-1}, \qquad
Z'_s = Y_s Z_s^{-1}.
\end{align}

\item {\em when $Q' = P_\sigma.Q$ \, : }
\begin{align*}
Y'_s = Y_{\sigma^{-1}(s)}, \qquad Z'_s = Z_{\sigma^{-1}(s)}, \qquad \forall s \in I.
\end{align*}
\end{enumerate}
In each case, the ratio coordinates for the shaded triangles of $Q'$ not mentioned in the above coordinate change formulas are equal to the corresponding ones of $Q$.
\end{lemma}

\begin{proof}
(1) Denote the vertices of the shaded triangles $r$ and $s$ of $Q$ by $a,b,c,d,e$, as described in Def.\ref{def:elementary_transformations_of_qdt}(1), and denote the corresponding vertices of $Q'$ by $a,b,c',d,e$; this is because the only vertex coordinate function of the underlying quiver of $Q$ that changes under the mutation corresponding to $T_{rs}: Q \leadsto Q'$ is the one at $c$. Then, from the definition \eqref{eq:YZ} of the ratio coordinates we have
\begin{align*}
\begin{array}{l}
Y_r = \Delta_a/\Delta_b, \\
Z_r = \Delta_c/\Delta_b,
\end{array}
\begin{array}{l}
Y_s = \Delta_e/\Delta_c, \\
Z_s = \Delta_d/\Delta_c,
\end{array}
\begin{array}{l}
Y'_r = \Delta_a/\Delta_{c'}, \\
Z'_r = \Delta_e/\Delta_{c'},
\end{array}
\begin{array}{l}
Y'_s = \Delta_{c'}/\Delta_b, \\
Z'_s = \Delta_d/\Delta_b.
\end{array}
\end{align*}
From the mutation formula \eqref{eq:A_mutation_formula} for the $A$-coordinates we have
$$
\Delta_c \Delta_{c'} = \Delta_a \Delta_d + \Delta_b \Delta_e.
$$
Hence we obtain \eqref{eq:T_rs_coordinate_change_formula} as follows:
\begin{align*}
Y'_r & = \frac{ \Delta_a \Delta_c}{ \Delta_a \Delta_d + \Delta_b \Delta_e} = \left( \frac{\Delta_d}{\Delta_c} + \frac{\Delta_b}{\Delta_a} \frac{\Delta_e}{\Delta_c} \right)^{-1} = (Z_s + Y_r^{-1} Y_s)^{-1}, \\
Z'_r & = \frac{ \Delta_e \Delta_c }{\Delta_a \Delta_d + \Delta_b \Delta_e} = \left( \frac{\Delta_a \Delta_d }{\Delta_e \Delta_c} + \frac{\Delta_b}{\Delta_c} \right)^{-1} = \left( \frac{\Delta_a}{\Delta_b} \frac{\Delta_b}{\Delta_c} \frac{\Delta_c}{\Delta_e}  \frac{\Delta_d}{\Delta_c} + \frac{\Delta_b}{\Delta_c} \right)^{-1} \\
& = (Y_r Z_r^{-1} Y_s^{-1} Z_s + Z_r^{-1})^{-1}, \\
Y'_s & = \frac{ \Delta_a \Delta_d + \Delta_b \Delta_e }{\Delta_c \Delta_b} = \frac{\Delta_a}{\Delta_b} \frac{\Delta_d}{\Delta_c} + \frac{\Delta_e}{\Delta_c} = Y_r Z_s + Y_s, \qquad
Z'_s = \frac{ \Delta_c}{\Delta_b} \frac{\Delta_d}{\Delta_c} = Z_r Z_s.
\end{align*}

\vs

(2) Denote the vertices of the shaded triangles $r$ and $s$ of $Q$ and $Q'$ by $a,b,c,d$, as described in Def.\ref{def:elementary_transformations_of_qdt}(2). Then we have
$$
\begin{array}{l}
Y_r = \Delta_a/\Delta_b, \\
Z_r = \Delta_d/\Delta_b,
\end{array}
\begin{array}{l}
Y_s = \Delta_b/\Delta_c, \\
Z_s = \Delta_d/\Delta_c,
\end{array}
\begin{array}{l}
Y'_r = \Delta_a/\Delta_b, \\
Z'_r = \Delta_c/\Delta_b,
\end{array}
\begin{array}{l}
Y'_s = \Delta_a/\Delta_c, \\
Z'_s = \Delta_d/\Delta_c.
\end{array}
$$
It is easy to verify \eqref{eq:F_rs_coordinate_change_formula}.

\vs

(3) Denote the vertices of the shaded triangle $s$ of $Q$ and $Q'$ by $a,b,c$, so that the orientations of the sides of $s$ are $a\to b\to c\to a$, and that the dot of $s$ is at $b$; then the dot of $s'$ must be at $c$. We thus have
\begin{align*}
\begin{array}{l}
Y_s = \Delta_a/\Delta_b, \\
Z_s = \Delta_c/\Delta_b,
\end{array}
\begin{array}{l}
Y'_s = \Delta_b/\Delta_c, \\
Z'_s = \Delta_a/\Delta_c,
\end{array}
\end{align*}
from which \eqref{eq:A_s_coordinate_change_formula} is immediate. We omit the proof of (4) as it is easy to see.
\end{proof}

\subsection{Coordinate change for a transformation of dotted triangulations}
\label{subsec:Kashaev_groupoid}

In the present subsection we study the coordinate change maps for the ratio coordinates of $\mathscr{A}_{{\rm SL}_m,\wh{S}}$, associated to each change of dotted triangulations of $\wh{S}$. We describe our result in terms of the groupoid formed by all possible transformations of dotted triangulations. We first define:

\begin{definition}[\cite{Penner2}: Ptolemy groupoid]
\label{def:Ptolemy_groupoid}
Let $\wh{S}$ be a marked hyperbolic surface. Let the {\em Ptolemy groupoid $Pt(\wh{S})$ of $\wh{S}$} be the category whose objects are all possible ideal triangulations of $\wh{S}$, where from any object $\tau$ to any object $\tau'$ there is a unique morphism denoted by $[\tau,\tau']$. The composition rule for morphisms is given by
$$
[\tau',\tau''] \circ [\tau, \tau'] = [\tau, \tau''].
$$
\end{definition}
It is a classical result that the Ptolemy groupoid is generated by `flips' (see Prop.\ref{prop:Whitehead}):
\begin{definition}[e.g. \cite{PennerBook}: flip]
\label{def:flip}
Suppose $e$ is an edge of an ideal triangulation $\tau$ that separates distinct triangles. Then $e$ is one diagonal of an ideal quadrilateral formed by two ideal triangles, and we may replace $e$ by the other diagonal $e'$ of this quadrilateral to produce another ideal triangulation $\tau'$. We say that $\tau'$ arises from $\tau$ by a {\em flip} or a {\em Whitehead move} $W_e$ along $e$. 
\end{definition}

Now we consider the groupoid formed by Kashaev's dotted triangulations.
\begin{definition}[see e.g. \cite{Ki12}: Kashaev groupoid]
\label{def:Kashaev_groupoid}
Let the {\em Kashaev groupoid $\mcal{K}(\wh{S})$ of $\wh{S}$} be the category whose objects are all possible dotted triangulations (Def.\ref{def:dotted_triangulation}) of $\wh{S}$, where from any object $\tau_{\rm dot}$ to any object $\tau'_{\rm dot}$ there is a unique morphism $[\tau_{\rm dot}, \tau'_{\rm dot}]$, with the composition rule for morphisms is given by
$$
[\tau'_{\rm dot},\tau''_{\rm dot}] \circ [\tau_{\rm dot}, \tau'_{\rm dot}] = [\tau_{\rm dot}, \tau''_{\rm dot}].
$$
\end{definition}
The Kashaev groupoid $\mcal{K}(\wh{S})$ admits an easy `presentation' by generators and relations. We start by defining the generating morphisms as follows. In the meantime, we keep in mind that each ideal triangle of an ideal triangulation of $\wh{S}$ inherits an orientation from that of $\wh{S}$, thus there is a counterclockwise ordering of the three vertices of it.

\begin{definition}[see e.g. \cite{Kash00}: elementary moves of the Kashaev groupoid $\mcal{K}(\wh{S})$]
Let $I$ be a fixed index set for the set of all ideal triangles for any ideal triangulation of $\wh{S}$. We define three types of {\em elementary moves} in $\mcal{K}(\wh{S})$. We call a morphism $[(\tau,D,L), \, (\tau',D',L')]$ of $\mcal{K}(\wh{S})$ an {\em elementary morphism} if it falls into one of the following, and in such a case we say that $(\tau',D',L')$ is obtained by {\em applying} the relevant elementary move to $(\tau,D,L)$:
\begin{enumerate}
\item {\em dot change $\til{A}_t$ for $t\in I$}: \\
We have $\tau'=\tau$ and $L'=L$, where $D'$ differs from $D$ only on the ideal triangle labeled by $t$ by the labeling bijection $L$. We move the dot of this triangle to the next one in the counterclockwise sense.

\begin{figure}[htbp!]
\hfill 
\begin{subfigure}[b]{0.40\textwidth}
\hspace{-3mm} \includegraphics[width=50mm]{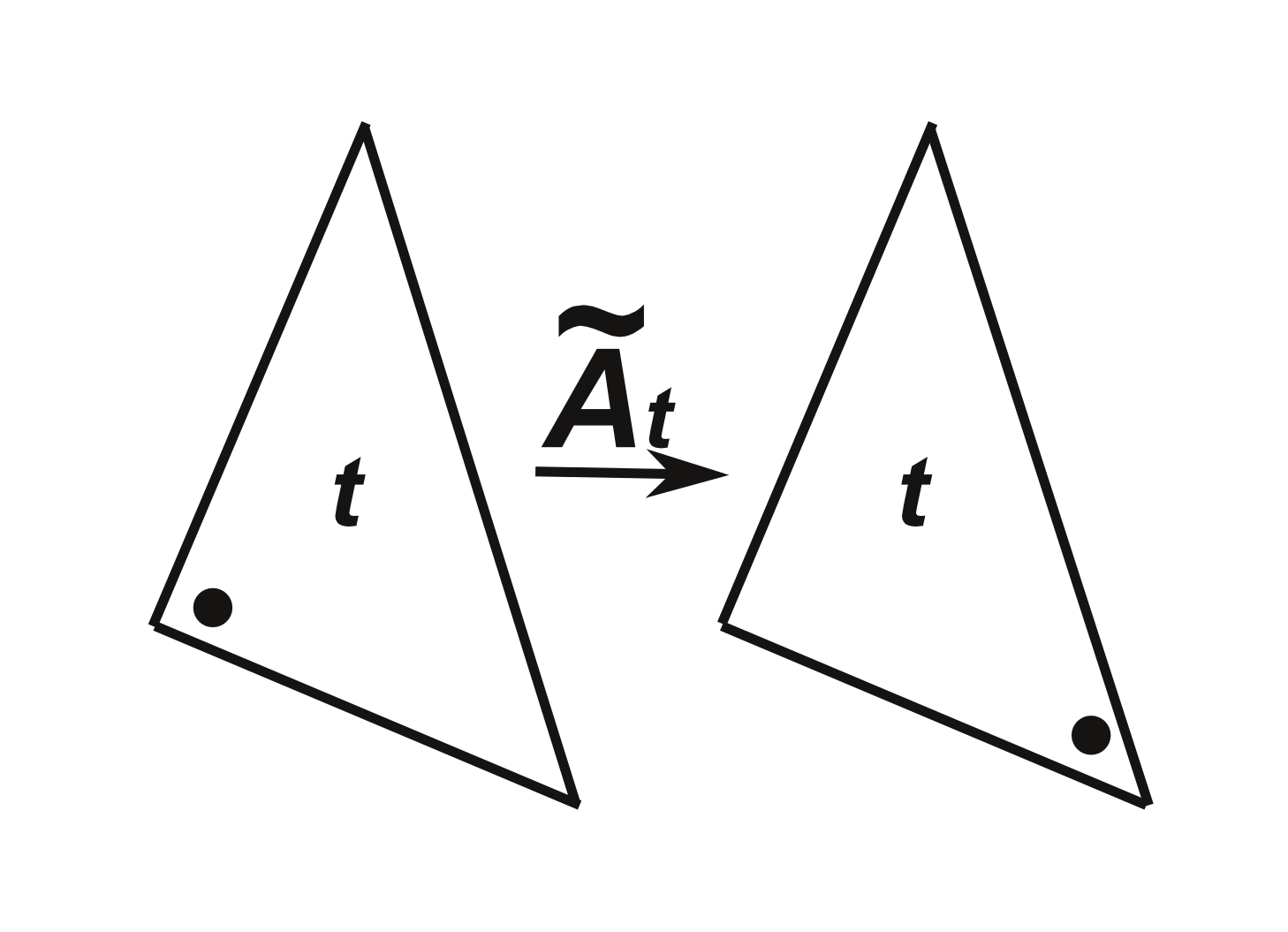}
\caption{The action of $\til{A}_t$}
\label{fig:action_of_til_A_t}
\end{subfigure}
\hfill
\begin{subfigure}[b]{0.40\textwidth}
\includegraphics[width=55mm]{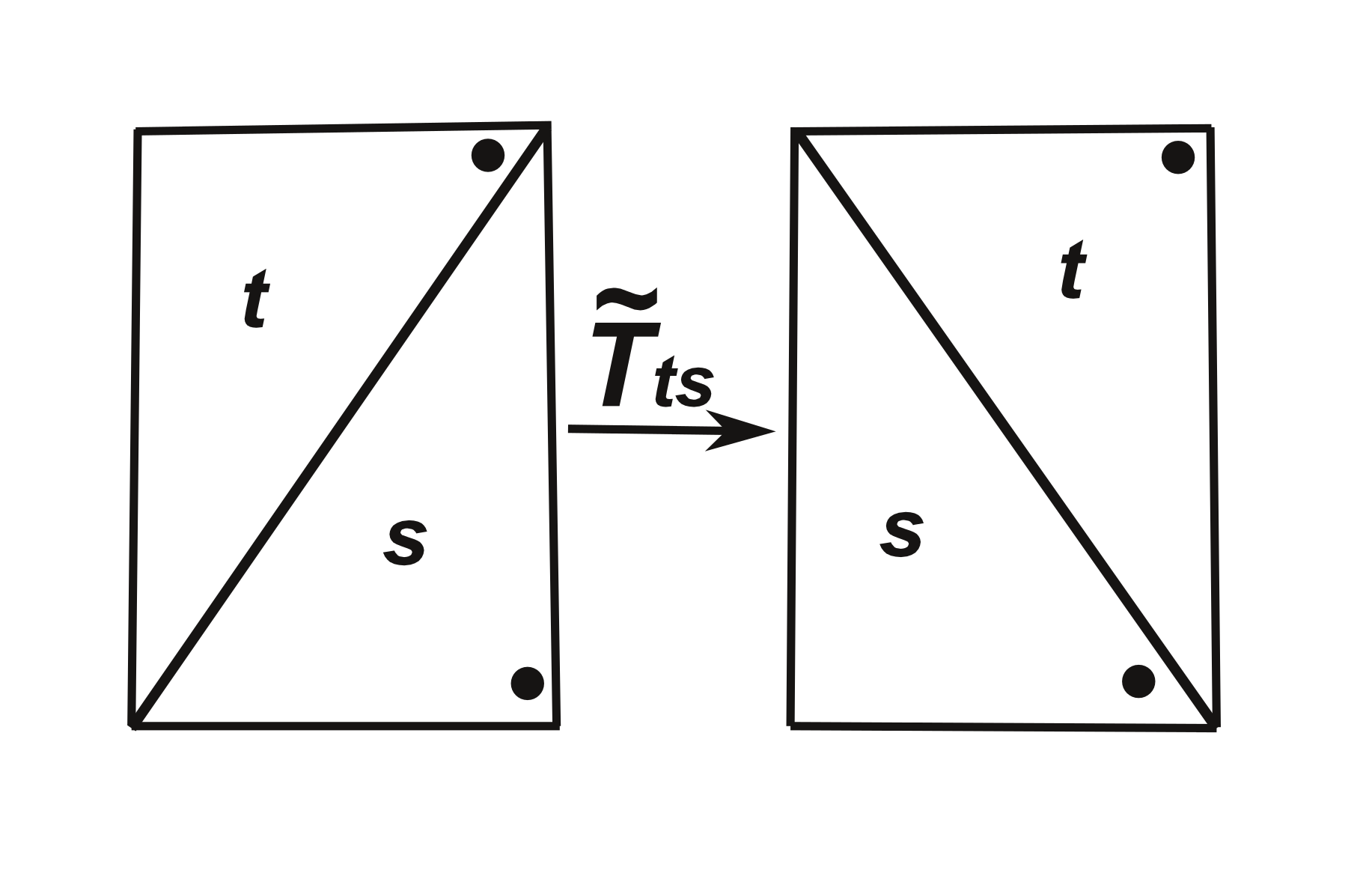}
\caption{The action of $\til{T}_{ts}$}
\label{fig:action_of_til_T_ts}
\end{subfigure}   
\hfill
\vspace{-2mm}
\caption{Some elementary moves of Kashaev groupoid $\mcal{K}(\wh{S})$}
\label{fig:elementary_moves_of_Kashaev_groupoid}
\end{figure}

\item {\em enhanced flip $\til{T}_{ts}$ for $t,s\in I$ :}  \\
This move can be defined only if the two triangles labeled by $t,s$ share a common side $e$, and their dots should be configured exactly as in Fig.\ref{fig:action_of_til_T_ts}, with respect to this side $e$; the dot of $s$ should be at the corner far from $e$, and that of $t$ at the clockwise next corner to the one of $t$ far from $e$. Then $\tau'$ is obtained from $\tau$ by  flipping along $e$ (Def.\ref{def:flip}). The rules $D'$ and $L'$ are defined as in Fig.\ref{fig:action_of_til_T_ts} for the two new ideal triangles resulting from this flip, while they are same as $D$ and $L$ for all other triangles. One can think of this as rotating the diagonal of the quadrilateral formed by these triangles clockwise by 90 degrees, while the dots and triangle labels are `floating'.

\item {\em triangle-label permutation $\til{P}_\sigma$, for a permutation $\sigma$ of $I$ :} \\
We have $\tau' = \tau$, $D' = D$, while $L' = \sigma \circ L$.
\end{enumerate}
\end{definition}
One can think of applying a word $g_\ell \cdots g_2 g_1$ in elementary moves and their inverses to some $(\tau, D,L)$. We read the word from right, that is, first apply $g_1$ and last apply $g_\ell$.
\begin{definition}
\label{def:support}
For a word in elementary moves and their inverses, we say $(\tau,D,L)$ {\em supports} the word if the word can be applied to $(\tau,D,L)$.
\end{definition}

By the requirement that $\mcal{K}(\wh{S})$ has only one morphism from any object to another, the elementary moves satisfy some algebraic relations:
\begin{lemma}[See e.g. \cite{Kash00}, \cite{Ki12}]
\label{lem:relations_of_elementary_moves_of_Kashaev_groupoid}
The elementary moves of $\mcal{K}(\wh{S})$ satisfy the following algebraic relations: for mutually distinct $t,s,u,v \in I$ and for any  permutations $\sigma_1,\sigma_2,\sigma$ of $I$,
\begin{align}
\label{eq:rho_cubed}
{\rm [order\mbox{-}three]} \quad & \til{A}_t^3 = {\rm identity}, \\
{\rm [pentagon]} \quad & \til{T}_{ts} \til{T}_{tu} \til{T}_{su}  =  \til{T}_{su} \til{T}_{ts}, 
\\
{\rm [inversion]} \quad & 
\til{T}_{ts} \til{A}_t \til{T}_{st} = \til{A}_t \til{A}_s \til{P}_{(t \, s)}, 
\\
\label{eq:consistency}
{\rm [consistency]} \quad & \til{A}_t \til{T}_{ts} \til{A}_s = \til{A}_s \til{T}_{st} \til{A}_t, 
\\
{\rm [permutation]} \quad & \til{P}_{\rm id} = {\rm identity}, \quad \til{P}_{\sigma_1} \til{P}_{\sigma_2} = \til{P}_{\sigma_1 \circ \sigma_2}, \\
{\rm [index~change]} \quad & \til{P}_\sigma \til{A}_t = \til{A}_{\sigma(t)} \til{P}_\sigma, \quad
\til{P}_\sigma \til{T}_{ts} = \til{T}_{\sigma(t) \, \sigma(s)} \til{P}_\sigma, 
\\
\label{eq:commutativity}
{\rm [commutativity]} \quad & \til{T}_{ts} \til{T}_{uv} = \til{T}_{uv} \til{T}_{ts}, \quad
\til{T}_{ts} \til{A}_u = \til{A}_u \til{T}_{ts}, \quad
\til{A}_t \til{A}_s = \til{A}_s \til{A}_t,
\end{align}
where each equality means that if $(\tau,D,L)$ supports the left-hand side, then it supports the right-hand side too, and the results of the application of both sides to $(\tau,D,L)$ are the same.
\end{lemma}

\begin{remark}
\label{rem:Kashaev_notation}
Kashaev uses $\rho_t := \til{A}_t$, $\omega_{ts} :=  \til{A}_t \til{A}_s^{-1} \til{T}_{ts} \til{A}_s \til{A}_t^{-1}$, $P_\sigma := \til{P}_\sigma$ as the elementary moves, so that the relations become $\rho_t^3 = {\rm identity}$, $\omega_{su} \omega_{tu} \omega_{ts}  = \omega_{ts} \omega_{su}$, $\omega_{ts} \rho_s \omega_{st} = \rho_t \rho_s P_{(t \, s)}$, $\rho_s \omega_{ts} \rho_t = \rho_t \omega_{st} \rho_t$, and so on. See \S\ref{subsec:proof_of_completeness}.
\end{remark}

As the above relations are easily verified by pictures, their proof is omitted. What is not so obvious is that the above generators and the relations \eqref{eq:rho_cubed}--\eqref{eq:commutativity} provide a full `presentation' of the groupoid $\mcal{K}(\wh{S})$. 
\begin{proposition}
\label{prop:Kashaev_full_presentation}
Any morphism of $\mcal{K}(\wh{S})$ is composition of finite number of elementary morphisms, and any algebraic relations among elementary moves are consequences of the ones in Lem.\ref{lem:relations_of_elementary_moves_of_Kashaev_groupoid}.
\end{proposition}
The above result has been assumed and used in the literature, but a complete proof has never been published. We present a proof of it in \S\ref{subsec:proof_of_completeness}, since it is crucially used in \S\ref{subsec:consistency_condition_for_3} which is the main proof of the present paper.

\vs

As a dotted triangulation induces a quiver with dotted triangles, each elementary morphism of $\mcal{K}(\wh{S})$ induces a change of quivers with dotted triangles, and we express it as a composition of elementary transformations of quivers with dotted triangles. The only nontrivial case is the enhanced flip $\til{T}_{ts}$, and we follow Fock-Goncharov's realization \cite{FG06} of a flip of an ideal triangulation in terms of a sequence of quiver mutations applied on the quiver associated to the $m$-triangulation. Namely, we first apply mutation at all vertices on the common side of the two ideal triangles $t$ and $s$, then at all vertices on the two parallel lines of distance $1$ from the common side, then at those on the lines of distance $2$ from the common side, etc.
\begin{lemma}[elementary move of Kashaev groupoid as sequence of elementary transformations of quivers with dotted triangles]
\label{lem:functor_from_Kashaev_groupoid_to_Qdt}
Suppose that a dotted triangulation $\tau'_{\rm dot} = (\tau',D',L')$ is obtained from $\tau_{\rm dot} = (\tau,D,L)$ by applying an elementary move, and that $I$ is the index set of the ideal triangles of $\tau$ and $\tau'$. Then the corresponding transformation of quivers with dotted triangles $Q_m(\tau_{\rm dot}) \leadsto Q_m(\tau'_{\rm dot})$ is expressed as a composition of elementary transformations of quivers with dotted triangles as follows.
\begin{enumerate}
\item {\em dot change $\til{A}_t$, for $t\in I$ :}
\begin{align}
\label{eq:til_A_t_as_composition}
\til{A}_t \longmapsto P_{\sigma^A_t} \prod_{(r,c) \in \mcal{S}_m} A_{t_{r,c}},
\end{align}
where $\sigma^A_t$ is the permutation acting on the labels $\{ t_{r,c} : (r,c) \in \mcal{S}_m \}$ of the shaded triangles in $t$, associated to the dot change $\til{A}_t$, where $\mcal{S}_m$ is as in \eqref{eq:mcal_S_m}; examples of the action of $\sigma^A_t$ are
$$
t_{1,1} \mapsto t_{m-1,m-1}, ~
t_{2,1} \mapsto t_{m-2,m-2}, ~
\cdots ~,
t_{m-1,1} \mapsto t_{1,1}, ~\cdots~,
t_{m-1,m-1}\mapsto t_{m-1,1}.
$$

\vs

\item {\em enhanced flip $\til{T}_{ts}$, for $t,s\in I$ :}
\begin{align}
\label{eq:til_T_ts_as_composition}
\til{T}_{ts} \longmapsto \prod_{i=1}^{m-1} M_{m-i} C_{m-i-1}
= (M_{m-1} C_{m-2}) (M_{m-2} C_{m-3})\cdots (M_2 C_1) (M_1 C_0),
\end{align}
where
\begin{align}
\left\{ \begin{array}{rl}
C_0 & = \displaystyle {\rm id}, \\
M_1 & = \displaystyle \prod_{j=1}^{m-1} T_{t_{j,j} s_{m-1,j}} = T_{t_{1,1} s_{m-1,1}} T_{t_{2,2} s_{m-1,2}} \cdots T_{t_{m-1,m-1} s_{m-1,m-1}}, \\
C_1 & = \displaystyle \prod_{j=1}^{m-2} F_{s_{m-1,j} t_{j+1,j+1}} = F_{s_{m-1,1} t_{2,2}} F_{s_{m-1,2} t_{3,3}} \cdots F_{s_{m-1,m-2} t_{m-1,m-1} },
\end{array} \right.
\end{align}
and
\begin{align}
\left\{ \begin{array}{rl}
M_\ell & = \displaystyle \left( \prod_{j=1}^{m-\ell} T_{t_{\ell-1+j,j} s_{m-1,j}}\right) \left(\prod_{j=1}^{m-\ell} T_{t_{j+\ell-1,j+\ell-1} s_{m-\ell,j}}\right), \quad \mbox{for} \quad 2\le \ell \le m-1, \\
C_\ell & = \displaystyle \left( \prod_{j=1}^{m-\ell-1} F_{s_{m-1,j} t_{\ell+j,j+1}}\right) \left(\prod_{j=1}^{m-\ell-1} T_{s_{m-\ell,j} t_{j+\ell,j+\ell}}\right), \quad \mbox{for} \quad 2\le \ell \le m-2.
\end{array} \right.
\end{align}

\vs

\item {\em label permutation $\til{P}_\sigma$, for a permutation $\sigma$ of $I$ :}
\begin{align}
\label{eq:til_P_sigma_as_composition}
\til{P}_\sigma \longmapsto \prod_{(r,c) \in \mcal{S}_m} P_{\sigma_{r,c}},
\end{align}
where for each $(r,s) \in \mcal{S}_m$, \, $\sigma_{r,c}$ is the permutation acting on the set $\{ t_{r,c} \, : \, t \in I \}$ by $\sigma_{r,c}(t_{r,c}) = (\sigma(t))_{r,c}$ for all $t\in I$.
\end{enumerate}
Here, each word in the elementary transformations of quivers with dotted triangulation is to be read from the right. The product orders in \eqref{eq:til_A_t_as_composition} and \eqref{eq:til_P_sigma_as_composition} do not matter, as all the factors commute with one another in each product.
\end{lemma}

\begin{proof}
In the cases (1) and (3), it is easy to verify that the asserted compositions of elementary transformations of quivers with dotted triangles indeed give the relevant transformations of quivers with dotted triangles $Q_m(\tau_{\rm dot}) \leadsto Q_m(\tau'_{\rm dot})$. In case (2), we apply the sequence of elementary transformations of quivers with dotted triangles as written in the RHS of \eqref{eq:til_T_ts_as_composition} to $Q_m(\tau_{\rm dot})$, step by step. We ignore $C_0$, and $M_1$ is the sequence of mutations at the vertices on the common side of the ideal triangles $t$ and $s$. Using the rule described in Def.\ref{def:elementary_transformations_of_qdt}(1), one can figure out that the quiver with dotted triangle becomes the second one in Fig.\ref{fig:til_T_ts_realization} after applying $M_1$. We see $m-1$ quadrilaterals formed by defective shaded triangles, along the common side of $t$ and $s$; the next step $C_1$ is to apply the invisible flips to all of them. Using the description in Def.\ref{def:elementary_transformations_of_qdt}(2), we see that we now land in the third one in Fig.\ref{fig:til_T_ts_realization}. Then we apply mutations at all the vertices on the `lines at distance $1$' from the common side of $t$ and $s$, which is realized by $M_2$; this puts us in the fourth in Fig.\ref{fig:til_T_ts_realization}, and so on. Details are easily verified by pictures.
\end{proof}

\begin{figure}[htbp!]
\includegraphics[width=160mm]{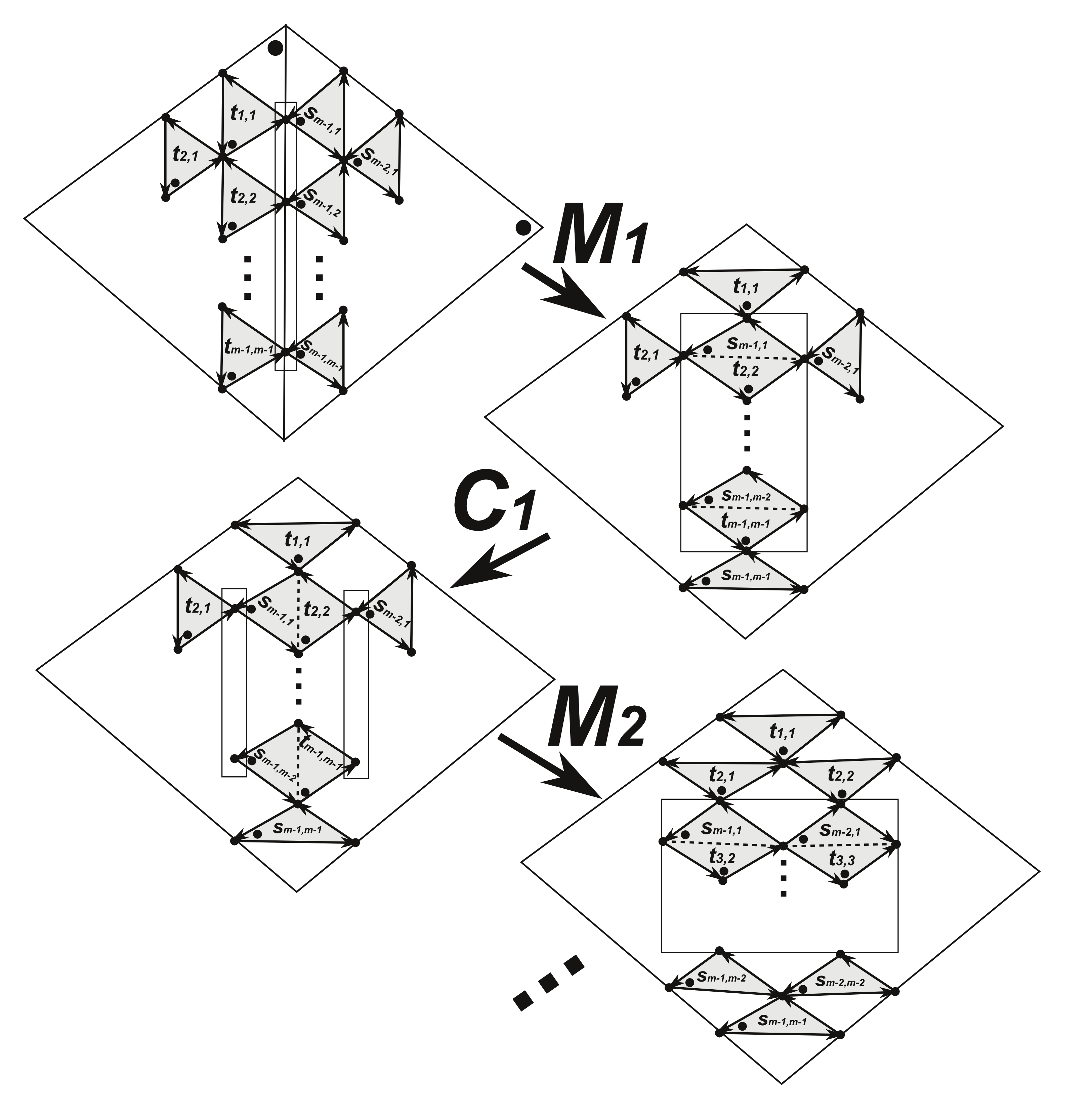}
\vspace{-7mm}
\caption{realization of $\til{T}_{ts}$ using transformations of quivers with dotted triangles}
\label{fig:til_T_ts_realization}
\end{figure}

Let us encode this result in the language of groupoids.
\begin{definition}[groupoid of admissible quivers with dotted triangles]
A quiver with dotted triangles is {\em $(m,\wh{S})$-admissible} if it can be obtained by applying a sequence of finite number of elementary transformations of quivers with dotted triangles to $Q_m(\tau_{\rm dot})$ for some dotted triangulation $\tau_{\rm dot}$ of $\wh{S}$. Let ${\rm Qdt}_m(\wh{S})$ be the groupoid whose objects are all $(m,\wh{S})$-admissible quivers with dotted triangles, whose morphisms are generated by {\em elementary morphisms}, corresponding to elementary transformations of quivers with dotted triangles.
\end{definition}
\begin{remark}
In this way, ${\rm Qdt}_m(\wh{S})$ has too many morphisms, so it is wise to identify some of them, by imposing certain algebraic relations among elementary morphisms. In order to find the `correct' generating set of relations, we must first know the generating set of relations for all the algebraic relations satisfied by the elementary transformations of quivers with dotted triangles. We do not do this here.
\end{remark}

To each dotted triangulation $\tau_{\rm dot}$ we know how to assign a quiver with dotted triangles $Q_m(\tau_{\rm dot})$, and Lem.\ref{lem:functor_from_Kashaev_groupoid_to_Qdt} tells us how to assign to each elementary morphism of $\mcal{K}(\wh{S})$ a morphism of ${\rm Qdt}_m(\wh{S})$. Hence we get a map
\begin{align}
\label{eq:mcal_K_S_to_Qdt}
\mcal{K}(\wh{S}) \to {\rm Qdt}_m(\wh{S}),
\end{align}
defined on objects and on elementary morphisms. In the meantime, for each $\tau_{\rm dot} \in \mcal{K}(\wh{S})$, Def.\ref{def:ratio_coordinates} assigns the higher Kashaev algebra $\mathscr{K}^{Q_m(\tau_{\rm dot})} (F)$ generated by the ratio coordinates $\{Y_{t_{r,c}},Z_{t_{r,c}}\}_{t,r,c}$ and their inverses, over the base field $F$. As mentioned, $Y_{t_{r,c}}$'s and $Z_{t_{r,c}}$'s satisfy some algebraic relations, due to the fact that they are ratios of $\Delta_v$'s (Lem.\ref{lem:linear_algebraic_relations_among_ratio_coordinates}). Now observe that Lem.\ref{lem:coordinate_change_for_elementary_transformation_of_Qdt} assigns to each elementary morphism of ${\rm Qdt}_m(\wh{S})$ from $Q$ to $Q'$  an isomorphism ${\rm Frac}(\mathscr{K}^{Q'}(F)) \to {\rm Frac}(\mathscr{K}^Q(F))$ of fields of fractions of the higher Kashaev algebras; the fact that we are dealing with coordinate functions on an actual geometric space guarantees that this isomorphism is well-defined with respect to the algebraic relations of the generators (i.e. the relations are `preserved'). We thus get a contravariant functor
\begin{align}
\label{eq:Qdt_to_CAlg}
{\rm Qdt}_m(\wh{S}) \to {\rm CAlg}(F),
\end{align}
where ${\rm CAlg}(F)$ is the category of commutative algebras over $F$ whose morphisms are isomorphisms of fields of fractions. By composing \eqref{eq:mcal_K_S_to_Qdt} and \eqref{eq:Qdt_to_CAlg}, we get a map
\begin{align}
\label{eq:Kashaev_groupoid_to_CAlg}
\mcal{K}(\wh{S}) \to {\rm CAlg}(F),
\end{align}
defined on objects and on elementary morphisms; again by geometry, we are guaranteed that the algebraic relations of elementary morphisms of $\mcal{K}(\wh{S})$ are satisfied by their images, hence \eqref{eq:Kashaev_groupoid_to_CAlg} induces a contravariant functor. This functor encodes the data of ratio coordinates on $\mathscr{A}_{G,\wh{S}}$, including the coordinate change maps associated to each change of dotted triangulation of $\wh{S}$. From the formulas in Lem.\ref{lem:coordinate_change_for_elementary_transformation_of_Qdt} we see that each morphism of $\mcal{K}(\wh{S})$ is sent to a positive rational isomorphism between the fraction fields of the coordinate rings, therefore one can think that \eqref{eq:Kashaev_groupoid_to_CAlg} defines an analog of a  `positive space', which we denote by $\mathscr{K}_{m,\wh{S}}$. Let $\mathscr{K}_{m,\wh{S}}^+$ be the $\mathbb{R}_{>0}$-points of this `space'. When $F=\mathbb{C}$, one could also impose a natural $*$-structure on $\mathscr{K}^{Q_m(\tau_{\rm dot})}(\mathbb{C})$, in which all the generators $Y_{t_{r,c}}$ and $Z_{t_{r,c}}$ are $*$-invariant; see \S\ref{subsec:Kashaev_quantization} and Rem.\ref{rem:star_structure}.

\section{New quantization of higher Teichm\"uller space}
\label{sec:new_quantization}

\subsection{Kashaev's quantization of $\mathscr{K}^+_{2,\wh{S}}$}
\label{subsec:Kashaev_quantization}

We first discuss what we mean by the word `quantization' of $\mathscr{K}_{m,\wh{S}}^+$, which is described by the functor \eqref{eq:Kashaev_groupoid_to_CAlg}. For each object $\tau_{\rm dot}$ of $\mcal{K}(\wh{S})$, one obtains a commutative $*$-algebra $\mathscr{K}^{Q_m(\tau_{\rm dot})}(\mathbb{C})$, freely generated over $\mathbb{C}$ by $Y_{t_{r,c}}$ and $Z_{t_{r,c}}$ and their inverses modulo certain algebraic relations. One usually defines quantization in the case when such an algebra has a Poisson bracket structure. In our case we have a certain $2$-form $\Omega_{{\rm SL}_m,\wh{S}}$ \eqref{eq:2-form_diagonal_higher}
on $\mathscr{A}_{{\rm SL}_m,\wh{S}}$, from which we should extract a Poisson bracket structure on $\mathscr{K}^{Q_m(\tau_{\rm dot})}(\mathbb{C})$. Since the $2$-form is diagonal with respect to the logarithmic ratio coordinates, the corresponding Poisson bracket on $\mathscr{K}^{Q_m(\tau_{\rm dot})}(\mathbb{C})$ better be diagonal too. Thus we consider the Poisson bracket $\{ \cdot, \cdot \}_{\tau_{\rm dot}}$ on $\mathscr{K}^{Q_m(\tau_{\rm dot})}(\mathbb{C})$ formally defined as
\begin{align}
\{  q_{t_{r,c}}, p_{t'_{r',c'}}, \}_{\tau_{\rm dot}} = \delta_{t,t'} \delta_{r,r'} \delta_{c,c'} , \qquad
\{ q_{t_{r,c}}, q_{t'_{r',c'}} \}_{\tau_{\rm dot}} = \{ p_{t_{r,c}}, p_{t'_{r',c'}} \}_{\tau_{\rm dot}} = 0,
\end{align}
where $q_{t_{r,c}} = \log Y_{t_{r,c}}$ and $p_{t_{r,c}} = \log Z_{t_{r,c}}$. Using the Leibniz rule one can write this bracket in terms of the ratio coordinates $Y_{t_{r,c}}$ and $Z_{t_{r,c}}$. We generalize this to $(m,\wh{S})$-admissible quivers with dotted triangles:
\begin{definition}
\label{def:Poisson_on_K_Q}
Let $Q$ be an $(m,\wh{S})$-admissible quiver with dotted triangles, and let $I$ be the index set of its shaded triangles. Define the Poisson bracket $\{ \cdot, \cdot\}_Q$ on the algebra $\mathscr{K}^Q(F)$ by
\begin{align}
\{ Y_r, Z_s  \}_Q = \delta_{r,s} \, Y_r Z_s, \qquad
\{ Y_r, Y_s \}_Q = \{ Z_r, Z_s \}_Q = 0, \qquad \forall r,s\in I.
\end{align}
\end{definition}
The following lemma can easily be verified, or one can think of it as a corollary of a similar result for the $A$-coordinates of $\mathscr{A}_{{\rm SL}_m,\wh{S}}$.
\begin{lemma}
Suppose that a morphism of ${\rm Qdt}_m(\wh{S})$ from $Q$ to $Q'$ is induced by an elementary transformation of quivers with dotted triangles. Then the image ${\rm Frac}(\mathscr{K}^{Q'}(F)) \to {\rm Frac}(\mathscr{K}^{Q}(F))$ of this morphism under the functor \eqref{eq:Qdt_to_CAlg} preserves the Poisson structures defined in Def.\ref{def:Poisson_on_K_Q}.
\end{lemma}
So we have a Poisson structure on $\mathscr{K}_{m,\wh{S}}^+$, and we now define the quantization along this Poisson structure:
\begin{definition}
An {\em algebraic quantization of $\mathscr{K}_{m,\wh{S}}^+$} is the family of contravariant functors
$$
a_h : \mcal{K}(\wh{S}) \to *\mbox{-}{\rm Alg}
$$
with a parameter $h$, where $*\mbox{-}{\rm Alg}$ is the category of $*$-algebras over $\mathbb{C}$ whose morphisms are $*$-isomorphisms of skew fields of fractions, such that
\begin{enumerate}
\item The algebra $a_h(\tau_{\rm dot})$ is isomorphic to $\mathbb{C}[[h]] \otimes_\mathbb{C} \mathscr{K}^{Q_m(\tau_{\rm dot})} (\mathbb{C})$ as $\mathbb{C}$-vector spaces,

\vs

\item The product structure $\star$ of the above algebra $a_h(\tau_{\rm dot}) \equiv \mathbb{C}[[h]] \otimes_\mathbb{C} \mathscr{K}^{Q_m(\tau_{\rm dot})} (\mathbb{C})$ satisfies
$$
f \star g = fg + \frac{i}{2} \, h \, \{f, g\}_{\tau_{\rm dot}} + o(h)
$$
for all $f,g\in \mathscr{K}^{Q_m(\tau_{\rm dot})} (\mathbb{C})$, where $i = \sqrt{-1}$.
\end{enumerate}

\vs

We say that an algebraic quantization of $\mathscr{K}_{m,\wh{S}}^+$ is {\em canonical}, if $a_h(\tau_{\rm dot})$ is the free associative $*$-algebra over $\mathbb{C}$ generated by all $\wh{Y}_s$ and $\wh{Z}_s$ and their inverses for $s\in I$ where $I$ is as in Def.\ref{def:Poisson_on_K_Q} for $Q=Q_m(\tau_{\rm dot})$, modulo the relations
\begin{align}
\label{eq:commutation_relations_YZ}
\wh{Y}_r \wh{Z}_s = q^{2 \delta_{r,s}} \, \wh{Z}_s \wh{Y}_r, \qquad \forall r,s\in I, \qquad \mbox{where} \quad q = e^{ih/2},
\end{align}
possibly together with extra relations that go to some of the algebraic relations mentioned in Lem.\ref{lem:linear_algebraic_relations_among_ratio_coordinates} as $h \to 0$, where the $*$-structure is uniquely determined by requiring that all generators are $*$-invariant. We denote the algebras $a_h(\tau_{\rm dot})$ collectively by the symbol $\wh{\mathscr{K}}^{\,\, h}_{m,\wh{S}}$.
\end{definition}

We now give some examples of the extra relations. Formally let
\begin{align}
\label{eq:quantum_log_variables}
\wh{q}_s := \log \wh{Y}_s, \qquad \wh{p}_s := \log \wh{Z}_s, \qquad \forall s\in I,
\end{align}
satisfying the commutation relations
$$
[\wh{q}_r, \wh{p}_s] = i h \, \delta_{r,s},
$$
which imply the relations \eqref{eq:commutation_relations_YZ} by the Baker-Campbell-Hausdorff (BCH) formula:
\begin{align}
\label{eq:BCH}
e^P e^Q = e^{P+Q + \frac{1}{2}[P,Q]}, \quad\mbox{if $[P,Q]$ commutes with $P$ and $Q$}.
\end{align}

\begin{definition}[standard lift of a linear relation]
\label{def:standard_lift}
For each linear relation in Lem.\ref{lem:linear_algebraic_relations_among_ratio_coordinates} satisfied by $p_s$'s and $q_s$'s, it {\em standard lift} is defined by the same equation with all $p_s$'s and $q_s$'s replaced by corresponding $\wh{p}_s$'s and $\wh{q}_s$'s.
\end{definition}
We shall consider the algebraic relations among $\wh{Y}_s = e^{\wh{q}_s}$'s and $\wh{Z}_s = e^{\wh{p}_s}$'s induced via the BCH formula \eqref{eq:BCH} by some of the standard lifts of linear relations of $p_s$'s and $q_s$'s in Lem.\ref{lem:linear_algebraic_relations_among_ratio_coordinates}.
\begin{definition}[identification of quantized algebras]
\label{def:identification_of_quantized_algebras}
Suppose we have an algebraic quantization $a_h : \mcal{K}(\wh{S}) \to *\mbox{-}{\rm Alg}$. Since we use a same index set $I$ for the set of shaded triangles of $Q=Q_m(\tau_{\rm dot})$ for all $\tau_{\rm dot} \in \mcal{K}(\wh{S})$, we get a natural identification of the generators of the algebras $a_h(\tau_{\rm dot})$ for different $\tau_{\rm dot}$'s. If this identification induces well-defined algebra isomorphisms for all $\tau_{\rm dot}$'s, we view $\wh{\mathscr{K}}^{\,\, h}_{m,\wh{S}}$ as a single algebra under this identification.
\end{definition}
\begin{remark}
\label{rem:star_structure}
One can understand the $*$-structure on $\wh{\mathscr{K}}^{\,\, h}_{m,\wh{S}}$ formally as
\begin{align}
\label{eq:star_structure}
\wh{q}_s^{\, *} = \wh{q}_s, \qquad \wh{p}_s^{\, *} = \wh{p}_s, \qquad (\mbox{which implies} \qquad \wh{Y}_s^{\, *} = \wh{Y}_s, \qquad \wh{Z}_s^{\, *} = \wh{Z}_s),
\end{align}
and by conjugation on complex constants, as the linear relations of $\wh{p}_s$ and $\wh{q}_s$ as well as the relations \eqref{eq:commutation_relations_YZ} are invariant under this $*$-structure. This means that in terms of representation theory, we want to represent $\wh{q}_s$ and $\wh{p}_s$ as self-adjoint  operators on a Hilbert space.
\end{remark}

For $m=2$, in his algebraic quantization \cite{Kash98} Kashaev took  just the relations \eqref{eq:commutation_relations_YZ} and no more, for the definition of $a_h(\tau_{\rm dot})$. One can then easily check that Def.\ref{def:identification_of_quantized_algebras} applies. Hence the images of morphisms of $\mcal{K}(\wh{S})$ under $a_h$ are now $*$-algebra {\em automorphisms} of the skew field of fractions of a single $*$-algebra $\wh{\mathscr{K}}^{\,\, h}_{2,\wh{S}}$. These automorphisms  are given by formal conjugation by certain expressions written in terms of $\wh{q}_s$'s and $\wh{p}_s$'s defined in \eqref{eq:quantum_log_variables}. Such an expression is not an element of ${\rm Frac}(\wh{\mathscr{K}}^{\,\, h}_{2,\wh{S}})$. As a matter of fact, as usual in quantum theories, in Kashaev's quantization this $*$-algebra $\wh{\mathscr{K}}^{\,\, h}_{2,\wh{S}}$ is represented as an algebra of operators on some Hilbert space $\mathscr{H}$, where $\wh{q}_s, \wh{p}_s, \wh{Y}_s, \wh{Z}_s$ are represented as self-adjoint operators. Then, each morphism of $\mcal{K}(\wh{S})$ is realized as conjugation by some unitary operator, written in terms of $\wh{q}_s$ and $\wh{p}_s$, using functional calculus. Kashaev's unitary operators yield a family of `projective' functors
$$
\rho_h : \mcal{K}(\wh{S}) \to {\rm Hilb},
$$
where ${\rm Hilb}$ is the category of Hilbert spaces, sending all objects to $\mathscr{H}$; this means that we have a {\em projective} representation of the groupoid $\mcal{K}(\wh{S})$ on ${\rm Hilb}$, namely the following condition
\begin{align}
\label{eq:consistency_condition}
\mbox{consistency:} \quad
\rho_h( [\tau_{\rm dot}',\tau_{\rm dot}''] \circ [\tau_{\rm dot},\tau_{\rm dot}'] ) = c_{\tau_{\rm dot},\tau_{\rm dot}',\tau_{\rm dot}''} \, \rho_h([\tau_{\rm dot}',\tau_{\rm dot}'']) \, \rho_h ([\tau_{\rm dot},\tau_{\rm dot}']), \quad
\end{align}
holds for some constants $c_{\tau_{\rm dot},\tau_{\rm dot}',\tau_{\rm dot}''} \in {\rm U}(1)$. To describe Kashaev's result, we recall the definition of a special function called {\em quantum dilogarithm}:
\begin{definition}[\cite{FK},\cite{F},\cite{B}: quantum dilogarithm function]
Let $b>0$, $b^2 \notin \mathbb{Q}$. Let the function $\Psi_b(z)$ on the complex plane be defined by 
\begin{align}
\label{eq:quantum_dilogarithm}
\Psi_b (z) = \exp\left(\frac{1}{4} \int_{\Omega_0} \frac{ e^{-2izw}}{ \sinh (wb) \sinh(w/b)} \frac{dw}{w} \right)
\end{align}
first in the strip $|{\rm Im} \, z| < (b+b^{-1})/2$, where $\Omega_0$ means the real line contour with a detour around $0$ (origin) along a small half circle above the real line, and analytically continued to a meromorphic function on the complex plane using the following two functional equations:
\begin{align}
\nonumber
\Psi_b(z - ib^{\pm 1}/2) = (1+e^{2\pi b^{\pm 1} z}) \Psi_b(z+ ib^{\pm 1}/2).
\end{align}
\end{definition}
Here $b$ is related to our $h$ by $h = 2 \pi b^2$, and to the usual quantum parameter $q$ by $q = e^{\pi i b^2}$.

\vs

When $m=2$, each ideal triangle of an ideal triangulation of $\wh{S}$ has only one shaded triangle, so we just denote the shaded triangle $t_{1,1}$ by $t$. Let $\{1,2,\ldots,N\}$ be the index set for the set of all ideal triangles of an ideal triangulation of $\wh{S}$. We observe also that there is no invisible flip. Now Kashaev's quantization of $\mathscr{K}_{2,\wh{S}}^+$ can be described as follows. The $*$-algebra $\wh{\mathscr{K}}^{\,\, h}_{2,\wh{S}}$ is represented as
\begin{align}
\label{eq:star_representation}
\left\{ {\renewcommand{\arraystretch}{1.4}
\begin{array}{l}
\wh{q}_s = 2\pi b \, Q_s, \quad \wh{p}_s = 2\pi b P_s, \, \quad \wh{Y}_s = e^{\wh{q}_s}, \quad \wh{Z}_s = e^{\wh{p}_s}, \quad \mbox{for}\quad s=1,\ldots,N, \quad \mbox{where} \\
\displaystyle Q_s = x_s, \quad P_s = \frac{1}{2\pi i} \frac{\partial}{\partial x_s} \quad \mbox{on} \quad
\mathscr{H} = L^2(\mathbb{R}^N, \wedge_{s=1}^N dx_s).
\end{array} } \right.
\end{align}
\begin{remark}
These $\wh{q}_s, \wh{p}_s$ form a different, but unitarily equivalent, representation from \eqref{eq:canonical_quantization}.
\end{remark}
The operators $\wh{q}_s$, $\wh{p}_s$, $\wh{Y}_s$, $\wh{Z}_s$ are symmetric with respect to the usual inner product of $\mathscr{H}$, and is defined only on some dense subspace. For example, they act on the following subspace
$$
W_s := {\rm span}_\mathbb{C} \{ P(x_s) e^{-B x_s^2 + Cx_s} : \mbox{$P$ is a polynomial, $B>0$, $C\in \mathbb{C}$} \} \subset L^2(\mathbb{R}, dx_s),
$$
and it is actually sufficient to consider this subspace \cite{G}. So the algebra $\wh{\mathscr{K}}^{\,\, h}_{2,\wh{S}}$ can be thought of as acting on the space $\bigotimes_{s=1}^N W_s$. Kashaev \cite{Kash98} \cite{Kash00} then represented the elementary morphisms of $\mcal{K}(\wh{S})$ defined in Def.\ref{def:elementary_transformations_of_qdt} as the following unitary operators: 
\begin{align}
\label{eq:Kashaev_operators}
\left\{ {\renewcommand{\arraystretch}{1.2}
\begin{array}{rl}
\rho_h( \til{A}_s ) = & {\bf A}_s := e^{-\pi i/3} e^{3\pi i Q_s^2} e^{\pi i (P_s + Q_s)^2}, \\
\rho_h (\til{T}_{rs} ) = & {\bf T}_{rs} := e^{2\pi i P_r Q_s} \Psi_b( Q_r + P_s - Q_s )^{-1}, \\
(\rho_h (\til{P}_\sigma) f) ( \{x_s\}_{s=1}^N ) = & ({\bf P}_\sigma f)(\{x_s\}_{s=1}^N) := f( \{x_{\sigma(s)} \}_{s=1}^N ).
\end{array} } \right.
\end{align}
A friendly concrete description of these operators is available in \cite{Ki12}. The conjugation action by these operators are:
\begin{lemma}
\label{lem:conjugation_action}
One has
\begin{align*}
& {\bf A}_s^{-1} \wh{Y}_s {\bf A}_s = \wh{Z}_s^{-1}, \qquad
{\bf A}_s^{-1} \wh{Z}_s {\bf A}_s = q \wh{Y}_s \wh{Z}_s^{-1}, \qquad \mbox{where $q = e^{\pi i b^2}$}, \\
& {\bf T}_{rs}^{-1} \wh{Y}_r {\bf T}_{rs} = (\wh{Z}_s + \wh{Y}_r^{-1} \wh{Y}_s)^{-1}, \qquad {\bf T}_{rs}^{-1} \wh{Z}_r {\bf T}_{rs} = ( q^2 \wh{Y}_r \wh{Z}_r^{-1} \wh{Y}_s^{-1} \wh{Z}_s + \wh{Z}_r^{-1} )^{-1}\\
& {\bf T}_{rs}^{-1} \wh{Y}_s {\bf T}_{rs} = \wh{Y}_r \wh{Z}_s + \wh{Y}_s, \qquad
{\bf T}_{rs}^{-1} \wh{Z}_s {\bf T}_{rs} = \wh{Z}_r \wh{Z}_s, \\
& {\bf P}_\sigma^{-1} \wh{Y}_s {\bf P}_\sigma = \wh{Y}_{\sigma^{-1}(s)}, \quad {\bf P}_\sigma^{-1} \wh{Z}_s {\bf P}_\sigma = \wh{Z}_{\sigma^{-1}(s)}
\end{align*}
\end{lemma}

\begin{proof}
The last line can easily be seen. For the first line, we use
\begin{align}
\label{eq:A_commutation}
{\bf A}_s^{-1} P_s {\bf A}_s = Q_s - P_s, \qquad {\bf A}_s^{-1} Q_s {\bf A}_s = - P_s,
\end{align}
for a proof of which we refer the readers to \cite[Prop.4.18]{FrKi}. Or one can formally prove it directly, using the basic relations
\begin{align}
\label{eq:basic_relations_PQ}
[P_r, Q_s] =  \frac{\delta_{rs}}{2\pi i}, \quad [P_r, P_s] = [Q_r,Q_s]=0.
\end{align}
We then use the BCH formula \eqref{eq:BCH} to get the first line from \eqref{eq:A_commutation}. 

\vs

Following the proof in \cite[Prop.5.3]{FrKi}, for the second line we use
\begin{align}
\label{eq:basic_commutation}
e^{\ell P} f(Q) e^{-\ell P} = f ( Q + \ell [P,Q] ),
\end{align}
which holds for all $\ell \in \mathbb{R}$, any real-analytic function $f$ on $\mathbb{R}$, and self-adjoint operators $P,Q$ on a Hilbert space such that $[P,Q]$ `commute' with $P$ and $Q$; this makes sense only on a suitable dense subspace of the Hilbert space, such as $W_s$ mentioned above. Thus
\begin{align*}
{\bf T}_{rs}^{-1} \wh{Y}_r {\bf T}_{rs} & = \Psi_b( Q_r + P_s - Q_s ) e^{-2\pi i P_r Q_s}  e^{2\pi b Q_r}  e^{2\pi i P_r Q_s} \Psi_b( Q_r + P_s - Q_s )^{-1}  \\
& \stackrel{\eqref{eq:basic_commutation}}{=} \Psi_b( Q_r + P_s - Q_s )  e^{2\pi b (Q_r-Q_s)}  \Psi_b( Q_r + P_s - Q_s )^{-1} \\
& \stackrel{\eqref{eq:basic_commutation}}{=} \Psi_b( Q_r + P_s - Q_s )    \Psi_b( Q_r + P_s - Q_s - ib)^{-1} e^{2\pi b (Q_r-Q_s)}.
\end{align*}
We then use the functional equation
$$
\Psi_b(w) \Psi_b(w-ib)^{-1} = (1+e^{-\pi i b^2} e^{2\pi bw})^{-1},
$$
to get
\begin{align*}
{\bf T}_{rs}^{-1} \wh{Y}_r {\bf T}_{rs} & = ( 1 + e^{-\pi i b^2} e^{2\pi b (Q_r + P_s - Q_s )})^{-1} e^{2\pi b Q_r} e^{-2\pi b Q_s} \\
& = ( e^{2\pi b Q_s} e^{-2\pi b Q_r} (1 + e^{-\pi i b^2} e^{2\pi b (Q_r + P_s - Q_s )}) )^{-1}  \\
& = ( e^{2\pi b Q_s}  e^{-2\pi b Q_r} + e^{-\pi i b^2} \ul{ e^{2\pi b Q_s}  e^{2\pi b (P_s - Q_s )} } )^{-1}.
\end{align*}
Applying the BCH formula \eqref{eq:BCH} to the underlined part, we get ${\bf T}_{rs}^{-1} \wh{Y}_r {\bf T}_{rs} = (e^{-2\pi b Q_r} e^{2\pi b Q_s}  + e^{2\pi b P_s})^{-1} = (\wh{Z}_s + \wh{Y}_r^{-1} \wh{Y}_s)^{-1}$ as desired. Similar proof goes for ${\bf T}_{rs}^{-1} \wh{Z}_r {\bf T}_{rs}$ and ${\bf T}_{rs}^{-1} \wh{Y}_s {\bf T}_{rs}$, so we omit them. Finally,
\begin{align*}
{\bf T}_{rs}^{-1} \wh{Z}_s {\bf T}_{rs} & = \Psi_b( Q_r + P_s - Q_s ) e^{-2\pi i P_r Q_s}  e^{2\pi b P_s}  e^{2\pi i P_r Q_s} \Psi_b( Q_r + P_s - Q_s )^{-1}  \\
& \stackrel{\eqref{eq:basic_commutation}}{=} \Psi_b( Q_r + P_s - Q_s )  e^{2\pi b (P_s+P_r)} \Psi_b( Q_r + P_s - Q_s )^{-1} 
= e^{2\pi b (P_s+P_r)},
\end{align*}
yielding the desired result for the remaining case.
\end{proof}
One may take the above proof only as a formal argument, and consult Kashaev's original works or \cite{G} \cite{FG09} for more rigorous treatment. In the present paper we will be dealing only with formal algebraic expressions instead of actual operators on Hilbert spaces, so such an argument suffices for our purposes.

\vs

As the right-hand-sides of the identities in Lem.\ref{lem:conjugation_action} go to the relevant formulas in Lem.\ref{lem:coordinate_change_for_elementary_transformation_of_Qdt} as $b\to 0$, we see that these operators of Kashaev indeed are quantum versions of the coordinate change formulas. Moreover, these operators satisfy the consistency condition \eqref{eq:consistency_condition} too; it means that the operators ${\bf A}_\cdot$, ${\bf T}_{\cdot, \cdot}$, ${\bf P}_\cdot$ satisfy all the relations in Lem.\ref{lem:relations_of_elementary_moves_of_Kashaev_groupoid}, possibly up to multiplicative constants. The major part of it can be written as follows:

\begin{proposition}[\cite{Kash00}]
\label{prop:lifted_Kashaev_relations}
The unitary operators ${\bf A}_{\cdot}$, ${\bf T}_{\cdot,\cdot}$, ${\bf P}_\cdot$ in \eqref{eq:Kashaev_operators} satisfy
\begin{align}
\label{eq:lifted_Kashaev_rel}
{\bf T}_{rs} {\bf A}_{r} {\bf T}_{sr} & = \zeta \, {\bf A}_{r} {\bf A}_{s} {\bf P}_{(rs)}, \quad \mbox{where} \quad
\zeta = e^{-\pi i (b+b^{-1})^2/12},
\end{align}
and strictly satisfy all other relations of $\til{A}_{\cdot}$, $\til{T}_{\cdot,\cdot}$, $\til{P}_\cdot$ mentioned in Lem.\ref{lem:relations_of_elementary_moves_of_Kashaev_groupoid}:
\begin{align}
\label{eq:lifted_Kashaev_relations_major}
& {\bf A}_{s}^3 = {\rm id}, \quad
{\bf T}_{st} {\bf T}_{rs} = {\bf T}_{rs} {\bf T}_{rt} {\bf T}_{st}, \quad
{\bf A}_{r} {\bf T}_{rs} {\bf A}_{s} = {\bf A}_{s} {\bf T}_{sr} {\bf A}_{r}, \\
\label{eq:lifted_Kashaev_relations_P}
& {\bf P}_{\rm id} = {\rm id}, \quad {\bf P}_{\gamma_1} {\bf P}_{\gamma_2} = {\bf P}_{\gamma_1 \circ \gamma_2}, \quad {\bf P}_\gamma {\bf A}_{s} = {\bf A}_{\gamma(s)} {\bf P}_\gamma, \quad
{\bf P}_\gamma {\bf T}_{rs} = {\bf T}_{\gamma(r) \, \gamma(s)} {\bf P}_\gamma, 
\\
\label{eq:lifted_Kashaev_relations_commutation}
& {\bf T}_{rs} {\bf T}_{tu} = {\bf T}_{tu} {\bf T}_{rs}, \quad
{\bf T}_{rs} {\bf A}_{t} = {\bf A}_{t} {\bf T}_{rs}, \quad
{\bf A}_{r} {\bf A}_{s} = {\bf A}_{s} {\bf A}_{r}, 
\end{align}
for mutually distinct $r,s,t,u \in \{1,2,\ldots,N\}$ and any permutations $\gamma,\gamma_1,\gamma_2$ of $\{1,2,\ldots,N\}$.
\end{proposition}

\subsection{Quantization of $\mathscr{K}^+_{m,\wh{S}}$ for $m\ge 3$}

Let $m\ge 3$. Recall from the previous subsection that the algebraic quantization of $\mathscr{K}^+_{m,\wh{S}}$, in case Def.\ref{def:identification_of_quantized_algebras} is applicable, is given as a non-commutative algebra $\wh{\mathscr{K}}^{\,\, h}_{m,\wh{S}}$, generated over $\mathbb{C}$ by $\wh{Y}_s$'s and $\wh{Z}_s$'s and their inverses for $s$ running in the index set $I$ for all shaded triangles of an $m$-triangulation of an ideal triangulation $\tau$ of $\wh{S}$, together with a consistent assignment of an algebra automorphism of the skew field of fractions of this algebra to each morphism of $\mcal{K}(\wh{S})$. 

\vs

So we first need to say what the algebra $\wh{\mathscr{K}}^{\,\, h}_{m,\wh{S}} = a_h(\tau_{\rm dot})$ is, for each $\tau_{\rm dot} \in \mcal{K}(\wh{S})$. All we need to specify is the `extra' defining relations other than \eqref{eq:commutation_relations_YZ}, so that the set of relations for different $\tau_{\rm dot}$'s are preserved under the identification of the generators of $a_h(\tau_{\rm dot})$, so that Def.\ref{def:identification_of_quantized_algebras} may apply. One extreme candidate is to use no extra relations, as in Kashaev's quantization for $m=2$. Then we would be able to get well-defined representation of $\wh{\mathscr{K}}^{\,\, h}_{m,\wh{S}}$ on a Hilbert space, which is an optimal situation. However, as we will see, this does not work for $m\ge 3$. Another extreme candidate is to use the relations induced by the standard lifts (Def.\ref{def:standard_lift}) of {\em all} relations satisfied by classical functions $p_s$'s and $q_s$'s, i.e. the ones in Lem.\ref{lem:linear_algebraic_relations_among_ratio_coordinates}, together with the ones obtained by applying any index permutation of $I$. For $m=3$ we propose yet another candidate, using only {\em some} of the relations from Lem.\ref{lem:linear_algebraic_relations_among_ratio_coordinates}. Namely, for any two adjacent triangles $t$ and $s$ with any dot configurations, consider the `small middle diamond' in the quadrilateral formed by $t$ and $s$; see the dotted loop in Fig.\ref{fig:linearrel} for an example. Then via Lem.\ref{lem:linear_algebraic_relations_among_ratio_coordinates} and Def.\ref{def:standard_lift} this loop yields a linear relation of $\wh{q}_s$'s and $\wh{p}_s$'s. Collect all such relations, together with their index-permuted versions; then it is not hard to see that Def.\ref{def:identification_of_quantized_algebras} indeed applies. 
\begin{remark}
There might probably be another formulation of quantization so that we do not have to include all the index-permuted equations, but let us leave it for a topic of  future research. 
\end{remark}

\begin{figure}[htbp!]
\includegraphics[width=60mm]{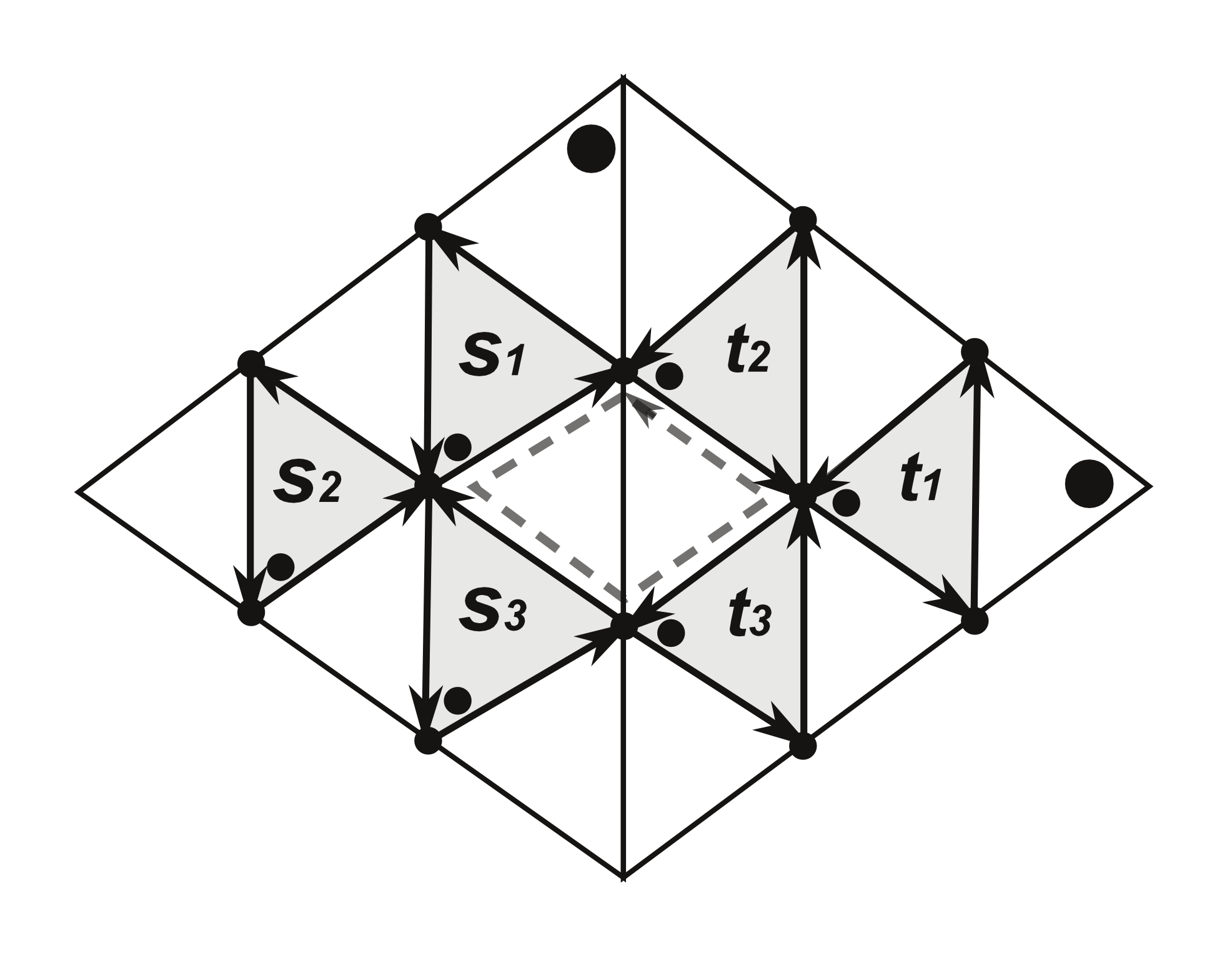}
\vspace{-7mm}
\caption{a source of an extra linear relation in $m=3$}
\label{fig:linearrel}
\end{figure}

\vspace{-2mm}

\vs

We now would like to associate to each morphism of $\mcal{K}(\wh{S})$ an automorphism of ${\rm Frac}(\wh{\mathscr{K}}^{\,\, h}_{m,\wh{S}})$. We shall use \eqref{eq:mcal_K_S_to_Qdt}; so we need to associate to each elementary morphism of ${\rm Qdt}_m(\wh{S})$ an automorphism of ${\rm Frac}(\wh{\mathscr{K}}^{\,\, h}_{m,\wh{S}})$, so that the resulting map $\mcal{K}(\wh{S}) \to *\mbox{-}{\rm Alg}$ is indeed a contravariant functor. The elementary transformations $A_s$, $T_{rs}$, $P_\sigma$ of ${\rm Qdt}_m(\wh{S})$ as in Def.\ref{def:elementary_transformations_of_qdt}  lead to automorphisms of ${\rm Frac}(\wh{\mathscr{K}}^{\,\, h}_{m,\wh{S}})$ given by the formal conjugation by the expressions in \eqref{eq:Kashaev_operators}. The concrete formulas for these automorphisms are documented in Lem.\ref{lem:conjugation_action}; however, it is still more convenient to work with the expressions \eqref{eq:Kashaev_operators} for ${\bf A}_{\cdot}$ and ${\bf T}_{\cdot \cdot}$, where $Q_s = \frac{1}{2\pi b} \wh{q}_s = \frac{1}{2\pi b} \log \wh{Y}_s$'s and $P_s = \frac{1}{2\pi b} \wh{p}_s = \frac{1}{2\pi b} \log \wh{Z}_s$'s, which comply with the notation for the operators $Q_s$ and $P_s$ which appeared in \eqref{eq:star_representation} in the previous subsection. We regard ${\bf P}_\cdot$ as a certain expression in $Q_s$'s and $P_s$'s satisfying the relations \eqref{eq:lifted_Kashaev_relations_P}, or more basically 
\begin{align}
\label{eq:P_conjugation_on_P_and_Q}
{\bf P}_\sigma^{-1} Q_s {\bf P}_\sigma = Q_{\sigma^{-1}(s)}, \quad 
{\bf P}_\sigma^{-1} P_s {\bf P}_\sigma = P_{\sigma^{-1}(s)},
\quad
{\bf P}_{\rm id} = {\rm id},
\quad
{\bf P}_{\gamma_1} \circ {\bf P}_{\gamma_2} = {\bf P}_{\gamma_1 \circ \gamma_2}.
\end{align}
In principle, we could express ${\bf P}_\cdot$ as a product of exponentials of quadratic expressions in $Q_s$'s and $P_s$'s like ${\bf A}_\cdot$, but let us just focus on its properties in  \eqref{eq:P_conjugation_on_P_and_Q}.

\vs

Now, in the case of $m\ge 3$, there are also invisible flips (see Def.\ref{def:elementary_transformations_of_qdt}), whose associated coordinate change maps are given as in Lem.\ref{lem:coordinate_change_for_elementary_transformation_of_Qdt}. The quantum version of this coordinate change map for an invisible flip $F_{rs}$ that we propose is the conjugation by the following expression:
\begin{align}
\label{eq:rho_h_F_rs}
\rho_h (F_{rs}) := {\bf F}_{rs} := e^{-2\pi i Q_r P_s},
\end{align}
Using \eqref{eq:basic_relations_PQ} and \eqref{eq:basic_commutation}, we can easily verify:
\begin{lemma}
\label{lem:conjugation_by_F_rs}
One has
\begin{align*}
{\bf F}_{rs}^{-1} \wh{Y}_r {\bf F}_{rs} = \wh{Y}_s, \quad
{\bf F}_{rs}^{-1} \wh{Z}_r {\bf F}_{rs} = \wh{Z}_r \wh{Z}_s^{-1}, \quad
{\bf F}_{rs}^{-1} \wh{Y}_s {\bf F}_{rs} = \wh{Y}_r \wh{Y}_s, \quad
{\bf F}_{rs}^{-1} \wh{Z}_s {\bf F}_{rs} = \wh{Z}_s. \qed
\end{align*}
\end{lemma}
Thus indeed the formal conjugation by ${\bf F}_{rs}$ recovers the classical coordinate change formula for the invisible flip $F_{rs}$ as $b \to 0$; see Lem.\ref{lem:coordinate_change_for_elementary_transformation_of_Qdt}. Hence, what remains to be done is to check for this algebraic quantization that the above constructed images of elementary morphisms $\mcal{K}(\wh{S})$, which are automorphisms of ${\rm Frac}(\wh{\mathscr{K}}^{\,\, h}_{m,\wh{S}})$, satisfy the algebraic relations in Lem.\ref{lem:relations_of_elementary_moves_of_Kashaev_groupoid}. This is done in \S\ref{subsec:consistency_condition_for_3} for the case $m=3$, by showing a formal version of the consistency condition \eqref{eq:consistency_condition}. Unlike the case $m=2$, our proof of the consistency condition for $m=3$ {\em does} use the lifted linear relations of $\wh{p}_r$'s and $\wh{q}_r$'s, and therefore we cannot use the representation as in \eqref{eq:star_representation}. 

\vs

So, to summarize, for $m=2$ there is a quantization of $\mathscr{K}^+_{2,\wh{S}}$ as operators on a Hilbert space. For $m=3$ we have only an algebraic quantization, but not represented on a Hilbert space yet. For $m>3$ we have a candidate for an algebraic quantization, but without a proof. What Kashaev quantized for $m =2$ is some lift of $\mathscr{K}^+_{2,\wh{S}}$, obtained by forgetting the linear relations of $q_s$'s and $p_s$'s in Def.\ref{def:standard_lift}. This is analogous to what Fock-Goncharov did \cite{FG09}; they quantized some lift of a Teichm\"uller space of a punctured surface parametrized by shear coordinates, obtained by forgetting the linear relations among the shear coordinates coming from small loops around punctures. Their lifted space is the holed Teichm\"uller space; a natural geometric meaning of the space which we quantized in the present paper is not clear.

\section{Proof}
\label{sec:proof}

\subsection{Proof of completeness of Kashaev relations for dotted triangulations}
\label{subsec:proof_of_completeness}

In the present subsection we present a proof of Prop.\ref{prop:Kashaev_full_presentation}. However, to comply with Kashaev's own notation \cite{Kash00}, we introduce slightly different elementary moves of the Kashaev groupoid $\mcal{K}(\wh{S})$ (Def.\ref{def:Kashaev_groupoid}), as mentioned in Rem.\ref{rem:Kashaev_notation}:
$$
\rho_i := \til{T}_i, \quad
\omega_{ij} :=  \til{A}_i \til{A}_j^{-1} \til{T}_{ij} \til{A}_j \til{A}_i^{-1}, \quad
P_\sigma := \til{P}_\sigma,
$$
for $i,j \in I$ and permutations $\sigma$ of $I$, where $I$ is a fixed index set of ideal triangles of an ideal triangulation of $\wh{S}$. From now on, when we say {\em elementary moves} of $\mcal{K}(\wh{S})$ we mean these. Then we can rewrite Lem.\ref{lem:relations_of_elementary_moves_of_Kashaev_groupoid} as follows: the new elementary moves satisfy
\begin{align}
\label{eq:new_elementary_moves_relations}
\left\{ {\renewcommand{\arraystretch}{1.4} 
\begin{array}{l}
\rho_i^3 = {\rm id}, \quad 
\omega_{jk} \omega_{ik} \omega_{ij}  = \omega_{ij} \omega_{jk}, \quad
\omega_{ij} \rho_j \omega_{ji} = \rho_i \rho_j P_{(i \, j)}, \quad
\rho_j \omega_{ij} \rho_i = \rho_i \omega_{ji} \rho_j, \\
P_{\rm id} = {\rm id}, \quad
P_{\sigma_1} P_{\sigma_2} = P_{\sigma_1 \circ \sigma_2}, \quad
P_\sigma \rho_i = \rho_{\sigma(i)} P_\sigma, \quad
P_\sigma \omega_{ij} = \omega_{\sigma(i) \, \sigma(j)} P_\sigma, \\
\omega_{ij} \omega_{k\ell} = \omega_{k\ell} \omega_{ij}, \quad
\omega_{ij} \rho_k = \rho_k \omega_{ij}, \quad
\rho_i \rho_j = \rho_j \rho_i,
\end{array} } \right.
\end{align}
for mutually distinct $i,j,k,\ell \in I$ and for any  permutations $\sigma_1,\sigma_2,\sigma$ of $I$, where each equality means that if $(\tau,D,L)$ supports the left-hand side (Def.\ref{def:support}), then it supports the right-hand side too, and the results of the application of both sides to $(\tau,D,L)$ are the same. The statement to prove is as follows.

\begin{proposition}
\label{ss}
Any morphism of $\mcal{K}(\wh{S})$ is composition of finite number of these new elementary morphisms, and any algebraic relations among elementary moves are consequences of the ones in \eqref{eq:new_elementary_moves_relations}.
\end{proposition}

We convert this into the language of CW-complexes.
\begin{definition}[Kashaev complex]
The {\em Kashaev complex} $\mcal{C}_{\mcal{K}}(\wh{S})$ of a marked hyperbolic surface $\wh{S}$ is a two-dimensional CW-complex constructed as follows.

The set of vertices ($0$-cells) are the set of all objects of $\mcal{K}(\wh{S})$, that is, all dotted ideal triangulations of $\wh{S}$ (Def.\ref{def:dotted_triangulation}). There is one edge ($1$-cell) between two (possibly same) vertices if and only if one vertex can be obtained by applying a single elementary move to the other. We can label the oriented edge with the corresponding elementary move, which we refer to as the {\em type} of the edge. If $(\tau', D', L')$ is obtained by applying an elementary move $g$ to $(\tau,D,L)$, we write $(\tau',D',L') = g(\tau,D,L)$ and we denote the edge between them by $(\Delta,D,L) \stackrel{g}{\longrightarrow} (\Delta',D',L')$ or $(\Delta',D',L') \stackrel{g^{-1}}{\longrightarrow} (\Delta,D,L)$. These are all the edges.

The $2$-cells come from the relations \eqref{eq:new_elementary_moves_relations} as follows. Take a relation among \eqref{eq:new_elementary_moves_relations}, and write it as $g_\ell g_{\ell-1} \ldots g_1 =1$, where each $g_a$ is an elementary move or its inverse. For each $(\Delta,D,L)$ supporting $g_\ell g_{\ell-1} \ldots g_1$, think of the $n$ vertices $g_a g_{a-1} \cdots g_1(\Delta,D,L)$, for $a=1,2,\ldots,\ell$. One observes that the vertices $g_a  \cdots g_1(\Delta,D,L)$ and $g_{a+1} \cdots g_1(\Delta,D,L)$ are connected by an edge, for $a=1,\ldots,\ell$, where $g_{\ell+1}$ is defined to be $g_1$. Then we attach a $2$-cell so that its boundary coincides with this $1$-dimensional subcomplex consisting of $\ell$ vertices and $\ell$ edges. The relation associated to each $2$-cell is called the {\em type} of the $2$-cell.

Attaching such $2$-cell to every $(\Delta,D,L)$ supporting the left-hand side of this relation, and doing this for each relation in \eqref{eq:new_elementary_moves_relations}, we get the CW-complex $\mcal{C}_{\mcal{K}}(\wh{S})$.
\end{definition}

What we shall actually prove in the present  subsection is the following:
\begin{proposition}
\label{prop:Kashaev_complex_simply-connected}
The Kashaev complex $\mcal{C}_{\mcal{K}}(\wh{S})$ is connected and simply-connected.
\end{proposition}
This means first that any two $(\Delta,D,L)$ and $(\Delta',D',L')$ can be connected by a finite sequence of elementary moves, and more importantly that any algebraic relation satisfied by elementary moves is a consequence of finitely many relations in \eqref{eq:new_elementary_moves_relations}. Instead of proving this from scratch, we quote a similar result for the complex constructed out of the Ptolemy groupoid $Pt(\wh{S})$ (Def.\ref{def:Ptolemy_groupoid}) of ideal triangulations $\tau$ of $\wh{S}$ (Def.\ref{def:ideal_triangulation}).
\begin{definition}[Ptolemy complex; see \cite{PennerBook}]
The {\em Ptolemy complex} $\mcal{C}_{Pt} (\wh{S})$ is the two-dimensional CW-complex, defined as follows. The vertices ($0$-cells) are all possible ideal triangulations $\tau$ of $\wh{S}$. For two distinct vertices $\tau$ and $\tau'$ of $\mcal{C}_{Pt} (\wh{S})$ that differ by a flip $W_e$ for an edge $e$ of $\tau$, there is an edge ($1$-cell) between them; these are all the edges. There are following types of $2$-cells:
\begin{align*}
\begin{array}{rl}
{\rm [pentagon]} & (WW =WWW){\rm -type}, \\
{\rm [commutativity]} & (WW= WW){\rm -type}, \\
{\rm [involutivity]} & (WW=1){\rm -type},
\end{array}
\end{align*}
attached according to the following relations, in an analogous manner as in the case of the Kashaev complex $\mcal{C}_\mcal{K}(\wh{S})$. First, for convenience, let us label the edges of an ideal triangulation $\tau$. Then ideal triangulations connected to $\tau$ by a flip naturally carries a labeling for edges, induced from that of $\tau$. For any edge $e$, the relation $W_e W_e = {\rm id}$ corresponds to the `involutivity $2$-cell'. For any edges $e_1,e_2$ not sharing an endpoint, the relation $W_{e_1} W_{e_2} = W_{e_2} W_{e_1}$ corresponds to the `commutativity $2$-cell'. For any edges $e_1,e_2$ sharing exactly one endpoint, the relation $W_{e_1} W_{e_2} W_{e_1} W_{e_2} W_{e_1} = {\rm id}$ corresponds to the `pentagon $2$-cell'.
\end{definition}

\begin{proposition}[Whitehead's classical fact]
\label{prop:Whitehead}
This complex $\mcal{C}_{Pt}(\wh{S})$ is connected and simply-connected. 
\end{proposition}
We refer the readers to Bob. Penner's book \cite{PennerBook} for a proof and original references. Now, for a proof of Prop.\ref{prop:Kashaev_complex_simply-connected}, we mimic Funar-Kapoudjian's idea \cite{FuKa2} of using the following lemma of Bakalov-Kirillov \cite{BaKi}:

\begin{lemma}[Prop.6.2 of {\cite{BaKi}}]
\label{lem:BaKi}
Let $\mcal{M}$ and $\mcal{C}$ be two $CW$-complexes of dimension $2$, with oriented edges, and $f : \mcal{M}^{(1)} \to \mcal{C}^{(1)}$ be a cellular map between their $1$-skeletons, which is surjective on $0$-cells and $1$-cells. Suppose that:
\begin{enumerate}
\item $\mcal{C}$ is connected and simply-connected;

\item For each vertex $c$ of $\mcal{C}$, $f^{-1}(c)$ is connected and simply-connected in $\mcal{M}$ (that is, every closed loop which lies completely in $f^{-1}(c)$ is contractible in $\mcal{M}$);

\item Let $c_1 \stackrel{ e}{\longrightarrow} c_2$ be an oriented edge of $\mcal{C}$, and let $m_1' \stackrel{ e'}{\longrightarrow} m_2'$ and $m_1'' \stackrel{e''}{\longrightarrow} m_2''$ be two lifts in $\mcal{M}$. Then we can find two paths $m_1' \stackrel{ p_1 }{\longrightarrow} m_1''$ in $f^{-1}(c_1)$ and $m_2' \stackrel{ p_2}{ \longrightarrow } m_2''$ in $f^{-1}(c_2)$ such that the loop
$$
\xymatrix{
m_1' \ar[r]^{e'} \ar[d]_{p_1} & m_2' \ar[d]^{p_2} \\
m_1'' \ar[r]^{e''} & m_2''
}
$$
is contractible in $\mcal{M}$;

\item For any $2$-cell $X$ of $\mcal{C}$, its boundary $\partial X$ can be lifted to a contractible loop of $\mcal{M}$.
\end{enumerate}
Then $\mcal{M}$ is connected and simply-connected.
\end{lemma}
\begin{remark}
The orientation of the edges are mainly for notational use.
\end{remark}

We will apply this lemma to
\begin{align}
\label{eq:our_situation}
\mcal{M} = \mcal{C}_{\mcal{K}}(\wh{S}) \quad\mbox{and} \quad \mcal{C} = \mcal{C}_{Pt}(\wh{S}).
\end{align}
Then the condition 1 of Lem.\ref{lem:BaKi} is satisfied because of Prop.\ref{prop:Whitehead}. We define the cellular map
\begin{align}
\label{eq:our_f}
f : \mcal{C}_{\mcal{K}}^{(1)}(\wh{S}) \to \mcal{C}_{Pt}^{(1)} (\wh{S})
\end{align}
between the $1$-skeletons by `forgetting the dots and labels'. The vertex $(\tau,D,L)$ is sent via $f$ to $\tau$. The edges of type $\rho_i$ and $P_\sigma$ emanating from $(\tau,D,L)$ are sent to the vertex $\tau$. The edge of type $\omega_{ij}$ connecting $(\tau,D,L)$ and $(\tau',D',L')$ is sent to the edge connecting $\tau$ and $\tau'$. Then the surjectivity on $0$-cells and $1$-cells is easily seen.

\vs

Now, to prove condition (2) of Lem.\ref{lem:BaKi} for our situation, fix any ideal triangulation $\tau$. One can observe that $f^{-1}(\tau)$ is a one-dimensional complex, whose vertices are $(\tau,D,L)$ for all possible dotting rules $D$ and labeling rules $L$ for $\tau$. It is easy to see that this is connected, because for any two vertices $(\tau,D,L)$ and $(\tau,D',L')$, there is an edge between $(\tau,D,L)$ and $(\tau,D,L')$, and a path between $(\tau,D,L')$ and $(\tau,D',L')$.  
We now need to show that any closed loop in $f^{-1}(\tau)$ is contractible in $\mcal{C}_{\mcal{K}}(\wh{S})$. 
Observe that some of the $2$-cells of $\mcal{C}_{\mcal{K}}(\wh{S})$ are attached to $f^{-1}(\tau)$: namely, the $2$-cells of type $\rho_i^3 = {\rm id}$, $P_{\rm id}={\rm id}$, $P_{\sigma_1} P_{\sigma_2} = P_{\sigma_1\circ \sigma_2}$, $P_\sigma \rho_i = \rho_{\sigma(i)} P_\sigma$, or $\rho_i \rho_j = \rho_j \rho_i$. Each of these relations is supported at every vertex of $f^{-1}(\tau)$, and the resulting $2$-cell is attached to $f^{-1}(\tau)$. Denote by $\mcal{M}_\tau$ the CW-subcomplex of $\mcal{C}_{\mcal{K}}(\wh{S})$ whose $1$-skeleton is $f^{-1}(\tau)$ and the $2$-cells are those just described. It suffices to show that $\mcal{M}_\tau$ is simply-connected, and for that we use Lem.\ref{lem:BaKi} again:

\begin{lemma}
\label{lem:M_tau_simply-connected}
The CW-complex $\mcal{M}_\tau$ defined above is simply-connected.
\end{lemma}

This lemma will be proved later. Assuming it, we then have condition (2) of Lem.\ref{lem:BaKi} for our situation \eqref{eq:our_situation}, \eqref{eq:our_f}. It takes a few pages to check condition (3), so it is postponed until Lem.\ref{lem:condition3_for_our_situation}. Let us first check condition (4) here.

\vs

Recall that there are three types of $2$-cells of $\mcal{C}_{Pt} (\wh{S})$. It is easy to see that the pentagon type $2$-cell of $\mcal{C}_{Pt} (\wh{S})$ is lifted by the $2$-cell of $\mcal{C}_{\mcal{K}}(\wh{S})$ of type $\omega_{jk} \omega_{ik} \omega_{ij}  =  \omega_{ij} \omega_{jk}$, where in the lift we chose the labels of the three triangles involved in the flips to be $i,j,k$ appropriately. The details of a proof is omitted. The commutativity type $2$-cell of $\mcal{C}_{Pt} (\wh{S})$ is lifted by the $2$-cell of $\mcal{C}_{\mcal{K}}$ of type $\omega_{ij} \omega_{k\ell} = \omega_{k\ell} \omega_{ij}$, where in the lift we chose the labels of the two triangles having the first edge $e_1$ in the relation $W_{e_1} W_{e_2} = W_{e_2} W_{e_1}$  as their shared side as $i,j$ and those of the two triangles for the second edge as $k,\ell$. For the involutivity type $2$-cell of $\mcal{C}_{Pt} (\wh{S})$, we lift its boundary to the loop $\omega_{ij}^{-1} \omega_{ij}$ in $\mcal{C}_{\mcal{K}}(\wh{S})$, which is contractible.

\vs

It remains to prove Lem.\ref{lem:M_tau_simply-connected}, and to check condition (3) for our situation \eqref{eq:our_situation}, \eqref{eq:our_f}.

\vs

{\it Proof of Lem.\ref{lem:M_tau_simply-connected}.} Define a new two-dimensional CW-complex $\mcal{C}_\tau$ as follows. Choose and fix one labeling rule $L_0$. The vertices of $\mcal{C}_\tau$ are $(\tau,D,L_0)$ for all possible dotting rule $D$ for $\tau$. Two vertices are connected by an edge if and only if one vertex is obtained from the other by applying $\rho_i$ for some $i\in I$, where $I$ is the fixed index set of ideal triangles. Then it is not hard to see that $\mcal{C}_\tau$ is connected. For any three distinct vertices which are connected by $\rho_i$ for the same $i$, we attach a $2$-cell so that its boundary is identified with the triangle formed by the three edges going among these vertices. These are all the $2$-cells. Now, suppose that we have a closed loop in $\mcal{C}_\tau$ starting from $(\tau,D,L_0)$. We choose an orientation of this loop, and write down the types of the edges as traversing along the loop once, to get a word $\rho^{\epsilon_\ell}_{i_\ell} \cdots \rho^{\epsilon_2}_{i_2} \rho^{\epsilon_1}_{i_1}$ (written from right). Here $i_j \in I$, $\epsilon_j \in \{\pm 1\}$, where $i_1,\ldots,i_\ell$ may not be mutually distinct. Using the commutativity relations $\rho_i \rho_j = \rho_j \rho_i$ and the relations $\rho_i^3=id$, we can write this as $\rho^{\epsilon_{\ell'}'}_{k_{\ell'}} \cdots \rho^{\epsilon_2'}_{k_2} \rho^{\epsilon_1'}_{k_1}$, where $k_1,\ldots,k_{\ell'}$ are mutually distinct elements of $I$, and $\epsilon_1',\ldots,\epsilon_{\ell'}'$ are elements of $\{\pm 1\}$. This means that the loop corresponding to this new word is homotopic  in $\mcal{C}_\tau$ to the original loop. When this word is applied to $(\tau,D,L_0)$ (i.e. apply $\rho^{\epsilon_1'}_{k_1}$ first, then $\rho^{\epsilon_2'}_{k_2}$, etc), then we should get back $(\tau,D,L_0)$, and the only way for this to happen is that $\ell'=0$, because the dot of the triangle labeled by $k_j$ should not be altered. So the new loop is a constant loop, and thus we proved that $\mcal{C}_\tau$ is simply-connected. Hence the condition (1) of Lem.\ref{lem:BaKi} is satisfied.

\vs

The map $f_\tau : \mcal{M}_\tau^{(1)} \to \mcal{C}_\tau^{(1)}$ is given by `setting the labels to be $L_0$',  i.e. the vertex $(\tau,D,L)$ is sent to $(\tau,D,L_0)$, the edge connecting $(\tau,D,L)$ and $(\tau,D,L')$ is sent to the vertex $(\tau,D,L_0)$, and the edge connecting $(\tau,D,L)$ and $(\tau,D',L)$ is sent to the edge connecting $(\tau,D,L_0)$ and $(\tau,D',L_0)$. For each vertex $(\tau,D,L_0)$ of $\mcal{C}_\tau^{(1)}$, its pre-image under $f_\tau$ consists of vertices $(\tau,D,L)$ for all possible labeling $L$ for $\tau$, and edges going between every pair of vertices $(\tau,D,L)$ and $(\tau,D,L')$. So $f_\tau^{-1}(\tau,D,L_0)$ is connected. We must show that any loop in $f_\tau^{-1}(\tau,D,L_0)$ is contractible in $\mcal{M}_\tau$. We first collapse each subcomplex of $\mcal{M}_\tau$ that is the closure of the a $2$-cell of type $P_{\rm id}={\rm id}$, which is contractible. By abuse of notation, still denote by $f_\tau^{-1}(\tau,D,L_0)$ and $\mcal{M}_\tau$ after this collapsing; then each edge of $f_\tau^{-1}(\tau,D,L_0)$ connects two distinct vertices. For any three distinct vertices of $f_\tau^{-1}(\tau,D,L_0)$, the three edges going among them form a triangle, and in $\mcal{M}_\tau$ a $2$-cell of type $P_{\sigma_1} P_{\sigma_2} = P_{\sigma_1 \circ \sigma_2}$ is attached so that its boundary is identified with this triangle.  One can observe that the CW-subcomplex of $\mcal{M}_\tau$ consisting of $f_\tau^{-1}(\tau,D,L_0)$ together with these $2$-cells is just the $2$-skeleton of the $(k-1)$-simplex, where $k$ is the number of vertices of $f_\tau^{-1}(\tau,D,L_0)$. Since $(k-1)$-simplex is contractible, any loop inside $f_\tau^{-1}(\tau,D,L_0)$ is contractible in $\mcal{M}_\tau$, fulfilling the condition (2) of Lem.\ref{lem:BaKi}.

\vs

Now, let $(\tau,D,L_0) \to (\tau,D',L_0)$ be an edge in $\mcal{C}_\tau$, and $(\tau,D,L) \to (\tau,D',L)$ and $(\tau,D,L') \to (\tau,D',L')$ be two lifts of it in $\mcal{M}_\tau$; it is easy to see that any lift looks like them, because $D$ and $D'$ are distinct. Since $(\tau,D,L_0) \to (\tau,D',L_0)$ is an edge in $\mcal{C}_\tau$, there is $i_0 \in I$ such that $\rho_{i_0}^{\epsilon} (\tau,D,L_0) = (\tau,D',L_0)$, where $\epsilon \in \{\pm 1\}$. Let $t$ be the ideal triangle of $\tau$ labeled by $i_0$ by $L_0$. Now let $i = L(t)$ and $i'=L'(t)$. Then we have $\rho_i^\epsilon (\tau,D,L) = (\tau,D',L)$ and $\rho_{i'}^\epsilon(\tau,D,L') = (\tau,D',L')$. Meanwhile, we can find a unique permutation $\sigma$ of $I$ such that $L' = \sigma \circ L$; in particular, $i' = \sigma(i)$. Then the edges $(\tau,D,L) \to (\tau,D,L')$ of $f_\tau^{-1}(\tau,D,L_0)$ and $(\tau,D',L) \to (\tau, D',L')$ of $f_\tau^{-1}(\tau,D',L_0)$ are both labeled by the elementary move $P_\sigma$. Thus we have the loop
$$
\xymatrix{
(\tau,D,L) \ar[r]^{\rho_i^\epsilon} \ar[d]_{P_\sigma} & (\tau,D',L) \ar[d]^{P_\sigma} \\
(\tau,D,L') \ar[r]^{\rho_{i'}^\epsilon} & (\tau,D',L')
}
$$
This loop is contractible in $\mcal{M}_\tau$ because there is a $2$-cell of type $P_\sigma \rho_i = \rho_{\sigma(i)} P_\sigma$ attached to it along the boundary. Hence the condition (3) of Lem.\ref{lem:BaKi}.

\vs

Any $2$-cell of $\mcal{C}_\tau$ is attached to the triangle with vertices $(\tau,D,L_0)$, $\rho_i (\tau,D,L_0) = (\tau,D',L_0)$, $\rho_i^2 (\tau,D,L_0) = (\tau,D'',L_0)$, with edges labeled by $\rho_i$ going among them, for some $D$ and $i\in I$. These vertices and edges can be thought of as living in $f_\tau^{-1}(\mcal{C}^{(1)}_\tau) \subset \mcal{M}_\tau$, and they are the lifts of the original ones in $\mcal{C}_\tau$, via $f_\tau$. The loop formed by the edges of this triangle
$$
\xymatrix{
(\tau,D,L_0) \ar[rr]^{\rho_i} & &(\tau,D',L_0) \ar[dl]^{\rho_i} \\
& (\tau,D'',L_0) \ar[ul]^{\rho_i} &
}
$$
is contractible in $\mcal{M}_\tau$ because there is a $2$-cell of type $\rho_i^3 ={\rm id}$ attached to it along the boundary. Hence the condition (4) of Lem.\ref{lem:BaKi}. Therefore, by Lem.\ref{lem:BaKi}, the CW-complex $\mcal{M}_\tau$ is simply connected, as desired. \qed

\vs

We then finish our proof of Prop.\ref{prop:Kashaev_complex_simply-connected} by the following:

\begin{lemma}\label{lem:condition3_for_our_situation}
The condition (3) of Lem.\ref{lem:BaKi} is fulfilled, for our situation \eqref{eq:our_situation}, \eqref{eq:our_f}.
\end{lemma}

\begin{proof} The temporary variables $k$'s and $n$'s appearing in this proof have nothing to do with the ones describing $\wh{S}$ in \S\ref{subsec:higher_Teichmuller_spaces}. Let $W_e : \tau \to \tau'$ be an oriented edge in $\mcal{C}_{Pt}(\wh{S})$, i.e. a flip of an ideal triangulation. Any lift of this edge in $\mcal{C}_{\mcal{K}}(\wh{S})$ via $f$ \eqref{eq:our_f} is an oriented edge of type $\omega_{ij}^\epsilon$ for $\epsilon \in \{\pm1\}$. Let
$$
(\tau,D,L) \stackrel{ \omega_{ij}^\epsilon }{ \longrightarrow } (\tau', D', L') \quad\mbox{and}\quad 
(\tau,D_0,L_0) \stackrel{ \omega_{i_0 j_0}^{\epsilon_0} }{ \longrightarrow } (\tau', D_0', L_0')
$$
be two lifts of this edge in $\mcal{C}_{\mcal{K}}(\wh{S})$. We claim that we can find edges $(\tau,D_k,L_k) \stackrel{ \omega_{i_k j_k}^{\epsilon_k} }{ \longrightarrow } (\tau', D_k', L_k')$ for $k=1,2,\ldots,n$, so that the edge for $k=n$ coincides with the edge $(\tau,D,L) \stackrel{ \omega_{ij}^\epsilon }{ \longrightarrow } (\tau', D', L')$, and that there exist paths $(\tau,D_{k-1},L_{k-1}) \stackrel{ p_k }{ \longrightarrow } (\tau,D_k,L_k)$ in $f^{-1}(\tau)$ and $(\tau',D_{k-1}',L_{k-1}') \stackrel{ p_k' }{\longrightarrow} (\tau',D_k',L_k')$ in $f^{-1}(\tau')$, where $p_k$ and $p_k'$ are some words in elementary moves and their inverses, such that the loop
$$
\xymatrix{
(\tau,D_{k-1},L_{k-1}) \ar[rr]^{\omega_{i_{k-1} j_{k-1}}^{\epsilon_{k-1}} } \ar[d]_{p_k} & & (\tau', D_{k-1}', L_{k-1}') \ar[d]^{p_k'} \\
(\tau,D_k,L_k) \ar[rr]^{\omega_{i_k j_k}^{\epsilon_k} } & & (\tau', D_k', L_k')
}
$$
is contractible in $\mcal{C}_{\mcal{K}}(\wh{S})$. Then, combining these $n$ contractible loops, we obtain a loop
$$
\xymatrix{
(\tau,D_0,L_0) \ar[rr]^{\omega_{i_0 j_0}^{\epsilon_0} } \ar[d]_{p_n \cdots p_2 p_1} & & (\tau', D_0', L_0') \ar[d]^{p_n'  \cdots p_2'  p_1'} \\
(\tau,D,L) \ar[rr]^{\omega_{i j}^{\epsilon} } & & (\tau', D', L')
}
$$
which is contractible in $\mcal{C}_{\mcal{K}}(\wh{S})$, where $(\tau,D_0,L_0) \stackrel{ p_n \cdots p_1 }{ \longrightarrow } (\tau,D,L)$ is a path in $f^{-1}(\tau)$ and $(\tau', D_0', L_0') \stackrel{ p_n'  \cdots p_1' }{ \longrightarrow } (\tau', D', L')$ is a path in $f^{-1}(\tau')$, thus fulfilling condition (3) of Lem.\ref{lem:BaKi}.

\vs

Without loss of generality, we may assume $\epsilon_0=1$. Among the ideal triangles of $\tau$ and $\tau'$, call the ones that are not involved in the flip $\tau\to \tau'$ {\em generic triangles}, and the ones that are involved in the flip {\em special triangles}. So, for each of $\tau$ and $\tau'$, there are two special triangles. The generic triangles are ideal triangles for both $\tau$ and $\tau'$.

\vs

First, we write $L = \sigma \circ L_0$ for some permutation $\sigma$ of $I$. Observe that $\omega_{ij}^{\pm 1}$ does not change the labeling of the generic ideal triangles, so the two rules $L'$ and $\sigma \circ L_0'$ may be different only for the special ideal triangles. Define
$$
(\tau,D_1,L_1) := P_\sigma(\tau,D_0,L_0) = (\tau,D_0, L) \quad\mbox{and}\quad
(\tau',D_1',L_1') := P_\sigma(\tau',D_0',L_0') = (\tau', D_0', \sigma \circ L_0'),
$$
Then, for $i_1 := \sigma(i_0)$ and $j_1 := \sigma(j_0)$, the loop
$$
\xymatrix{
(\tau,D_0,L_0) \ar[rr]^{\omega_{i_0 j_0} } \ar[d]_{P_\sigma} & & (\tau', D_0', L_0') \ar[d]^{P_\sigma} \\
(\tau,D_1,L_1) = (\tau,D_0,L) \ar[rr]^-{\omega_{i_1 j_1} } & & (\tau', D_1', L_1') = (\tau',D_0',L_1')
}
$$
exists and is contractible in $\mcal{C}_{\mcal{K}}(\wh{S})$, because it is the boundary of the $2$-cell of type $P_\sigma \omega_{i_0 j_0} = \omega_{\sigma(i_0) \sigma(j_0)} P_\sigma$.

\vs

Now, note that $D$ and $D'$ are same on all generic triangles. Likewise for $D_0$ and $D_0'$. By applying a sequence of some $\rho_u$'s to $(\tau,D_0,L)$, we can get $(\tau,D,L)$. From this sequence, erase the ones involving the special triangles. Then we get
$$
\rho_{u_\ell}^{a_\ell} \cdots \rho_{u_1}^{a_1} (\tau,D_0,L) =: (\tau, \til{D}_0, L),
$$
where $\til{D}_0$ and $D$ are same on the generic triangles, where $u_1,\ldots,u_\ell$ are mutually distinct elements of $I$, none of them are $i_1$ or $j_1$ (i.e. the $L$-labels for the special triangles), and $a_1,\ldots,a_\ell \in \{\pm 1\}$. If we write
$$
\rho_{u_\ell}^{a_\ell} \cdots \rho_{u_1}^{a_1} (\tau', D_0', L_1') =: (\tau', \til{D}_0', L_1'),
$$
then we can show that $\til{D}_0'$ and $D'$ are the same on all generic triangles. First, observe that $L$ and $L_1'$, as well as $D_0$ and $D_0'$, are same on the generic triangles. By applying the dot change $\rho_{u_\ell}^{a_\ell} \cdots \rho_{u_1}^{a_1}$ to $D_0$ with respect to the labeling $L$, we get $\til{D}_0$, which is same as $D$ on generic triangles. Thus, the dotting rule $\til{D}_0'$, obtained by applying the dot change $\rho_{u_\ell}^{a_\ell} \cdots \rho_{u_1}^{a_1}$ to $D_0'$ with respect to the labeling $L_1'$, is same as $D$ on the generic triangles. Since $D$ and $D'$ are the same on the generic triangles, we get the desired conclusion.

\vs

Define
$$
(\tau, D_2, L_2) := \rho_{u_1}^{a_1} (\tau,D_1,L_1) = \rho_{u_1}^{a_1} (\tau,D_0,L), \quad
(\tau',D_2',L_2') := \rho_{u_1}^{a_1} (\tau',D_1',L_1') = \rho_{u_1}^{a_1}  (\tau',D_0',L_1').
$$
Then the loop
$$
\xymatrix{
(\tau,D_1,L_1) \ar[rr]^{\omega_{i_1 j_1} } \ar[d]_{\rho_{u_1}^{a_1}} & & (\tau', D_1', L_1') \ar[d]^{\rho_{u_1}^{a_1}} \\
(\tau,D_2,L_2) = \rho_{u_1}^{a_1}(\tau,D_0,L) \ar[rr]^-{\omega_{i_1 j_1} } & & (\tau', D_2', L_2') = \rho_{u_1}^{a_1}(\tau',D_0',L_1')
}
$$
exists and is contractible in $\mcal{C}_{\mcal{K}}(\wh{S})$, because it is the boundary of the $2$-cell of type $\omega_{i_1j_1} \rho_{u_1} = \rho_{u_1} \omega_{i_1j_1}$. Likewise, we define
$$
(\tau, D_{k+1},L_{k+1}) := \rho_{u_k}^{a_k} (\tau,D_k,L_k), \qquad
(\tau', D_{k+1}',L_{k+1}') := \rho_{u_k}^{a_k} (\tau',D_k',L_k')
$$
for each $k=1,2,\ldots,\ell$, and each time the similar loop exists and is contractible for a $2$-cell of similar type. We end up with the edge
$$
\xymatrix{
(\tau, D_{\ell+1},L_{\ell+1}) = (\tau, \til{D}_0,L) 
\ar[rr]^-{ \omega_{i_1 j_1} } &&
(\tau', D_{\ell+1}',L_{\ell+1}') = (\tau',\til{D}_0',L_1'),
}
$$
where $\til{D}_0$ is same as $D$ on generic triangles and as $D_0$ on special triangles, $\til{D}_0'$ is same as $D'$ on generic triangles, and $L_1'$ is same as $L$ on the generic triangles and hence also as $L'$ on generic triangles. So we finally have to form a loop
\begin{align}
\label{eq:loop_to_be_contractible}
\begin{array}{c}
\xymatrix{
(\tau, \til{D}_0,L) 
\ar[rr]^-{ \omega_{i_1 j_1} } \ar[d]_{p_{\ell+2}} & &
(\tau',\til{D}_0',L_1') \ar[d]^{p_{\ell+2}'} \\
(\tau, D,L) \ar[rr]^-{ \omega_{ij}^\epsilon } & & (\tau',D',L')
}
\end{array}
\end{align}
and prove that it is contractible in $\mcal{C}_{\mcal{K}}(\wh{S})$. We can now concentrate on the special triangles.

\vs

\begin{figure}[htbp!]
\includegraphics[width=55mm]{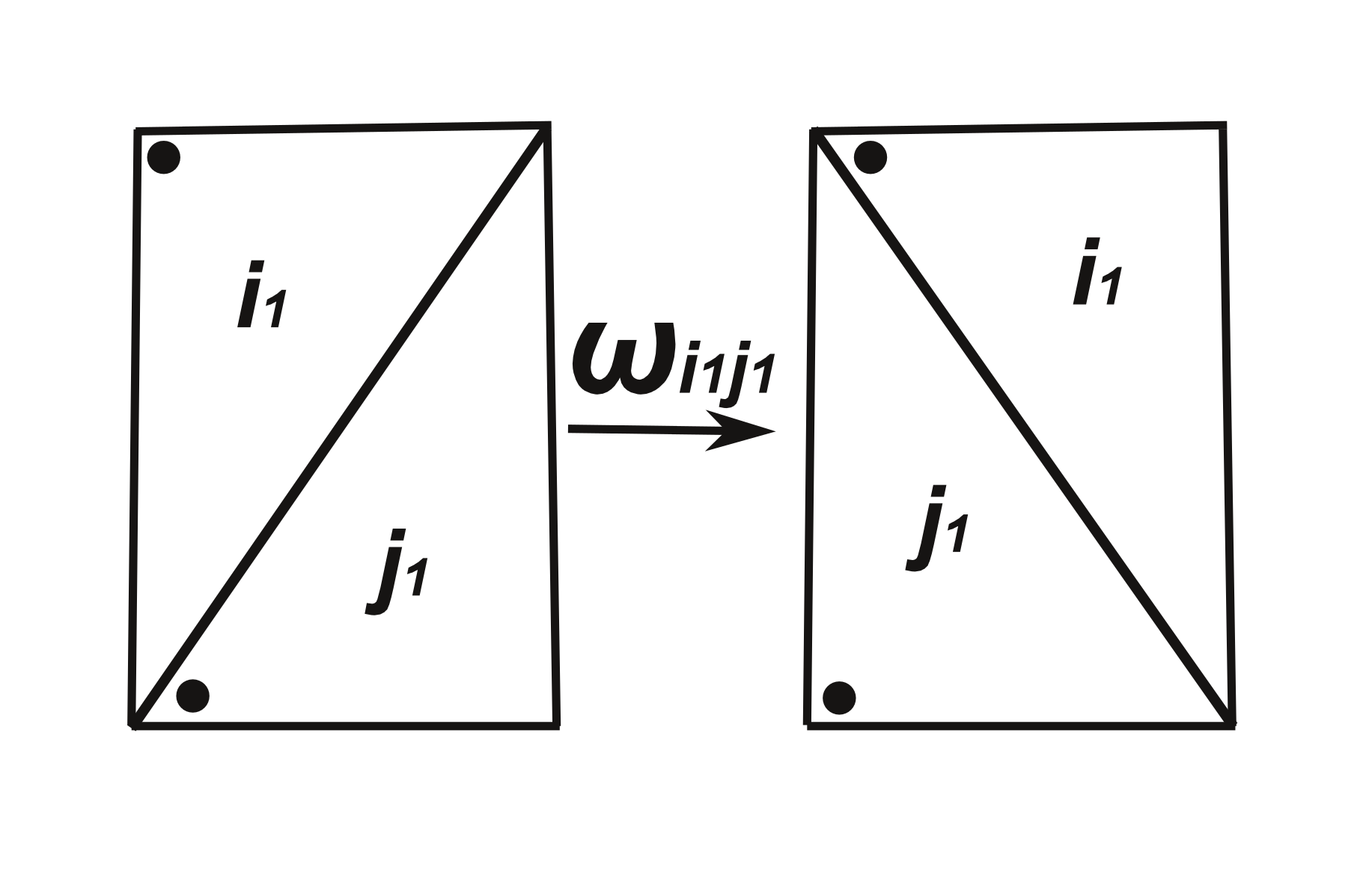}
\vspace{-7mm}
\caption{picture for $\omega_{i_1 j_1}$, case 1 for $\omega_{ij}^\epsilon$}
\label{fig:omega1}
\end{figure}

\vspace{-2mm}

What happens to the special triangles for the top horizontal edge $(\tau, \til{D}_0,L) \stackrel{ \omega_{i_1 j_1} }{\longrightarrow} (\tau',\til{D}_0',L_1')$ of \eqref{eq:loop_to_be_contractible} is determined as in Fig.\ref{fig:omega1}. 
Then, one can see that there are only four possibilities for the edge $(\tau,D,L) \stackrel{\omega_{ij}^\epsilon}{\longrightarrow} (\tau',D',L')$.

\vs

{\em Case 1.} The edge $(\tau,D,L) \stackrel{\omega_{ij}^\epsilon}{\longrightarrow} (\tau',D',L')$ is as in Fig.\ref{fig:omega1}. Then the edges $(\tau, \til{D}_0,L) \stackrel{ \omega_{i_1 j_1} }{\longrightarrow} (\tau',\til{D}_0',L_1')$ and $(\tau,D,L) \stackrel{\omega_{ij}^\epsilon}{\longrightarrow} (\tau',D',L')$ coincide with each other, so end of proof.

\begin{figure}[htbp!]
\begin{subfigure}[b]{0.34\textwidth}
\hspace{-3mm} \includegraphics[width=55mm]{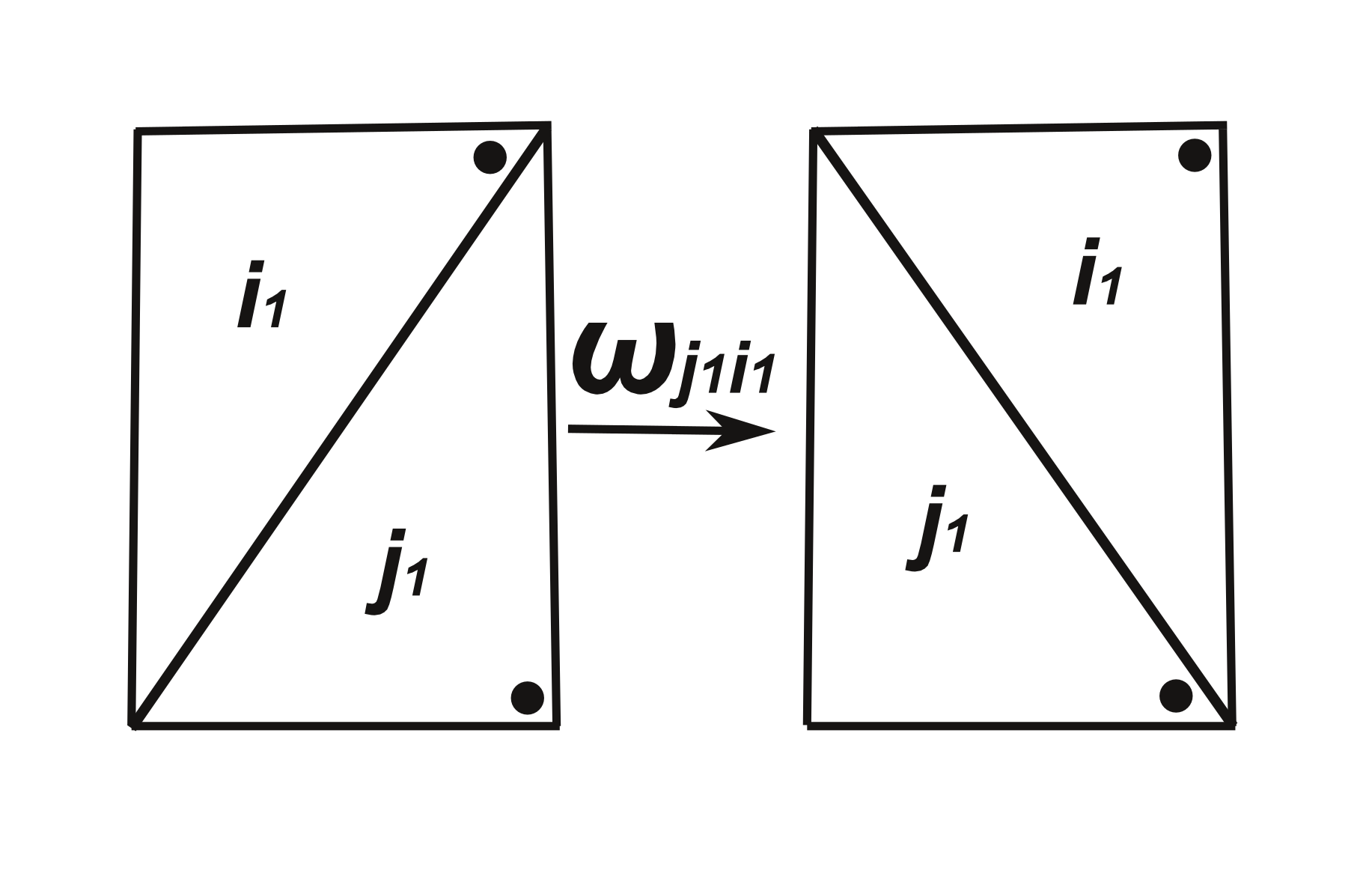}
\vspace{-10mm}
\caption{case 2 : $\omega_{ij}^\epsilon = \omega_{j_1i_1}$}
\label{fig:omega2}
\end{subfigure}
\hfill
\begin{subfigure}[b]{0.33\textwidth}
\hspace{-3mm} \includegraphics[width=55mm]{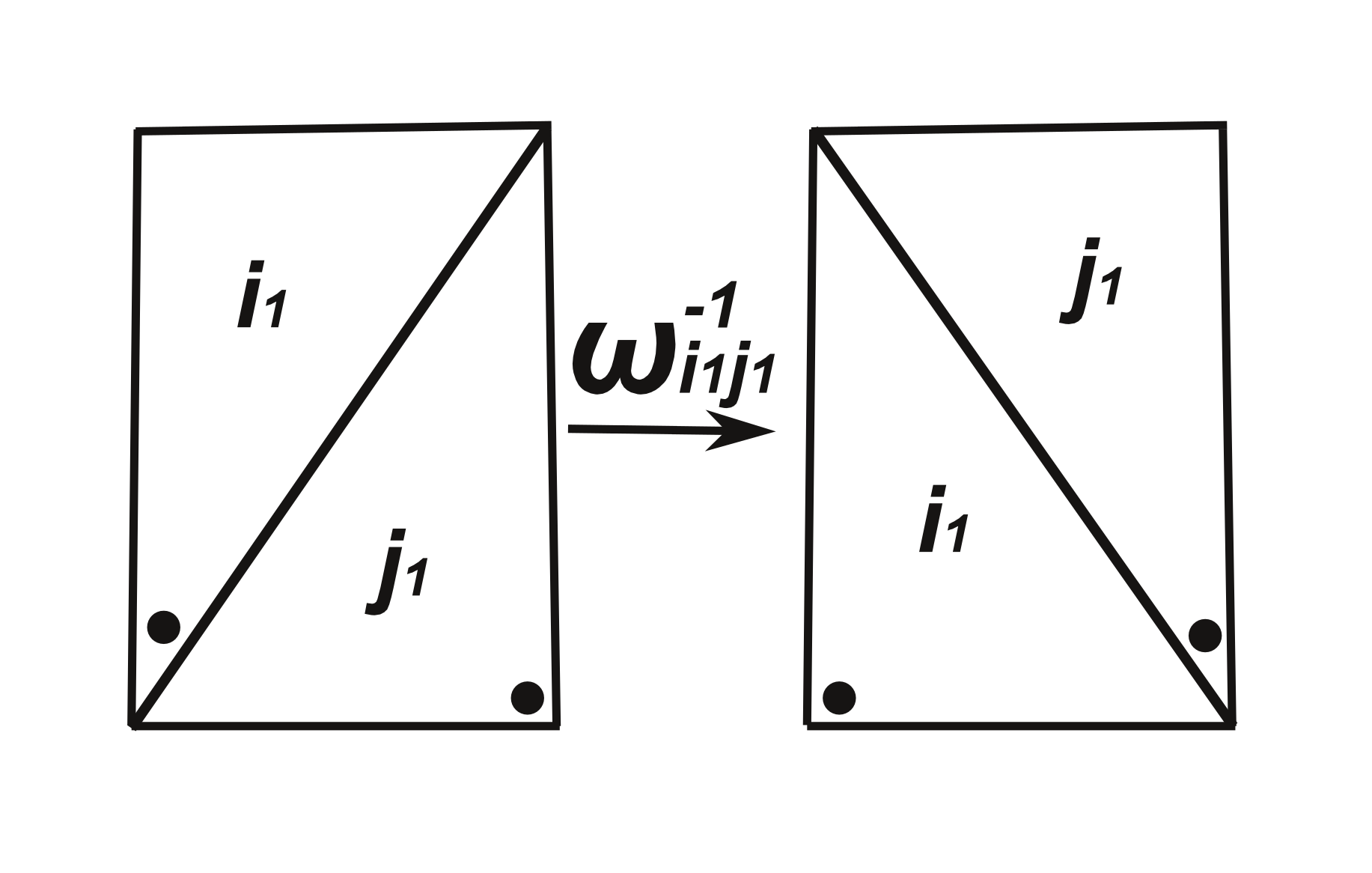}
\vspace{-10mm}
\caption{case 3 : $\omega_{ij}^\epsilon = \omega_{i_1 j_1}^{-1}$}
\label{fig:omega3}
\end{subfigure}   
\hfill
\begin{subfigure}[b]{0.30\textwidth}
\includegraphics[width=55mm]{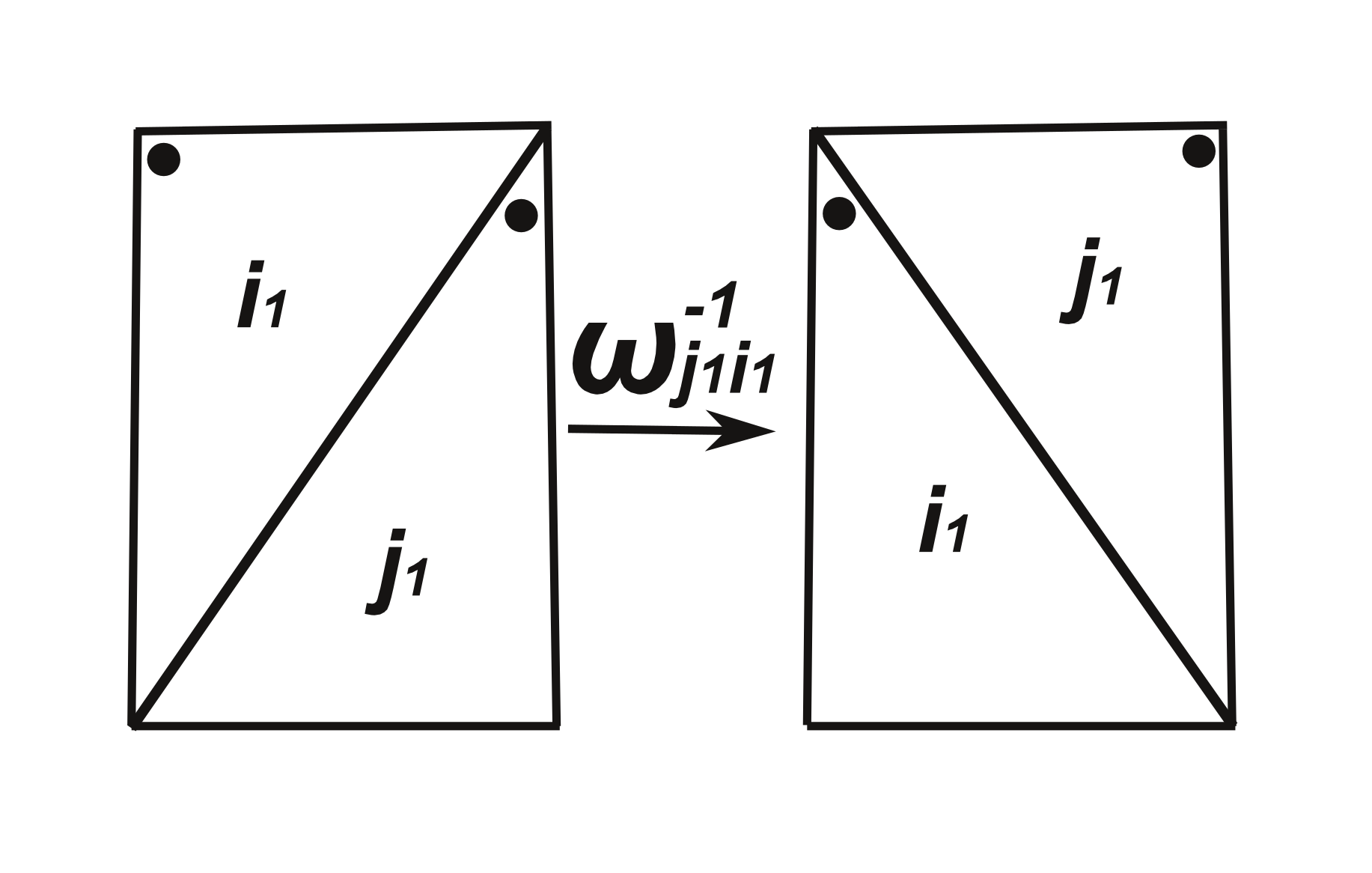}
\vspace{-10mm}
\caption{case 4 : $\omega_{ij}^\epsilon = \omega_{j_1 i_1}^{-1}$}
\label{fig:omega4}
\end{subfigure}   
\hfill 
\vspace{-2mm}
\caption{Remaining three possibilities for $\omega_{ij}^\epsilon$}
\label{fig:remaining_possibilities_for_omega_ij_epsilon}
\end{figure}

\vspace{-2mm}

{\em Case 2.} The edge $(\tau,D,L) \stackrel{\omega_{ij}^\epsilon}{\longrightarrow} (\tau',D',L')$ is as in Fig.\ref{fig:omega2}.
Then, in view of Fig.\ref{fig:omega1}, we have $(\tau,D,L) = \rho_{j_1} \rho_{i_1}^{-1} (\tau, \til{D}_0, L)$ and $(\tau', D', L') = \rho_{j_1} \rho_{i_1}^{-1} (\tau', \til{D}_0', L_1')$, while $\omega_{ij}^\epsilon = \omega_{j_1i_1}$, so we have the following loop
$$
\xymatrix{
(\tau, \til{D}_0,L) 
\ar[rr]^-{ \omega_{i_1 j_1} } \ar[d]_{\rho_{j_1} \rho_{i_1}^{-1} } & &
(\tau',\til{D}_0',L_1') \ar[d]^{\rho_{j_1} \rho_{i_1}^{-1} } \\
(\tau, D,L) \ar[rr]^-{ \omega_{j_1i_1} } & & (\tau',D',L'),
}
$$
which is contractible because it is the boundary of the $2$-cell of type $\rho_{j_1} \omega_{i_1j_1} \rho_{i_1} = \rho_{i_1} \omega_{j_1i_1} \rho_{j_1}$.

\vs

{\em Case 3.} The edge $(\tau,D,L) \stackrel{\omega_{ij}^\epsilon}{\longrightarrow} (\tau',D',L')$ is as in Fig.\ref{fig:omega3}. Then, in view of Fig.\ref{fig:omega1}, we have $(\tau,D,L) = \rho_{j_1} \rho_{i_1} (\tau, \til{D}_0,L)$ and $(\tau',D',L') = P_{(i_1 \, j_1)} \rho_{i_1} (\tau', \til{D}_0', L_1')$ (where $(i_1~j_1)$ is the transposition), while $\omega_{ij}^\epsilon = \omega_{i_1 j_1}^{-1}$. In order to cook up a loop \eqref{eq:loop_to_be_contractible} for this case and prove that it is contractible in $\mcal{C}_{\mcal{K}}(\wh{S})$, we will use two $2$-cells of $\mcal{C}_{\mcal{K}}(\wh{S})$. Consider the loop
$$
\xymatrix{
(\tau, \til{D}_0,L) 
\ar[rr]^-{ \omega_{i_1 j_1} } \ar[d]_{\rho_{j_1} \rho_{i_1} P_{(j_1\, i_1)} } & &
(\tau',\til{D}_0',L_1') \ar[d]^{ \rho_{i_1} } \\
(\tau, D,(i_1~ j_1) \circ L) \ar[rr]^-{ \omega_{j_1i_1}^{-1} } & & (\tau',D', (i_1~ j_1) \circ L'),
}
$$
whose bottom horizontal edge can be depicted by the same picture as Fig.\ref{fig:omega3} with $i_1$ and $j_1$ exchanged everywhere. This loop is contractible because it is the boundary of the $2$-cell of type $\omega_{j_1i_1} \rho_{i_1} \omega_{i_1j_1} = \rho_{j_1} \rho_{i_1} P_{(j_1 \, i_1)}$. Now consider the loop
$$
\xymatrix{
(\tau, D,(i_1~ j_1) \circ L) \ar[rr]^-{ \omega_{j_1i_1}^{-1} } \ar[d]_{P_{(j_1\, i_1)} } & & (\tau',D', (i_1~ j_1) \circ L') \ar[d]^{P_{(i_1\, j_1)}} \\
(\tau,D,L) \ar[rr]^-{\omega_{i_1 j_1}^{-1}} & & (\tau', D', L')
},
$$
which is contractible in $\mcal{C}_{\mcal{K}}(\wh{S})$ because it is the boundary of the $2$-cell of type $P_{\sigma_1}  \omega_{i_1 j_1} = \omega_{\sigma_1(i_1) \sigma_1(j_1)} P_{\sigma_1}$ for $\sigma_1 = (i_1~ j_1)$. Therefore, combining these two cells, we have a loop
$$
\xymatrix{
(\tau, \til{D}_0,L) 
\ar[rr]^-{ \omega_{i_1 j_1} } \ar[d]_{P_{(i_1\, j_1)} \rho_{j_1} \rho_{i_1} P_{(j_1\, i_1)} } & &
(\tau',\til{D}_0',L_1') \ar[d]^{ P_{(i_1\, j_1)} \rho_{i_1} } \\
(\tau, D,L) \ar[rr]^-{ \omega_{i_1j_1}^{-1} } & & (\tau',D',L').
}
$$
which is contractible in $\mcal{C}_{\mcal{K}}(\wh{S})$.

\vs

{\em Case 4.} The edge $(\tau,D,L) \stackrel{\omega_{ij}^\epsilon}{\longrightarrow} (\tau',D',L')$ is as in Fig.\ref{fig:omega4}. Then, in view of Fig.\ref{fig:omega1}, we have $(\tau,D,L) = \rho_{j_1}^{-1} (\tau, \til{D}_0, L)$ and $(\tau',D',L') =  \rho_{j_1} \rho_{i_1} P_{(i_1\, j_1)} (\tau', \til{D}_0', L_1')$, while $\omega_{ij}^\epsilon = \omega_{j_1 i_1}^{-1}$, so we have the following loop
$$
\xymatrix{
(\tau, \til{D}_0,L) 
\ar[rr]^-{ \omega_{i_1 j_1} } \ar[d]_{\rho_{j_1}^{-1} } & &
(\tau',\til{D}_0',L_1') \ar[d]^{ \rho_{j_1} \rho_{i_1} P_{(i_1\, j_1)}} \\
(\tau, D,L) \ar[rr]^-{ \omega_{j_1i_1}^{-1} } & & (\tau',D',L'),
}
$$
which is contractible in $\mcal{C}_{\mcal{K}}(\wh{S})$ because it is the boundary of the $2$-cell of type $\omega_{i_1 j_1} \rho_{j_1} \omega_{j_1 i_1} = \rho_{j_1} \rho_{i_1} P_{(i_1\, j_1)}$.

\end{proof}

\subsection{Proof of the consistency of quantization for $m=3$}
\label{subsec:consistency_condition_for_3}

Let $m=3$. Let us first define the expressions $\rho_h(\til{A}_t)$, $\rho_h(\til{T}_{ts})$, $\rho_h(\til{P}_\sigma)$ as follows:
\begin{align}
\label{eq:3_Kashaev_operators}
\left\{ {\renewcommand{\arraystretch}{1.2} \begin{array}{l}
\displaystyle  \rho_h(\til{A}_t) := {\bf P}_{\sigma^A_t} \, {\bf A}_{t_{1,1}} {\bf A}_{t_{2,1}} {\bf A}_{t_{2,2}}, \\
\rho_h(\til{T}_{ts}) := {\bf T}_{t_{2,1} s_{2,1}} {\bf T}_{t_{2,2} s_{1,1}} {\bf F}_{s_{2,1} t_{2,2}} {\bf T}_{t_{1,1} s_{2,1}} {\bf T}_{t_{2,2} s_{2,2}}, \\
\displaystyle  \rho_h(\til{P}_\sigma):= {\bf P}_{\sigma_{1,1}} {\bf P}_{\sigma_{2,1}} {\bf P}_{\sigma_{2,2}},
\end{array} } \right.
\end{align}
where $\mcal{S}_3 = \{(1,1),(2,1),(2,2)\}$ (see \eqref{eq:mcal_S_m}), the expressions ${\bf A}_{\cdot}$, ${\bf T}_{\cdot \cdot}$, ${\bf F}_{\cdot \cdot}$ are as defined by the formulas \eqref{eq:Kashaev_operators} and \eqref{eq:rho_h_F_rs} for $Q_s = \frac{1}{2\pi b} \wh{q}_s = \frac{1}{2\pi b} \log \wh{Y}_s$'s and $P_s = \frac{1}{2\pi b} \wh{p}_s = \frac{1}{2\pi b} \log \wh{Z}_s$'s, the permutations $\sigma^A_t$ and $\sigma_{r,c}$ are as described in Lem.\ref{lem:functor_from_Kashaev_groupoid_to_Qdt}, and we regard ${\bf P}_\cdot$ as a certain expression in $Q_s$'s and $P_s$'s satisfying \eqref{eq:P_conjugation_on_P_and_Q}; we recall the readers that $h = 2\pi b^2$.

\vs

The consistency condition \eqref{eq:consistency_condition} means that these expressions $\rho_h(\til{A}_t)$, $\rho_h(\til{T}_{ts})$, $\rho_h(\til{P}_\sigma)$ defined above satisfy all the relations in Lem.\ref{lem:relations_of_elementary_moves_of_Kashaev_groupoid} up to multiplicative constants. We focus on the first four relations, as others are trivially checked. In the present subsection, for convenience we replace the elements $(1,1), (2,1), (2,2)$ of $\mcal{S}_3$ by $1,2,3$; for example, the triangles labeled by $t_{1,1}$, $t_{2,1}$, $t_{2,2}$ are now labeled by the symbols $t_1$, $t_2$, $t_3$, respectively. Thus we rewrite \eqref{eq:3_Kashaev_operators} as
\begin{align}
\label{eq:3_Kashaev_operators2}
\left\{ {\renewcommand{\arraystretch}{1.2} \begin{array}{l}
\displaystyle  \rho_h(\til{A}_t) := {\bf P}_{(t_1 t_3 t_2)} {\bf A}_{t_1} {\bf A}_{t_2} {\bf A}_{t_3}, \\
\rho_h(\til{T}_{ts}) := {\bf T}_{t_3  s_1} {\bf T}_{t_2 s_2} {\bf F}_{s_2 t_3}  {\bf T}_{t_3 s_3} {\bf T}_{t_1 s_2}, \\
\displaystyle  \rho_h(\til{P}_\sigma):= {\bf P}_{\sigma_1} {\bf P}_{\sigma_2} {\bf P}_{\sigma_3},
\end{array} } \right.
\end{align}
where $\sigma_1,\sigma_2,\sigma_3$ should be understood in an obvious appropriate sense, and we used \eqref{eq:lifted_Kashaev_relations_commutation} for $\rho_h(\til{T}_{ts})$. We must show that
\begin{align}
\label{eq:to_prove_3}
\left\{ {\renewcommand{\arraystretch}{1.4}
\begin{array}{l}
\rho_h(\til{A}_t)^3={\rm id}, \quad \rho_h(\til{A}_t) \rho_h(\til{T}_{ts}) \rho_h(\til{A}_s) = \rho_h(\til{A}_s) \rho_h(\til{T}_{st}) \rho_h(\til{A}_t), \\
\rho_h(\til{T}_{ts}) \rho_h(\til{T}_{tu}) \rho_h(\til{T}_{su}) = \rho_h(\til{T}_{su}) \rho_h(\til{T}_{ts}), \quad
\rho_h(\til{T}_{ts}) \rho_h(\til{A}_t) \rho_h(\til{T}_{st}) = \, \rho_h(\til{A}_t) \rho_h(\til{A}_s) \rho_h(\til{P}_{(t \, s)})
\end{array}} \right.
\end{align}
hold up to multiplicative constants. However, these expressions are  formal ones only. To be more precise, what we want to show is that for each equation ${\rm LHS} = {\rm RHS}$, the conjugation by ${\rm LHS}$ and that by ${\rm RHS}$ yield the same result on all $Q_s$'s and $P_s$'s.
\begin{definition}
\label{def:equal_as_conjugators}
Two expressions (i.e. functions) in $Q_s$'s and $P_s$'s, $s\in I$, are said to be {\em equal as conjugators}, if the conjugation by the two on every $Q_s$ and $P_s$ agree.
\end{definition}
Thus we would like to show that all equations in \eqref{eq:to_prove_3} hold as conjugators. It is straightforward how to prove this; we just apply conjugation by each expression on all $Q_s$'s and $P_s$'s. For this we use Lem.\ref{lem:conjugation_action} and Lem.\ref{lem:conjugation_by_F_rs}, which tell us how ${\bf A}_\cdot$, ${\bf T}_{\cdot \cdot}$, ${\bf P}_\cdot$, ${\bf F}_{\cdot \cdot}$ act by conjugation on $\wh{Y}_s = e^{2\pi b Q_s}$'s and $\wh{Z}_s = e^{2\pi b P_s}$'s. This information is in fact not sufficient to obtain equalities as conjugators, defined in Def.\ref{def:equal_as_conjugators}. Here we use the $b\leftrightarrow b^{-1}$ symmetry of $\Psi_b$; namely, $\Psi_b = \Psi_{b^{-1}}$. So, the results of Lemmas \ref{lem:conjugation_action} and \ref{lem:conjugation_by_F_rs} still hold if we replace all $\wh{Y}_s$'s and $\wh{Z}_s$'s by
$$
\wh{\wh{Y}}_s := \wh{Y}_s^{1/b^2} = e^{2\pi b^{-1} Q_s} \quad\mbox{and}\quad
\wh{\wh{Z}}_s := \wh{Z}_s^{1/b^2} = e^{2\pi b^{-1} P_s},
$$
and $q = e^{\pi i b^2}$ by $\wh{q} = e^{\pi i b^{-2}}$. 

\begin{definition}
\label{def:weakly_equal_as_conjugators}
Two expressions (i.e. functions) in $Q_s$'s and $P_s$'s, $s\in I$, are said to be {\em weakly equal as conjugators}, if the conjugation by the two on every $\wh{Y}_s$, $\wh{\wh{Y}}_s$, $\wh{Z}_s$, $\wh{\wh{Z}}_s$ agree.
\end{definition}

What we actually claim to hold is the following:
\begin{proposition}
\label{prop:weakly_hold}
All equations in \eqref{eq:to_prove_3} weakly hold as conjugators.
\end{proposition}
When $m=2$, it suffices to have Prop.\ref{prop:weakly_hold}. In this case, $Q_s$ and $P_s$ are realized as concrete operators on a specific Hilbert space, and it is well known that this representation \eqref{eq:star_representation} is `strongly irreducible', in the sense that any (bounded) operator on this Hilbert space commuting with $Q_s$ and $P_s$ is a scalar operator. Therefore, equality as (bounded) conjugators imply equality up to a multiplicative constant. Moreover, in this case, the algebra generated by $\wh{\wh{Y}}_s$ and $\wh{\wh{Z}}_s$ is the `modular double' \cite{F} counterpart of the `quantum plane' algebra generated by $\wh{Y}_s$ and $\wh{Z}_s$ \cite{FrKi}. The modular double representation induced by \eqref{eq:star_representation} is strongly irreducible, so that any equation of unitary operators which holds weakly as conjugators genuinely holds up to a multiplicative constant.

\vs

For $m=3$, using Lemmas \ref{lem:conjugation_action} and \ref{lem:conjugation_by_F_rs}, their `$b^{-1}$ versions', together with \eqref{eq:P_conjugation_on_P_and_Q}, it is a tedious job to prove Prop.\ref{prop:weakly_hold}, and it quickly requires much space to write all computation down. As an alternative shorter proof, we can manipulate the expressions $\rho_h(\til{A}_\cdot)$, $\rho_h(\til{T}_{\cdot \cdot})$, $\rho_h(\til{P}_\cdot)$ directly. In such a proof, what we use are the relations in Prop.\ref{prop:lifted_Kashaev_relations}, which should now be understood as (weak) equalities as conjugators. It is also useful to have some more equalities like ${\bf A}_r {\bf F}_{sr} {\bf A}_s = {\bf A}_s {\bf F}_{rs} {\bf A}_r$ or $\Psi_b(P)\Psi_b(Q) = \Psi_b(Q) \Psi_b(P+Q) \psi_b(P)$ (for $[P,Q] = \frac{1}{2\pi i}$), which are also (weak) equalities as conjugators. We present a proof of one of the sought-for equations here, which is special because the proof requires the linear relation of $Q_s$'s and $P_s$'s, while other equations do not.
\begin{lemma}[proof of the inversion relation, as conjugators]
One has
$$
\rho_h(\til{T}_{ts}) \rho_h(\til{A}_t) \rho_h(\til{T}_{st}) =  \rho_h(\til{A}_t) \rho_h(\til{A}_s) \rho_h(\til{P}_{(t \, s)})
$$
as conjugators. The proof depends on the following linear relation of $Q_\cdot$'s and $P_\cdot$'s:
\begin{align}
\label{eq:the_crucial_linear_relation}
P_{s_1} + Q_{s_3} - P_{s_3} - Q_{t_3} + P_{t_2} = 0,
\end{align}
coming from the `small middle diamond' in the quadrilateral formed by ideal triangles $t$ and $s$ of any dotted triangulation to which $\til{T}_{st}$ can be applied; see the dotted loop in Fig.\ref{fig:linearrel}.
\end{lemma}

\begin{proof}
Note that
\begin{align*}
& \rho_h(\til{T}_{ts}) \rho_h( \til{A}_t) \rho_h( \til{T}_{st})  \\
& \hspace{-2mm} \stackrel{\eqref{eq:3_Kashaev_operators2}}{=} ({\bf T}_{t_3 s_1} {\bf T}_{t_2 s_2} e^{-2\pi i Q_{s_2} P_{t_3} } {\bf T}_{t_3 s_3} {\bf T}_{t_1 s_2})
(\ulu{ {\bf P}_{(t_1 t_3 t_2)} }{{\rm to~left}} {\bf A}_{t_1} {\bf A}_{t_2} {\bf A}_{t_3}) ({\bf T}_{s_3 t_1} {\bf T}_{s_2 t_2} e^{-2\pi i Q_{t_2} P_{s_3} } {\bf T}_{s_3 t_3} {\bf T}_{s_1 t_2}) \\
& \hspace{-2.5mm} \stackrel{\eqref{eq:lifted_Kashaev_relations_P}}{=} ( {\bf P}_{(t_1 t_3 t_2)} )
{\bf T}_{t_1 s_1} {\bf T}_{t_3 s_2} e^{-2\pi i Q_{s_2} P_{t_1}} {\bf T}_{t_1 s_3} \ul{ {\bf T}_{t_2 s_2} }
{\bf A}_{t_1} \ul{ {\bf A}_{t_2} } {\bf A}_{t_3} 
{\bf T}_{s_3 t_1} \ul{ {\bf T}_{s_2 t_2} } e^{-2\pi i Q_{t_2} P_{s_3}} {\bf T}_{s_3 t_3} {\bf T}_{s_1 t_2}
\end{align*}

\begin{align*}
& \hspace{-6.5mm} \stackrel{\eqref{eq:lifted_Kashaev_relations_commutation}, \eqref{eq:lifted_Kashaev_relations_major}}{=} \zeta {\bf P}_{(t_1 t_3 t_2)} 
{\bf T}_{t_1 s_1} {\bf T}_{t_3 s_2} e^{-2\pi i Q_{s_2} P_{t_1}} \ul{ {\bf T}_{t_1 s_3}  }
\ul{ {\bf A}_{t_1} }
({\bf A}_{t_2} {\bf A}_{s_2} {\bf P}_{(t_2 s_2)})
{\bf A}_{t_3} 
\ul{ {\bf T}_{s_3 t_1} } e^{-2\pi i Q_{t_2} P_{s_3}} {\bf T}_{s_3 t_3} {\bf T}_{s_1 t_2} \\
& \hspace{-6.5mm} \stackrel{\eqref{eq:lifted_Kashaev_relations_commutation}, \eqref{eq:lifted_Kashaev_relations_major}}{=} \zeta^2 {\bf P}_{(t_1 t_3 t_2)} 
{\bf T}_{t_1 s_1} {\bf T}_{t_3 s_2} e^{-2\pi i Q_{s_2} P_{t_1}}
({\bf A}_{t_1} {\bf A}_{s_3} \ulu{ {\bf P}_{(t_1 s_3)} }{{\rm to~left}})
{\bf A}_{t_2} {\bf A}_{s_2} {\bf P}_{(t_2  s_2)}
{\bf A}_{t_3} 
e^{-2\pi i Q_{t_2} P_{s_3}} {\bf T}_{s_3 t_3} {\bf T}_{s_1 t_2} \\
& \hspace{-2.5mm} \stackrel{\eqref{eq:lifted_Kashaev_relations_P}}{=}  \zeta^2 {\bf P}_{(t_1 t_3 t_2)} ({\bf P}_{(t_1 s_3)} )
{\bf T}_{s_3 s_1} {\bf T}_{t_3 s_2} e^{-2\pi i Q_{s_2} P_{s_3}}
{\bf A}_{s_3} {\bf A}_{t_1} 
{\bf A}_{t_2} {\bf A}_{s_2} \ulu{ {\bf P}_{(t_2 s_2)} }{{\rm to~left}}
{\bf A}_{t_3}
e^{-2\pi i Q_{t_2} P_{s_3}} {\bf T}_{s_3 t_3} {\bf T}_{s_1 t_2} \\
& \hspace{-2.5mm} \stackrel{\eqref{eq:lifted_Kashaev_relations_P}}{=}  \zeta^2 {\bf P}_{(t_1 t_3 t_2)} {\bf P}_{(t_1 s_3)}  ({\bf P}_{(t_2  s_2)} )
{\bf T}_{s_3 s_1} {\bf T}_{t_3 t_2} e^{-2\pi i Q_{t_2} P_{s_3}}
{\bf A}_{s_3} \ulu{ {\bf A}_{t_1} 
{\bf A}_{s_2} }{{\rm to~left}} {\bf A}_{t_2} 
{\bf A}_{t_3}
\ulu{ e^{-2\pi i Q_{t_2} P_{s_3}} }{{\rm to~right}} {\bf T}_{s_3 t_3} {\bf T}_{s_1 t_2} \\
& \hspace{-10mm} \stackrel{\eqref{eq:lifted_Kashaev_relations_commutation},\eqref{eq:3_Kashaev_operators2},\eqref{eq:basic_commutation}}{=} \zeta^2 {\bf P}_{(t_1 t_3 t_2)} {\bf P}_{(t_1 s_3)} {\bf P}_{(t_2 s_2)}  ({\bf A}_{t_1} {\bf A}_{s_2} )
{\bf T}_{s_3 s_1} {\bf T}_{t_3 t_2} e^{-2\pi i Q_{t_2} P_{s_3}}
{\bf A}_{s_3}
{\bf A}_{t_2}
{\bf A}_{t_3} \\
& \qquad \quad \cdot (\ulu{ e^{2\pi i Q_{t_3} P_{s_3}} }{{\rm to~right}} \Psi_b(Q_{s_3} + P_{t_3} - Q_{t_3} +(- Q_{t_2}) )^{-1} e^{-2\pi i Q_{t_2} P_{s_3}} )  {\bf T}_{s_1 t_2} \\
& \hspace{-6.5mm} \stackrel{\eqref{eq:basic_commutation},\eqref{eq:lifted_Kashaev_relations_P}}{=} \zeta^2 {\bf P}_{(t_1 t_3 t_2) (t_1 s_3) (t_2 s_2)}  {\bf A}_{t_1} {\bf A}_{s_2} 
{\bf T}_{s_3 s_1} {\bf T}_{t_3 t_2} e^{-2\pi i Q_{t_2} P_{s_3}}
{\bf A}_{s_3}
\ulu{ {\bf A}_{t_2} }{{\rm to~right}}
{\bf A}_{t_3} \\
& \qquad \cdot (\Psi_b(Q_{s_3} + P_{t_3} - \cancel{ Q_{t_3}  }- Q_{t_2} + ( \cancel{ Q_{t_3} } - P_{s_3}) )^{-1} e^{2\pi i Q_{t_3} P_{s_3}} )  e^{-2\pi i Q_{t_2} P_{s_3}}  {\bf T}_{s_1 t_2} \\
& \hspace{-2mm} \stackrel{\eqref{eq:A_commutation}}{=} \zeta^2 {\bf P}_{(t_1 t_3 t_2) (t_1 s_3) (t_2 s_2)}  {\bf A}_{t_1} {\bf A}_{s_2} 
{\bf T}_{s_3 s_1} {\bf T}_{t_3 t_2} e^{-2\pi i Q_{t_2} P_{s_3}}
\ulu{ {\bf A}_{s_3} }{{\rm to~right}}
{\bf A}_{t_3} \\
& \quad \cdot (\Psi_b(Q_{s_3} + P_{t_3}  + (Q_{t_2} -P_{t_2} ) - P_{s_3} )^{-1} {\bf A}_{t_2}  )
e^{2\pi i Q_{t_3} P_{s_3}} e^{-2\pi i Q_{t_2} P_{s_3}}  {\bf T}_{s_1 t_2}, \\
& \hspace{-2mm} \stackrel{\eqref{eq:A_commutation}}{=} \zeta^2 {\bf P}_{(t_1 t_3 t_2) (t_1 s_3) (t_2 s_2)}  {\bf A}_{t_1} {\bf A}_{s_2} 
{\bf T}_{s_3 s_1} {\bf T}_{t_3 t_2} \ulu{ e^{-2\pi i Q_{t_2} P_{s_3}} }{{\rm to~right}}
{\bf A}_{t_3} \\
& \quad \cdot (\Psi_b((P_{s_3} ) + P_{t_3} + Q_{t_2} -P_{t_2}   )^{-1}   {\bf A}_{s_3} )
{\bf A}_{t_2} e^{2\pi i Q_{t_3} P_{s_3}} e^{-2\pi i Q_{t_2} P_{s_3}}  {\bf T}_{s_1 t_2} \\
& \hspace{-2.5mm} \stackrel{\eqref{eq:basic_commutation}}{=} \zeta^2 {\bf P}_{(t_1 t_3 t_2) (t_1 s_3) (t_2 s_2)}  {\bf A}_{t_1} {\bf A}_{s_2} 
{\bf T}_{s_3 s_1} {\bf T}_{t_3 t_2} 
{\bf A}_{t_3} \cdot \ul{ 1  } \\
& \quad \cdot (\Psi_b( \cancel{ P_{s_3}  } + P_{t_3} + Q_{t_2} -P_{t_2}  +(- \cancel{ P_{s_3} } ) )^{-1}   e^{-2\pi i Q_{t_2} P_{s_3}}  )
{\bf A}_{s_3} {\bf A}_{t_2} e^{2\pi i Q_{t_3} P_{s_3}} e^{-2\pi i Q_{t_2} P_{s_3}}  {\bf T}_{s_1 t_2} \\
& = \zeta^2 {\bf P}_{(t_1 t_3 t_2) (t_1 s_3) (t_2 s_2)}  {\bf A}_{t_1} {\bf A}_{s_2} 
{\bf T}_{s_3 s_1} {\bf T}_{t_3 t_2} 
{\bf A}_{t_3} (e^{2\pi i Q_{t_3} P_{t_2}} \ulu{ e^{-2\pi i Q_{t_3} P_{t_2}}  }{{\rm to~right}})  \\
& \quad \cdot \Psi_b( P_{t_3} + Q_{t_2} -P_{t_2}   )^{-1}   
e^{-2\pi i Q_{t_2} P_{s_3}}   {\bf A}_{s_3} {\bf A}_{t_2} e^{2\pi i Q_{t_3} P_{s_3}} e^{-2\pi i Q_{t_2} P_{s_3}}  {\bf T}_{s_1 t_2} \\
& \hspace{-2.5mm} \stackrel{\eqref{eq:basic_commutation}}{=} \zeta^2 {\bf P}_{(t_1 t_3 t_2) (t_1 s_3) (t_2 s_2)}  {\bf A}_{t_1} {\bf A}_{s_2} 
{\bf T}_{s_3 s_1} \\
& \qquad \cdot \ul{ {\bf T}_{t_3 t_2} 
{\bf A}_{t_3} \underbrace{ e^{2\pi i Q_{t_3} P_{t_2}}  ( \Psi_b( P_{t_3} + Q_{t_2} - \cancel{P_{t_2}} +(  - Q_{t_3} + \cancel{P_{t_2}} ) )^{-1}  }_{={\bf T}_{t_2t_3}} } \\
& \qquad \cdot     e^{-2\pi i Q_{t_3} P_{t_2}}  ) 
e^{-2\pi i Q_{t_2} P_{s_3}}   {\bf A}_{s_3} {\bf A}_{t_2} e^{2\pi i Q_{t_3} P_{s_3}} e^{-2\pi i Q_{t_2} P_{s_3}}  {\bf T}_{s_1 t_2} \\
& \hspace{-2.5mm} \stackrel{\eqref{eq:lifted_Kashaev_relations_major}}{=} \zeta^3 {\bf P}_{(t_1 t_3 t_2) (t_1 s_3) (t_2 s_2)}  {\bf A}_{t_1} {\bf A}_{s_2} 
{\bf T}_{s_3 s_1}
(\ulu{ {\bf A}_{t_3} {\bf A}_{t_2} {\bf P}_{(t_3 t_2)} }{{\rm to~left}} ) \\
& \quad \cdot e^{-2\pi i Q_{t_3} P_{t_2}}  
e^{-2\pi i Q_{t_2} P_{s_3}}  \ulu{  {\bf A}_{s_3} }{{\rm to~left}} {\bf A}_{t_2} e^{2\pi i Q_{t_3} P_{s_3}} \ulu{ e^{-2\pi i Q_{t_2} P_{s_3}}  }{{\rm to~right}} {\bf T}_{s_1 t_2} \\
& \hspace{-13mm} \stackrel{\eqref{eq:lifted_Kashaev_relations_commutation},\eqref{eq:A_commutation},\eqref{eq:3_Kashaev_operators2},\eqref{eq:basic_commutation}}{=} \zeta^3 {\bf P}_{(t_1 t_3 t_2) (t_1 s_3) (t_2 s_2)} 
({\bf A}_{t_3} {\bf A}_{t_2} {\bf P}_{(t_3 t_2)})
 {\bf A}_{t_1} {\bf A}_{s_2} 
{\bf T}_{s_3 s_1}  ( {\bf A}_{s_3} )  e^{-2\pi i Q_{t_3} P_{t_2}}  
 \\
& \quad \cdot 
 e^{-2\pi i Q_{t_2} (Q_{s_3} - P_{s_3})}  {\bf A}_{t_2} e^{2\pi i Q_{t_3} P_{s_3}} 
(\ulu{ e^{2\pi i Q_{t_2} P_{s_1}} }{{\rm to~right}} \Psi_b(Q_{s_1} + P_{t_2} - Q_{t_2} + (P_{s_3} ))^{-1}
e^{-2\pi i Q_{t_2} P_{s_3}} )
 \\
& \hspace{-2.5mm} \stackrel{\eqref{eq:basic_commutation}}{=} \zeta^3 {\bf P}_{(t_1 t_3 t_2) (t_1 s_3) (t_2 s_2)} 
{\bf A}_{t_3} {\bf A}_{t_2} {\bf P}_{(t_3 t_2)} 
 {\bf A}_{t_1} {\bf A}_{s_2} 
{\bf T}_{s_3 s_1}  {\bf A}_{s_3}  e^{-2\pi i Q_{t_3} P_{t_2}}   e^{-2\pi i Q_{t_2} (Q_{s_3} - P_{s_3})} 
 \\
& \quad \cdot 
{\bf A}_{t_2} \ulu{ e^{2\pi i Q_{t_3} P_{s_3}}  }{{\rm to~right}}
(\Psi_b(Q_{s_1} + P_{t_2} - \cancel{ Q_{t_2} } + P_{s_3}  + (\cancel{ Q_{t_2} } - P_{s_1} ) )^{-1} e^{2\pi i P_{s_1} Q_{t_2}}  )
e^{-2\pi i Q_{t_2} P_{s_3}} 
\end{align*}

\begin{align*}
& \hspace{-2.5mm} \stackrel{\eqref{eq:basic_commutation}}{=} \zeta^3 {\bf P}_{(t_1 t_3 t_2) (t_1 s_3) (t_2 s_2)} 
{\bf A}_{t_3} {\bf A}_{t_2} {\bf P}_{(t_3 t_2)}
 {\bf A}_{t_1} {\bf A}_{s_2} 
{\bf T}_{s_3 s_1}  {\bf A}_{s_3}  e^{-2\pi i Q_{t_3} P_{t_2}}   e^{-2\pi i Q_{t_2} (Q_{s_3} - P_{s_3})} 
 \\
& \quad \cdot 
\ulu{ {\bf A}_{t_2}  }{{\rm to~right}}
(\Psi_b(Q_{s_1} + P_{t_2}  + P_{s_3}   - P_{s_1}  )^{-1}  e^{2\pi i Q_{t_3} P_{s_3}}  )
e^{2\pi i Q_{t_2} P_{s_1}}  e^{-2\pi i Q_{t_2}  P_{s_3}} 
 \\
& \hspace{-2mm} \stackrel{\eqref{eq:A_commutation}}{=} \zeta^3 {\bf P}_{(t_1 t_3 t_2) (t_1 s_3) (t_2 s_2)} 
{\bf A}_{t_3} {\bf A}_{t_2} {\bf P}_{(t_3 t_2)} 
 {\bf A}_{t_1} {\bf A}_{s_2} 
{\bf T}_{s_3 s_1}  {\bf A}_{s_3}  e^{-2\pi i Q_{t_3} P_{t_2}}   \ulu{ e^{-2\pi i Q_{t_2} (Q_{s_3} - P_{s_3})}  }{{\rm to~right}}
 \\
& \quad \cdot 
(\Psi_b(Q_{s_1} +(- Q_{t_2})  + P_{s_3}   - P_{s_1}  )^{-1}   {\bf A}_{t_2}    )
e^{2\pi i Q_{t_3} P_{s_3}}  e^{2\pi i Q_{t_2} P_{s_1}}  e^{-2\pi i Q_{t_2} P_{s_3}} 
 \\
& \hspace{-2.5mm} \stackrel{\eqref{eq:basic_commutation}}{=} \zeta^3 {\bf P}_{(t_1 t_3 t_2) (t_1 s_3) (t_2 s_2)} 
{\bf A}_{t_3} {\bf A}_{t_2} {\bf P}_{(t_3 t_2)} 
{\bf A}_{t_1} {\bf A}_{s_2} 
{\bf T}_{s_3 s_1}  {\bf A}_{s_3}  \cdot \ul{1} \cdot \ulu{e^{-2\pi i Q_{t_3} P_{t_2}}   }{{\rm to~right}}
 \\
& \quad \cdot 
(\Psi_b(Q_{s_1}- \cancel{ Q_{t_2} } + P_{s_3}   - P_{s_1} + (\cancel{ Q_{t_2} }  ))^{-1}    e^{-2\pi i Q_{t_2} (Q_{s_3} - P_{s_3})})
{\bf A}_{t_2}  e^{2\pi i Q_{t_3} P_{s_3}}  e^{2\pi i Q_{t_2} P_{s_1}}  e^{-2\pi i Q_{t_2} P_{s_3}} 
 \\
& \hspace{-2.5mm} \stackrel{\eqref{eq:basic_commutation}}{=} \zeta^3 {\bf P}_{(t_1 t_3 t_2) (t_1 s_3) (t_2 s_2)} 
{\bf A}_{t_3} {\bf A}_{t_2} {\bf P}_{(t_3 t_2)} 
{\bf A}_{t_1} {\bf A}_{s_2} 
{\bf T}_{s_3 s_1}  {\bf A}_{s_3}  \cdot (e^{2\pi i Q_{s_3} P_{s_1}} \ulu{ e^{-2\pi i Q_{s_3} P_{s_1}} }{{\rm to~right}} )
 \\
& \quad \cdot 
(\Psi_b(Q_{s_1} + P_{s_3}   - P_{s_1}  )^{-1} e^{-2\pi i Q_{t_3} P_{t_2}}  )
e^{-2\pi i Q_{t_2} (Q_{s_3} - P_{s_3})} {\bf A}_{t_2}  e^{2\pi i Q_{t_3} P_{s_3}}  e^{2\pi i Q_{t_2} P_{s_1}}  e^{-2\pi i Q_{t_2} P_{s_3}} 
 \\
& \hspace{-2.5mm} \stackrel{\eqref{eq:basic_commutation}}{=} \zeta^3 {\bf P}_{(t_1 t_3 t_2) (t_1 s_3) (t_2 s_2)} 
{\bf A}_{t_3} {\bf A}_{t_2} {\bf P}_{(t_3 t_2)} 
{\bf A}_{t_1} {\bf A}_{s_2} 
\ul{ {\bf T}_{s_3 s_1}  {\bf A}_{s_3}  e^{2\pi i Q_{s_3}  P_{s_1}} 
(\Psi_b(Q_{s_1} + P_{s_3}   -\cancel{ P_{s_1} } }
 \\
& \quad \ul{ +(- Q_{s_3}+\cancel{ P_{s_1} } ) )^{-1}  } 
e^{-2\pi i Q_{s_3} P_{s_1}}   )
e^{-2\pi i Q_{t_3} P_{t_2}}   e^{-2\pi i Q_{t_2}  (Q_{s_3} - P_{s_3})} {\bf A}_{t_2}  e^{2\pi i Q_{t_3}  P_{s_3}}  e^{2\pi i Q_{t_2} P_{s_1}}  e^{-2\pi i Q_{t_2} P_{s_3}} 
 \\
& \hspace{-6mm} \stackrel{\eqref{eq:3_Kashaev_operators2}, \eqref{eq:lifted_Kashaev_relations_major}}{=} \zeta^4 {\bf P}_{(t_1 t_3 t_2) (t_1 s_3) (t_2 s_2)} 
{\bf A}_{t_3} {\bf A}_{t_2} \ulu{ {\bf P}_{(t_3 t_2)}  }{{\rm to~left}}
{\bf A}_{t_1} {\bf A}_{s_2} 
({\bf A}_{s_3} {\bf A}_{s_1} \ulu{ {\bf P}_{(s_3 s_1)} }{{\rm to~left}})
 \\
& \quad \cdot 
e^{-2\pi i Q_{s_3} P_{s_1}}   
e^{-2\pi i Q_{t_3} P_{t_2}}   e^{-2\pi i Q_{t_2} (Q_{s_3} - P_{s_3})} {\bf A}_{t_2}  e^{2\pi i Q_{t_3} P_{s_3}}  e^{2\pi i Q_{t_2} P_{s_1}}  e^{-2\pi i Q_{t_2} P_{s_3}} \\
& \hspace{-2.5mm} \stackrel{\eqref{eq:lifted_Kashaev_relations_P}}{=} \zeta^4 ( {\bf P}_{(t_1 t_3 t_2) (t_1 s_3) (t_2 s_2) (t_3 t_2) (s_3 s_1)}  )
{\bf A}_{t_2} {\bf A}_{t_3} 
{\bf A}_{t_1} {\bf A}_{s_2} 
{\bf A}_{s_1} {\bf A}_{s_3} 
 \\
& \quad \cdot
e^{-2\pi i P_{s_1} Q_{s_3}}   
e^{-2\pi i P_{t_2} Q_{t_3} }   e^{-2\pi i (Q_{s_3} - P_{s_3}) Q_{t_2}  } {\bf A}_{t_2}  e^{2\pi i P_{s_3} Q_{t_3}}  e^{2\pi i P_{s_1} Q_{t_2}}  e^{-2\pi i P_{s_3} Q_{t_2}}.
\end{align*}

Meanwhile, observe that
\begin{align*}
\rho_h(\til{A}_t) \rho_h( \til{A}_s) \rho_h( \til{P}_{(ts)}) & \stackrel{\eqref{eq:3_Kashaev_operators2}}{=} {\bf P}_{(t_1 t_3 t_2)} {\bf A}_{t_1} {\bf A}_{t_2} {\bf A}_{t_3} \ulu{ {\bf P}_{(s_1 s_3 s_2)} }{{\rm to~left}} {\bf A}_{s_1} {\bf A}_{s_2} {\bf A}_{s_3} \ulu{ {\bf P}_{(t_1 s_1)(t_2 s_2)(t_3 s_3)} }{{\rm to~left}} \\
& \stackrel{\eqref{eq:lifted_Kashaev_relations_P}}{=} {\bf P}_{(t_1 t_3 t_2) (s_1 s_3 s_2) (t_1 s_1)(t_2 s_2)(t_3 s_3)} {\bf A}_{t_1} {\bf A}_{t_2} {\bf A}_{t_3} {\bf A}_{s_1} {\bf A}_{s_2} {\bf A}_{s_3},
\end{align*}
as well as
\begin{align*}
& {\bf P}_{(t_1 t_3 t_2) (s_1 s_3 s_2) (t_1 s_1)(t_2 s_2)(t_3 s_3)}^{-1} {\bf P}_{(t_1 t_3 t_2) (t_1 s_3) (t_2 s_2) (t_3 t_2) (s_3 s_1)} \\
& \stackrel{\eqref{eq:lifted_Kashaev_relations_P}}{=} {\bf P}_{(s_3 t_3)(s_2 t_2)(s_1 t_1)(s_2 s_3 s_1)\cancel{ (t_2 t_3 t_1)} \cancel{ (t_1 t_3 t_2)  }(t_1 s_3) (t_2 s_2) (t_3 t_2) (s_3 s_1)}
\stackrel{\eqref{eq:lifted_Kashaev_relations_P}}{=} {\bf P}_{(t_2 s_3)}.
\end{align*}
Therefore $(\rho_h(\til{A}_t) \rho_h(\til{A}_s) \rho_h(\til{P}_{(ts)}))^{-1} \rho_h(\til{T}_{ts}) \rho_h(\til{A}_t) \rho_h(\til{T}_{st})$ equals
\begin{align*}
{\bf K}_{ts} := \zeta^4 {\bf P}_{(t_2 s_3)} e^{-2\pi i P_{s_1} Q_{s_3}}   
e^{-2\pi i P_{t_2} Q_{t_3} }   e^{-2\pi i (Q_{s_3} - P_{s_3}) Q_{t_2}  } {\bf A}_{t_2}  e^{2\pi i P_{s_3} Q_{t_3}}  e^{2\pi i P_{s_1} Q_{t_2}}  e^{-2\pi i P_{s_3} Q_{t_2}}.
\end{align*}
We should now show that conjugation by ${\bf K}_{ts}$ is trivial.

\vs

Indeed, one can easily check by using \eqref{eq:basic_commutation} that
\begin{align*}
& {\bf K}_{ts}^{-1} P_{s_1} {\bf K}_{ts} =  P_{s_1}, \quad
{\bf K}_{ts}^{-1} P_{s_2} {\bf K}_{ts} = P_{s_2}, \quad
{\bf K}_{ts}^{-1} Q_{s_2} {\bf K}_{ts} = Q_{s_2}, \quad
{\bf K}_{ts}^{-1} P_{t_1} {\bf K}_{ts} = P_{t_1}, \\
& {\bf K}_{ts}^{-1} Q_{t_1} {\bf K}_{ts} = Q_{t_1}, \quad {\bf K}_{ts}^{-1} P_{t_2} {\bf K}_{ts} = \cancel{ P_{s_3} } - \cancel{ P_{s_1} } + P_{t_2} + \cancel{ P_{s_1} } - \cancel{ P_{s_3} } = P_{t_2}, \quad
{\bf K}_{ts}^{-1} Q_{t_3} {\bf K}_{ts} = Q_{t_3},
\end{align*}
and
\begin{align*}
{\bf K}_{ts}^{-1} Q_{s_1} {\bf K}_{ts} & = Q_{s_1} + Q_{s_3} + P_{t_2} - Q_{t_3} - \cancel{ Q_{t_2} } + P_{s_1} + \cancel{ Q_{t_2} } - P_{s_3} \stackrel{\eqref{eq:the_crucial_linear_relation}}{=} Q_{s_1}, \\
{\bf K}_{ts}^{-1} P_{s_3} {\bf K}_{ts} & = \cancel{ Q_{t_2} } - P_{t_2} - Q_{s_3} + P_{s_3} + Q_{t_3} - P_{s_1} - \cancel{ Q_{t_2} } + P_{s_3} \stackrel{\eqref{eq:the_crucial_linear_relation}}{=} P_{s_3}, \\
{\bf K}_{ts}^{-1} Q_{s_3} {\bf K}_{ts} & = - P_{t_2} + Q_{t_3} - P_{s_1}  + P_{s_3} \stackrel{\eqref{eq:the_crucial_linear_relation}}{=}  Q_{s_3}, \\
{\bf K}_{ts}^{-1} Q_{t_2} {\bf K}_{ts} & = Q_{s_3} + P_{t_2} - Q_{t_3} + P_{s_1} + Q_{t_2} - P_{s_3} \stackrel{\eqref{eq:the_crucial_linear_relation}}{=} Q_{t_2}, \\
{\bf K}_{ts}^{-1} P_{t_3} {\bf K}_{ts} & =  P_{t_3} - \cancel{ Q_{t_2} } + P_{t_2} + Q_{s_3} - \cancel{ P_{s_3} }  - Q_{t_3} + \cancel{ P_{s_3} } + P_{s_1}  + \cancel{ Q_{t_2} } - P_{s_3} \stackrel{\eqref{eq:the_crucial_linear_relation}}{=} P_{t_3},
\end{align*}
which depend on \eqref{eq:the_crucial_linear_relation}.
\end{proof}

\end{document}